\documentclass[11pt,reqno]{amsart}
\usepackage{amssymb,amsmath,amsfonts,amsthm,enumerate,stmaryrd,tensor,mathtools,dsfont,upgreek,bbm,mathrsfs,nameref,microtype,xcolor,bm,tikz}
\usetikzlibrary{arrows}
\usepackage[left=0.75 in, right=0.75 in,top=0.75 in, bottom=0.75 in]{geometry}

\usepackage{hyperref,cleveref}

\numberwithin{equation}{section}
\definecolor{popblue}{RGB}{55,115,255}
\definecolor{lightbl}{RGB}{155,205,255}
\definecolor{depthbl}{RGB}{145,215,255}
\definecolor{fancyre}{RGB}{225,55,115}
\definecolor{darkblu}{RGB}{15,75,185}
\definecolor{mellowy}{RGB}{225,225,35}
\renewcommand{\tilde}[1]{\widetilde{#1}}
\renewcommand{\Bar}{\overline}

\newcommand{\R}{\mathbb{R}}
\newcommand{\N}{\mathbb{N}}
\newcommand{\Z}{\mathbb{Z}}

\newcommand{\C}{\mathbb{C}}
\newcommand{\X}{\mathbb{X}}
\newcommand{\Y}{\mathbb{Y}}

\newcommand{\F}{\mathbb{F}}
\newcommand{\m}{\mathrm}

\newcommand{\lv}{\lVert}
\newcommand{\rv}{\rVert}

\newcommand{\al}{\alpha}
\newcommand{\be}{\beta}
\newcommand{\es}{\varnothing}

\newcommand{\ep}{\varepsilon}
\newcommand{\f}{\frac}
\newcommand{\sig}{\sigma}
\newcommand{\gam}{\gamma}
\newcommand{\del}{\delta}

\newcommand{\pd}{\partial}

\newcommand{\grad}{\nabla}
\newcommand{\bpm}{\begin{pmatrix}}
\newcommand{\epm}{\end{pmatrix}}

\newcommand{\loc}{\m{loc}}
\renewcommand{\bar}{\overline}

\newcommand{\emb}{\hookrightarrow}

\renewcommand{\le}{\leqslant}
\renewcommand{\ge}{\geqslant}
\newcommand{\jump}[1]{\left\llbracket#1\right\rrbracket}
\newcommand{\tjump}[1]{\llbracket#1\rrbracket}

\newcommand{\norm}[1]{\left\lv#1\right\rv}
\newcommand{\bnorm}[1]{\Big\lv#1\Big\rv}

\newcommand{\tnorm}[1]{\lv#1\rv}
\newcommand{\p}[1]{\left(#1\right)}
\newcommand{\bp}[1]{\Big(#1\Big)}
\renewcommand{\sp}[1]{\big(#1\big)}
\newcommand{\tp}[1]{(#1)}
\newcommand{\tfloor}[1]{\lfloor #1\rfloor}

\newcommand{\abs}[1]{\left|#1\right|}
\newcommand{\babs}[1]{\Big|#1\Big|}

\newcommand{\tabs}[1]{|#1|}
\newcommand{\bsb}[1]{\Big[{#1}\Big]}
\newcommand{\ssb}[1]{\big[{#1}\big]}
\newcommand{\tsb}[1]{[{#1}]}
\newcommand{\bcb}[1]{\Big\{{#1}\Big\}}
\newcommand{\tcb}[1]{\{{#1}\}}
\newcommand{\br}[1]{\left\langle #1 \right\rangle}
\providecommand{\bbr}[1]{\Big\langle #1 \Big\rangle}
\providecommand{\sbr}[1]{\big\langle #1 \big\rangle}
\providecommand{\tbr}[1]{\langle #1 \rangle}
\renewcommand{\bf}[1]{\mathbf{#1}}
\newcommand{\ii}{\m{i}}
\DeclareMathOperator{\supp}{supp}

\DeclareMathOperator*{\esssup}{ess\,sup}
\newtheorem{prop}{\color{popblue}{Proposition}}[section]
\newtheorem{thm}[prop]{\color{popblue}{Theorem}}
\newtheorem{defn}[prop]{\color{popblue}{Definition}}
\newtheorem{lem}[prop]{\color{popblue}{Lemma}}
\newtheorem{coro}[prop]{\color{popblue}{Corollary}}
\newtheorem{rmk}[prop]{\color{popblue}{Remark}}

\newtheorem{innercustomthm}{\color{popblue}{Theorem}}
\newenvironment{customthm}[1]
{\renewcommand\theinnercustomthm{#1}\innercustomthm}
{\endinnercustomthm}
\newtheorem{innercustomcoro}{\color{popblue}{Corollary}}
\newenvironment{customcoro}[1]
{\renewcommand\theinnercustomcoro{#1}\innercustomcoro}
{\endinnercustomcoro}
\author{Noah Stevenson}
\address{
	Department of Mathematics\\
	Princeton University\\
	Princeton, NJ 08544, USA
}
\email[N. Stevenson]{stevenson@princeton.edu}
\thanks{N. Stevenson was supported by an NSF Graduate Research Fellowship}
\author{Ian Tice}
\address{
	Department of Mathematical Sciences\\
	Carnegie Mellon University\\
	Pittsburgh, PA 15213, USA
}
\email[I. Tice]{iantice@andrew.cmu.edu}
\thanks{I. Tice was supported by an NSF Grant (DMS \#2204912). }

\title[Stationary and slowly traveling waves]{
    Well-posedness of the stationary and slowly traveling wave problems for the free boundary incompressible Navier-Stokes equations
 }

\begin{document}
% _+__+_ -_+__+_ -_+__+_ -_+__+_ -_+__+_ -_+__+_ -_+__+_ -_+__+_ -_+__+_ -_+__+_ -_+__+_ -_+__+_ -_+__+_ -
\maketitle
%\tableofcontents
% _+__+_ -_+__+_ -_+__+_ -_+__+_ -_+__+_ -_+__+_ -_+__+_ -_+__+_ -_+__+_ -_+__+_ -_+__+_ -_+__+_ -_+__+_ -

\begin{abstract}
We establish that solitary stationary waves in three dimensional viscous incompressible fluids are a generic phenomenon and that every such solution is a vanishing wave-speed limit along a one parameter family of traveling waves. The setting of our result is a horizontally-infinite fluid of finite depth with a flat, rigid bottom and a free boundary top. A constant gravitational field acts normal to bottom, and the free boundary experiences surface tension. In addition to these gravity-capillary effects, we allow for applied stress tensors to act on the free surface region and applied forces to act in the bulk.  These are posited to be in either stationary or traveling form.

In the absence of any applied stress or force, the system reverts to a quiescent equilibrium; in contrast, when such sources of stress or force are present, stationary or traveling waves are generated.  We develop a small data well-posedness theory for this problem by proving that there exists a neighborhood of the origin in stress, force, and wave speed data-space in which we obtain the existence and uniqueness of stationary and traveling wave solutions that depend continuously on the stress-force data, wave speed, and other physical parameters. To the best of our knowledge, this is the first proof of well-posedness of the solitary stationary wave problem and the first continuous embedding of the stationary wave problem into the traveling wave problem. Our techniques are based on vector-valued harmonic analysis, a novel method of indirect symbol calculus, and the implicit function theorem.
\end{abstract}
% _+__+_ -_+__+_ -_+__+_ -_+__+_ -_+__+_ -_+__+_ -_+__+_ -_+__+_ -_+__+_ -_+__+_ -_+__+_ -_+__+_ -_+__+_ -
% \maketitle
% _+__+_ -_+__+_ -_+__+_ -_+__+_ -_+__+_ -_+__+_ -_+__+_ -_+__+_ -_+__+_ -_+__+_ -_+__+_ -_+__+_ -_+__+_ -
% \pagebreak
% _+__+_ -_+__+_ -_+__+_ -_+__+_ -_+__+_ -_+__+_ -_+__+_ -_+__+_ -_+__+_ -_+__+_ -_+__+_ -_+__+_ -_+__+_ -
\section{Introduction}
% _+__+_ -_+__+_ -_+__+_ -_+__+_ -_+__+_ -_+__+_ -_+__+_ -_+__+_ -_+__+_ -_+__+_ -_+__+_ -_+__+_ -_+__+_ -

% _+__+_ -_+__+_ -_+__+_ -_+__+_ -_+__+_ -_+__+_ -_+__+_ -_+__+_ -_+__+_ -_+__+_ -_+__+_ -_+__+_ -_+__+_ -
\subsection{The free boundary Navier-Stokes system}
% _+__+_ -_+__+_ -_+__+_ -_+__+_ -_+__+_ -_+__+_ -_+__+_ -_+__+_ -_+__+_ -_+__+_ -_+__+_ -_+__+_ -_+__+_ -

Our goal in this paper is to study stationary and slowly traveling solutions to the free boundary incompressible Navier-Stokes equations in three dimensions.  These equations govern the dynamics of a finite-depth layer of viscous, incompressible fluid lying between a fixed, rigid, flat bottom and an unknown (free) top that evolves with the fluid.  In order to properly phrase the equations, we first establish some notation for describing the unknown fluid domain. 

The fluids we study will always be assumed to occupy  three-dimensional sets of the form
\begin{equation}
    \Omega[\eta]=\tcb{(x,y)\in\R^{2}\times\R\;:\;0<y<b+\eta(x)},
\end{equation}
where $b\in\R^+$ is a fixed parameter giving the equilibrium depth of the fluid, and $\eta:\R^2\to(-b,\infty)$ is the unknown free surface function.  We will always have that $\eta$ is continuous so that the fluid domain $\Omega[\eta]$ is open and connected.  The upper free boundary and the fixed lower boundary will be denoted by 
\begin{equation}
    \Sigma[\eta]=\tcb{(x,y)\in\R^2\times\R\;:\;y=b+\eta(x)}
    \text{ and }
    \Sigma_0=\R^2\times\tcb{0}.
\end{equation}
Throughout the paper we will also denote the equilibrium sets with the short-hand
\begin{equation}
    \Omega=\Omega[0]=\R^2\times(0,b)\text{ and }\Sigma=\Sigma[0]=\R^2\times\tcb{b}.
\end{equation}

The motion of the fluid domain is encoded through the use of a time-dependent free surface function $\zeta\p{t,\cdot}:\R^{2}\to\R$ satisfying $\zeta\p{t,\cdot}+b>0$, which then generates the moving fluid domain $\Omega[\zeta\p{t,\cdot}]\subset \R^3$ and the free upper boundary $\Sigma[\zeta\p{t,\cdot}]$ as above.  The fluid is described by its velocity vector field $w\p{t,\cdot}:\Omega[\zeta\p{t,\cdot}]\to\R^n$ and its scalar pressure $r\p{t,\cdot}:\Omega[\zeta\p{t,\cdot}]\to\R$. The viscous stress tensor within the fluid is the symmetric tensor
\begin{equation}
    S_\mu(r,w)=r{I_{3\times 3}}-\mu \mathbb{D}w,
\end{equation}
where $\mathbb{D}w=\grad w+\grad w^{\m{t}}$ is the symmetrized gradient, and $\mu \in \R^+$ is the fluid viscosity.

In this paper we will assume that there are two sources of bulk force that act on the fluid through vector fields defined on $\Omega[\zeta\p{t,\cdot}]$ for all $t$, as well as three sources of stress that act on the fluid through vector fields defined on $\Sigma[\zeta\p{t,\cdot}]$.  The first bulk force is a uniform gravitational field $-\rho \mathfrak{g} e_3$, where $\rho \in \R^+$ is the constant fluid density, $\mathfrak{g} \in \R^+$ is the gravitational constant, and $e_3 = (0,0,1) \in \R^3$.  The second is a spatially-varying generic bulk force $\mathcal{F}\p{t,\cdot}:\Omega[\zeta\p{t,\cdot}]\to\R^3$.  The first of the stresses is due to a constant external pressure $P_{\m{ext}} \in \R$, which then acts via the vector field $P_{\m{ext}} \nu_{\zeta\p{t,\cdot}}$, where $\nu_{\zeta\p{t,\cdot}}$ is the unit normal to the free surface at time $t$.  The second is generated by a generic spatially-varying stress tensor $\mathcal{T}\p{t,\cdot}:\Sigma[\zeta\p{t,\cdot}]\to\R^{3\times 3}$, which then defines the stress vector $\mathcal{T}\p{t,\cdot} \nu_{\zeta\p{t,\cdot}}$.  We note that in continuum mechanics it is usually the case that $\mathcal{T}\p{t,\cdot}$ is symmetric, but this condition plays no role in our analysis, so we have allowed for the most general case. The third, and final, source of stress is due to surface tension and is given by the vector field $\kappa \mathscr{H}(\zeta\p{t,\cdot}))\nu_{\zeta\p{t,\cdot}}$, where $\kappa \in \R^+$ is the coefficient of surface tension and the mean curvature operator is 
\begin{equation}\label{mean_curvature_def} 
    \mathscr{H}(\zeta) = \grad_{\|} \cdot ( (1+ \abs{\grad_{\|} \zeta}^2)^{-1/2}    \grad_{\|}   \zeta).
\end{equation}
Here we have written $\grad_{\|}=(\pd_1,\pd_2)$ to refer to the `tangential gradient.'

The free boundary incompressible Navier-Stokes equations then dictate how $\zeta$, $w$, and $r$ evolve in time as the result of applied stresses and forces:
\begin{equation}\label{formulation in Eulerian coordinates}
    \begin{cases}
    \rho(\pd_tw+w\cdot\grad w) + \grad\cdot S_\mu(r,w)=-\rho \mathfrak{g} e_3+\mathcal{F}&\text{in }\Omega[\zeta(t,\cdot)],\\
    \grad\cdot w=0&\text{in }\Omega[\zeta(t,\cdot)]\\
    -S_\mu(r,w)\nu_\zeta+(P_{\m{ext}}-\kappa \mathscr{H}(\zeta))\nu_\zeta=\mathcal{T}\nu_\zeta&\text{on }\Sigma[\zeta(t,\cdot)],\\
    \pd_t\zeta+w\cdot(\grad_{\|}\zeta,-1)=0&\text{on }\Sigma[\zeta(t,\cdot)],\\
    w=0&\text{on }\Sigma_0.
    \end{cases}
\end{equation}
The first equation in~\eqref{formulation in Eulerian coordinates} is the momentum equation, and it requires a Newtonian balance of forces in the fluid bulk. Next is the incompressibility constraint, which asserts conservation of mass. After this is the dynamic boundary condition, which enforces a balance of stresses acting on the free surface.  The penultimate equation is the kinematic boundary condition, which determines how the free surface evolves according to the fluid velocity.  The final equation in~\eqref{formulation in Eulerian coordinates} is simply the no-slip boundary condition for the velocity on the rigid bottom.  For the sake of simplicity, we will henceforth assume that $\rho=1$.  This is no loss of generality, as we will continue to track the generic constants $(\mathfrak{g},\mu,\kappa) \in (\R^+)^3$ as well as generic sources of force and stress.

% _+__+_ -_+__+_ -_+__+_ -_+__+_ -_+__+_ -_+__+_ -_+__+_ -_+__+_ -_+__+_ -_+__+_ -_+__+_ -_+__+_ -_+__+_ -
\subsection{Equilibria, stationary and traveling reformulations, and the role of stresses and forces}
% _+__+_ -_+__+_ -_+__+_ -_+__+_ -_+__+_ -_+__+_ -_+__+_ -_+__+_ -_+__+_ -_+__+_ -_+__+_ -_+__+_ -_+__+_ -

The free boundary incompressible Navier-Stokes equations admit a flat, stationary solution in the absence of external stress or forces, i.e. $\mathcal{F}=0$ and $\mathcal{T}=0$; namely, fluid domain $\Omega = \Omega[0]$ and
\begin{equation}
    (r_{\m{eq}},w_{\m{eq}},\zeta_{\m{eq}})=(P_{\m{ext}}+\mathfrak{g}(b-\m{id}_{\R^3}\cdot e_3),0,0).
\end{equation}
In this paper, we work perturbatively around this equilibrium to study stationary and slowly traveling solutions.  To describe these, we let $\gam\in\R$ denote a fixed speed and make the ansatz that $\mathcal{F}$ and $\mathcal{T}$ are time-independent in the frame moving at velocity $\gam e_1$.  In turn, we assume that $\zeta(t,\cdot)=\eta(\cdot-\gam te_1)$, $w(t,\cdot)=v(\cdot-\gam te_1)$, $r(t,\cdot)= P_{\m{ext}}+\mathfrak{g}(b-\m{id}_{\R^3}\cdot e_3)  + q(\cdot-\gam te_1)+ \mathfrak{g} \eta(\cdot-\gam te_1)$, for new unknowns $\eta:\R^2\to(-b,\infty)$, $v:\Omega[\eta]\to\R^3$, and $q:\Omega[\eta]\to\R$. Rewriting the system~\eqref{formulation in Eulerian coordinates} under this ansatz yields:
\begin{equation}\label{the stationarity ansatz}
    \begin{cases}
        (v-\gam e_1)\cdot\grad v+\grad\cdot S_\mu(q+\mathfrak{g} \eta,v)=\mathcal{F}&\text{in }\Omega[\eta],\\
        \grad\cdot v=0&\text{in }\Omega[\eta],\\
        -S_\mu(q,v)\mathcal{N}_\eta- \kappa \mathscr{H}(\eta)\mathcal{N}_\eta=\mathcal{T}\mathcal{N}_\eta&\text{on }\Sigma[\eta],\\
        \gam\pd_1\eta+v\cdot\mathcal{N}_\eta=0&\text{on }\Sigma[\eta],\\
        v=0&\text{on }\Sigma_0,
    \end{cases}
\end{equation}
where $\mathcal{N}_\eta=(-\grad_{\|}\eta,1)$.  We emphasize three key features of this reformulation.  First, the external pressure, $P_{\m{ext}}$, only appears in the hydrostatic background that has been subtracted off and will play no further role in the analysis of \eqref{the stationarity ansatz}.  Second, the constant gravitational force field has been shifted into the term $\mathfrak{g} \eta$ in the first equation.  Third, while the set $\Omega[\eta]$ is determined by $\eta$, only derivatives of $\eta$ appear in the equations themselves.

The problem~\eqref{the stationarity ansatz} lies at the confluence of two distinct lines of inquiry in the mathematical fluid mechanics literature. The first line of inquiry treats the dynamic problem~\eqref{formulation in Eulerian coordinates} as an initial value problem.  In this context, the stationary problem ($\gamma =0$ in~\eqref{the stationarity ansatz}) arises naturally as a special type of global-in-time solution with stationary sources of force and stress.  One then expects solutions to the stationary problem to play an essential role in the study of long-time asymptotics or attractors for the dynamic problem (see, for instance, Robinson \cite{robinson_2001}). The second line of inquiry, which dates back essentially to the beginning of mathematical fluid mechanics, concerns the search for traveling wave solutions moving with speed $\gamma \neq 0$. In this context, a huge literature exists for the corresponding inviscid problem, but progress on the viscous problem was initiated much more recently in the work of Leoni and Tice~\cite{leoni2019traveling}, and further developed by Stevenson and Tice~\cite{MR4337506,stevenson2023wellposedness}, Koganemaru and Tice~\cite{koganemaru2022traveling}, and Nguyen and Tice~\cite{nguyen_tice_2022}. The analysis in~\cite{koganemaru2022traveling,leoni2019traveling,nguyen_tice_2022,MR4337506,stevenson2023wellposedness} crucially relies on the condition $\gamma \neq 0$ to provide an estimate for the free surface function in a scale of anisotropic Sobolev spaces.  When $\gamma =0$, this estimate degenerates, and~\cite{leoni2019traveling} fails to construct solutions within their functional framework.  Thus, a natural question is whether there exists an alternate functional framework in which solutions can be constructed for all $\gamma$ in a neighborhood of $0$.

Our main goal in the paper can now be roughly summarized as follows.  For every $(\mathfrak{g},\mu,\kappa) \in (\R^+)^3$ and $\gamma$ in an open set containing $0$ we wish to identify an open set of force and stress data that give rise to locally unique nontrivial solutions.  Moreover, we aim to prove well-posedness in the sense of continuity of the solution triple with respect to the force-stress data as well as the various physical parameters and wave speed. 

The stated goal suggests that the force and stress should play an essential role in the construction of solutions.  This is indeed the case, as we now aim to justify.  An elementary formal calculation yields the following balance between dissipation and power for solutions to~\eqref{the stationarity ansatz}:
\begin{equation}\label{she said I know what its like to be dead}
    \int_{\Omega[\eta]}\frac{\mu}{2} |\mathbb{D}v|^{2}=\int_{\Omega[\eta]}\mathcal{F}\cdot v+\int_{\Sigma[\eta]}\mathcal{T}\nu_\eta\cdot v.
\end{equation}
The physical interpretation of this identity is that if a stationary or traveling wave solution exists, then the power supplied by the forces and stress (the right side of~\eqref{she said I know what its like to be dead}) must be in exact balance with the energy dissipation rate due to viscosity (the left side of~\eqref{she said I know what its like to be dead}).  Identity~\eqref{she said I know what its like to be dead} tells us even more if we take $\mathcal{F}=0$ and $\mathcal{T}=0$, in which case the $L^2$-norm of $\mathbb{D} v$ vanishes in $\Omega[\eta]$.  By a version of the Korn inequality, this implies that $v=0$, and in turn, this implies that $q=0$ and $\eta$ is constant.  Thus,  we only expect to be able to generate non-trivial stationary or traveling wave solutions (in a Sobolev-type framework in which~\eqref{she said I know what its like to be dead} is valid) via the application of  nontrivial $\mathcal{F}$ or $\mathcal{T}$.

% _+__+_ -_+__+_ -_+__+_ -_+__+_ -_+__+_ -_+__+_ -_+__+_ -_+__+_ -_+__+_ -_+__+_ -_+__+_ -_+__+_ -_+__+_ -
\subsection{Previous work}
% _+__+_ -_+__+_ -_+__+_ -_+__+_ -_+__+_ -_+__+_ -_+__+_ -_+__+_ -_+__+_ -_+__+_ -_+__+_ -_+__+_ -_+__+_ -

We now turn our attention to a brief survey of the mathematical literature associated to~\eqref{formulation in Eulerian coordinates} and~\eqref{the stationarity ansatz}.  This is vast, so we will restrict our focus to those results most closely related to ours.

For a thorough review of the fully dynamic problem~\eqref{formulation in Eulerian coordinates} in various geometries we refer to the surveys of Zadrzy\'{n}ska \cite{Zadrynska_2004} and Shibata and Shimizu \cite{SS_2007}.  Beale \cite{Beale_1981} established local well-posedness with surface tension neglected.  With surface tension accounted for, Beale \cite{Beale_1983} established the existence of global solutions and derived their decay properties with Nishida \cite{BN_1985}.  Solutions with surface tension were also constructed in other settings by Allain \cite{allain_1987}, Tani \cite{tani_1996}, Bae \cite{Bae_2011}, and Shibata and Shimizu \cite{SS_2011}.  Solutions without surface tension were also constructed in various settings by Abels \cite{Abels_2005_3},  Guo and Tice \cite{GT_2013_inf,GT_2013_lwp},  and Wu \cite{Wu_2014}.  Related analysis of linearized and resolvent problems can be found in the work of  Abe and Shibata \cite{AS_2003,AS_2003_2}, Abels \cite{Abels_2005,Abels_2005_2,Abels_2006},  Abels and Wiegner \cite{AW_2005}, and Abe and Yamazaki \cite{AY_2010}.

The inviscid analog of the traveling wave problem, \eqref{the stationarity ansatz} with $\gamma \neq 0$, which is also known as the traveling water wave problem, has been the subject of intense work for more than a century.  The survey articles of  Toland~\cite{Toland_1996}, Groves~\cite{Groves_2004}, Strauss~\cite{Strauss_2010}, and Haziot, Hur, Strauss, Toland, Wahl\'en, Walsh, and Wheeler~\cite{MR4406719} contain a thorough review.  A significant portion of this work concerns spatially periodic solutions.  The only result for the non-periodic (solitary) stationary ($\gamma=0$) inviscid problem we are aware of is the recent construction by Ehrnstr\"{o}m, Walsh, and Zeng \cite{EWZ_2023} of stationary gravity-capillary water waves with localized vorticity.

In contrast, progress on the viscous traveling wave problem has only recently commenced. Leoni and Tice~\cite{leoni2019traveling} developed a well-posedness theory for small forcing and stress data, provided $\gamma \neq 0$.  This was generalized to multi-layer, inclined, and periodic configurations by Stevenson and Tice~\cite{MR4337506} and Koganemaru and Tice~\cite{koganemaru2022traveling}.  The corresponding well-posedness theory for the compressible analog of \eqref{the stationarity ansatz} was developed by Stevenson and Tice \cite{stevenson2023wellposedness}.  Traveling waves for the Muskat problem were constructed with similar techniques by Nguyen and Tice \cite{nguyen_tice_2022}.  There are also experimental studies of viscous traveling waves; for details, we refer to the work of Akylas, Cho, Diorio, and Duncan \cite{CDAD_2011,DCDA_2011}, Masnadi and Duncan \cite{MD_2017}, and Park and Cho \cite{PC_2016,PC_2018}.

We now turn our attention to the viscous, stationary $(\gam=0)$ literature. To the best of our knowledge, the precise configuration we study in~\eqref{the stationarity ansatz} - including bulk force, a surface stress, and a non-compact free surface -  has not yet appeared in the literature for either the three-dimensional or two-dimensional problem.  However, numerous models of similar physical scenarios have been considered.   

The non-compactness of the free boundary presents a fundamental difficulty in studying~\eqref{the stationarity ansatz}, as it creates a low-mode degeneracy that simply is not present in, say, the spatially periodic variant or related problems with compact free boundaries.  As such, we only briefly review the stationary literature for compact free boundaries.   Benjamin~\cite{MR134999} studied periodic disturbances to steady flow along an inclined plane in two dimensions.  Periodic solutions in two dimensions were also studied by Puhna\v{c}ev~\cite{MR0308618}.   Solonnikov~\cite{MR525947,MR737234}, Jean~\cite{MR591221}, and Ja Jin~\cite{MR2170526} studied various compact free surface problems with sources and sinks or inflow and outflow conditions in a bounded container with an applied force.  In three dimensions, Bemelmans~\cite{MR727877,MR637857,MR929474} studied various stationary droplet problems, considering both the cases with and without surface tension. Abergel~\cite{MR1215410} gave a geometric approach for studying various configurations in both two and three dimensions, which was expanded on in Abergel and Rouy~\cite{abergel1995interfaces}. We refer also to Solonnikov and Denisova~\cite{MR3916796} for more references regarding the bounded free surface stationary literature.

Next, we discuss the literature involving unbounded domains and non-compact free surfaces.  Much of the attention of the existing work is devoted to steady flows driven by gravity down inclined planes with possibly non-uniform structure.  In two dimensions, this configuration was considered by Socolescu~\cite{MR0564661}, Nazarov and Pileckas~\cite{MR1215650}, Pelickas and Socolowsky~\cite{MR1936348,MR2128433}, Socolowsky~\cite{MR3032463}, and Pileckas and Solonnikov~\cite{MR3208795}. In three dimensions  Gellrich~\cite{MR1245932} studied stationary flows with large viscosity and small localized bulk force.

The remaining non-compact literature in two dimensions is primarily devoted to more complicated geometries. For instance, Pileckas
~\cite{MR1083808,MR741910,MR789999,MR643983} considered boundary inflows, moving lower boundaries, gliding plates, and flow down a plane making a corner, while Socolowsky~\cite{MR1246506} described fluid flowing out of a pipe, driven by gravity. In three dimensions the situation is similar: Pileckas
~\cite{MR966116}  studied liquid coming out of a narrow channel onto an incline plane, and Solonnikov~\cite{MR1638135,MR1692348} studied the flow generated by the slow rotation of an immersed rod and flow out of a circular tube.

To conclude, we again emphasize that, to the best of our knowledge, there are no results in the literature that study either: (1) the well-posedness in all parameter regimes of the three dimensional, non-compact stationary wave problem with applied bulk force and surface stress, i.e. system~\eqref{the stationarity ansatz} with $\gam=0$; or (2)  the continuous connection between the recent developments in the viscous free boundary traveling wave literature and the stationary wave problem.  We address both of these in this paper.

% _+__+_ -_+__+_ -_+__+_ -_+__+_ -_+__+_ -_+__+_ -_+__+_ -_+__+_ -_+__+_ -_+__+_ -_+__+_ -_+__+_ -_+__+_ -
\subsection{Flattened reformulation}
% _+__+_ -_+__+_ -_+__+_ -_+__+_ -_+__+_ -_+__+_ -_+__+_ -_+__+_ -_+__+_ -_+__+_ -_+__+_ -_+__+_ -_+__+_ -

It will be convenient to reformulate the system~\eqref{the stationarity ansatz} in the stationary domain $\Omega=\R^2\times(0,b)$.  To this end, we construct a flattening map from $\eta$ by way of $\mathfrak{F}_\eta:\Omega\to\Omega[\eta]$ defined via
\begin{equation}\label{definition of the flattening map}
     \mathfrak{F}_\eta(x,y)=(x,y+\mathcal{E}\eta(x,y)), 
\end{equation}
where $\mathcal{E}$ is the extension operator considered in Proposition~\ref{prop on extension operator}. Note that in Proposition~\ref{prop on properties of the flattening map} we show that the above flattening map is well-defined and enjoys a collection of useful properties on the class of free surface functions considered in this paper.

Given $\eta$ in an appropriate function space (which will be specified later), and hence  $\mathfrak{F}_\eta$, we define two related quantities: the Jacobian $J_\eta:\Omega\to\R^+$  and (when $J_\eta$ is nowhere vanishing) the geometry matrix $\mathcal{A}_\eta:\Omega\to\R^{3\times 3}$, defined  respectively via
 \begin{equation}\label{geometry_and_jacobian_def}
 J_\eta=\det(\grad\mathfrak{F}_\eta)=1+\pd_3\mathcal{E}\eta = \pd_3(\mathfrak{F}_\eta\cdot e_3) 
 \text{ and }
 \mathcal{A}_\eta=(\grad\mathfrak{F}_\eta)^{-\m{t}}.
 \end{equation}
 Provided that $J_\eta>0$  and $J_\eta,1/J_\eta\in L^\infty(\Omega)$, we then have that  $\mathfrak{F}_\eta(\Omega)=\Omega[\eta]$ and $\mathfrak{F}_\eta$ is a  homeomorphism  from $\Bar{\Omega}$ to $\Bar{\Omega[\eta]}$ such that its restriction to $\Omega$ defines a smooth diffeomorphism to $\Omega[\eta]$,  $\mathfrak{F}_\eta(\Sigma) = \Sigma[\eta]$, and $\mathfrak{F}_\eta$ is the identity on  $\Sigma_0$.   It will also be useful to introduce the map
 \begin{equation}\label{Mississippi}		M_{\eta}=J_\eta\mathcal{A}_{\eta}^{\m{t}}=\bpm(1+\pd_3\mathcal{E}\eta) I_{2\times2}&0_{2\times 1}\\-\mathcal{E}(\grad_{\|}\eta)&1\epm:\Omega \to \R^{3 \times 3}
\end{equation}
when reformulating~\eqref{the stationarity ansatz}.

We then introduce the new unknowns $p=q\circ\mathfrak{F}_\eta:\Omega\to\R$ and $u=M_\eta(v\circ\mathfrak{F}_\eta):\Omega\to\R^3$.   The problem~\eqref{the stationarity ansatz} then transforms to the following system:
\begin{equation}\label{final nonlinear equations}
\begin{cases}
    M_\eta^{-\m{t}}((u-\gam M_\eta e_1)\cdot\grad(M_\eta^{-1}u))+\grad(p+\mathfrak{g} \eta) - \mu M_\eta^{-\m{t}}(\grad\cdot((\mathbb{D}_{\mathcal{A}_\eta}(M_\eta^{-1}u))M_\eta^{\m{t}}))=J_\eta M_\eta^{-\m{t}}\mathcal{F}\circ\mathfrak{F}_\eta&\text{in }\Omega,\\
    \grad\cdot u=0&\text{in }\Omega,\\
    -(pI-\mu\mathbb{D}_{\mathcal{A}_\eta}(M_\eta^{-1}u))M_\eta^{\m{t}}e_3 - \kappa \mathscr{H}(\eta)M_\eta^{\m{t}}e_3=\mathcal{T}\circ\mathfrak{F}_\eta M_\eta^{\m{t}}e_3&\text{on }\Sigma,\\
    u\cdot e_3+\gam\pd_1\eta=0&\text{on }\Sigma,\\
    u=0&\text{on }\Sigma_0.
\end{cases}
\end{equation}
Here we have used the notation $\mathbb{D}_{\mathcal{M}}w=\grad w\mathcal{M}^{\m{t}}+\mathcal{M}\grad w^{\m{t}}$, for $\mathcal{M}=\mathcal{A}_\eta$ and $w=M_\eta^{-1}u$.

% _+__+_ -_+__+_ -_+__+_ -_+__+_ -_+__+_ -_+__+_ -_+__+_ -_+__+_ -_+__+_ -_+__+_ -_+__+_ -_+__+_ -_+__+_ -
\subsection{Statement of main result and discussion}\label{section on statement of the main theorem and discussion}
% _+__+_ -_+__+_ -_+__+_ -_+__+_ -_+__+_ -_+__+_ -_+__+_ -_+__+_ -_+__+_ -_+__+_ -_+__+_ -_+__+_ -_+__+_ -

In order to state and discuss our principal results, we must first introduce the function spaces we will employ in our analysis.  We will do so rapidly here, emphasizing that these spaces are more thoroughly developed in Section \ref{section on novel Sobolev spaces}.   Fix $1<r<2$.  For $I$ denoting either the set $(0,b)$ or else $\R$ and $U=\R^2\times I$, we define $L_{r,2}(U)=L^r(\R^2;L^2(I))$ to be the mixed-type Lebesgue space. For $s\in\N$ we define the mixed-type Sobolev spaces $H^s_{r,2}(U)$ modeled on $L_{r,2}(U)$ in the natural way (see Definition \ref{mixed-type sobolev def}).   For $t\in[0,\infty)$, we let $H^{t,r}(\R^2) = \{f \in L^r(\R^2) \;:\; \tbr{D}^{t} f \in L^r(\R^2)\}$ denote the standard Bessel potential Sobolev space (see, e.g., Section 1.3.1 in Grafakos~\cite{MR3243741} or Section 6.2 in Berg and L\"ofstr\"om~\cite{MR0482275}) and let $\tilde{H}^{1+t,r}(\R^2)$ denote the space of $L^{2r/(2-r)}(\R^2)$ functions whose distributional derivatives belong to $H^{t,r}(\R^2)$ (see Definition \ref{defn of subcritical gradient spaces}).  

For $s\in\N$, $1<r<2$ we set $\bf{X}_{s,r}=H^{1+s}_{r,2}(\Omega)\times{H^{2+s}_{r,2}(\Omega;\R^3)}\times\tilde{H}^{5/2+s,r}(\Sigma)$ and $\bf{W}_{s,r}=H^{1+s}_{r,2}(\R^3;\R^{3\times 3})\times H^s_{r,2}(\R^3;\R^3)$.   With the notation established, we come to our main theorem.

\begin{customthm}{1}[Proved in Section~\ref{section on nonlinear analysis}: see Theorem~\ref{coro fourth times the charm for well-posedness}, Proposition~\ref{prop on properties of the flattening map}, and Corollary~\ref{coro on some further conclusions}]\label{the main theorem}
    Let  $1<r<2$ and $\N\ni s>3/r+1$.  Then there exist open sets $W_s\subset\R\times(\R^+)^3\times\bf{W}_{s,r}$ and $\tcb{V_s(\bf{v})}_{\bf{v}\in(\R^+)^3}\subset\bf{X}_{s,r}$, satisfying
    \begin{equation}\label{non degeneracy conditions}
        \tcb{0}\times(\R^+)^3\times\tcb{0}\subset W_s
        \text{ and }
        0\in \bigcap_{\bf{v}\in(\R^+)^3}V_s(\bf{v}),
    \end{equation}
    and a continuous map
    \begin{equation}\label{the solution operator}
        W_s\ni(\gam,\mathfrak{g},\mu,\kappa,\mathcal{T},\mathcal{F})\mapsto(p,u,\eta)\in \bigcup_{\bf{v}\in(\R^+)^3}V_s(\bf{v}) \subset\bf{X}_{s,r}
    \end{equation}
    such that the following hold.
    \begin{enumerate}

        \item \textbf{Classical regularity and flattening map diffeomorphism:}  Let $k=s-2-\tfloor{3/r}\in\N$.  Then $\bf{X}_{s,r} \emb   C^{2+k}_0(\Omega)\times C^{3+k}_0(\Omega;\R^3)\times C^{4+k}_0(\Sigma)$.  Moreover, for every $(p,u,\eta)\in \bigcup_{\bf{v}\in(\R^+)^3}V_s(\bf{v})$, the associated flattening map $\mathfrak{F}_\eta$ defined in~\eqref{definition of the flattening map} is a smooth diffeomorphism from $\Omega$ to $\Omega[\eta]$ that extends to a  $C^{4+k}$ diffeomorphism from $\Bar{\Omega}$ to $\Bar{\Omega[\eta]}$.
        
        \item \textbf{Solution operator:} The map~\eqref{the solution operator} is a solution operator to the flattened system~\eqref{final nonlinear equations} in the sense that for each  $(\gam,\mathfrak{g},\mu,\kappa,\mathcal{T},\mathcal{F})\in W_s$ the corresponding pressure, velocity, and free surface $(p,u,\eta)\in V_s(\mathfrak{g},\mu,\kappa)$ is  
        the unique triple in $V_s(\mathfrak{g},\mu,\kappa)$ classically solving~\eqref{final nonlinear equations} with stress-force data $(\mathcal{T},\mathcal{F})$,  wave speed $\gam$, and physical parameters $(\mathfrak{g},\mu,\kappa)$.  Moreover, the free surface $\eta$ obeys an extra `degenerating anisotropic estimate' in the sense that the composition map
        \begin{equation}\label{degen_aniso}
            W_s\ni(\gam,\mathfrak{g},\mu,\kappa,\mathcal{T},\mathcal{F})\mapsto(p,u,\eta)\mapsto\gam\mathcal{R}_1\eta\in L^r(\Sigma)
        \end{equation}
        is well-defined and continuous. Here $\mathcal{R}_1$ refers to the Riesz transform in the $e_1$-direction.

        \item\textbf{Eulerian transfer:} Each solution to the flattened system~\eqref{final nonlinear equations} produced by~\eqref{the solution operator} gives rise to a classical solution to the stationary-traveling Eulerian formulation of the problem given by system~\eqref{the stationarity ansatz} by undoing the change of unknowns that led from~\eqref{the stationarity ansatz} to~\eqref{final nonlinear equations}.

    \end{enumerate}
\end{customthm}

The theorem statement packages several results together in a fairly concise form, so we will now pause to unpack the content with a few comments and remarks.   The high-level summary of the theorem is that stationary solutions to the free boundary incompressible Navier-Stokes equations are generic, and moreover, every such solution lies along a one parameter family of slowly traveling waves.  Indeed, the theorem guarantees that for every choice of positive physical parameters $\mathfrak{g}$, $\mu$, and $\kappa$ there exists a non-empty open neighborhood of the origin in wave-speed, stress, and force data $(\gam,\mathcal{T},\mathcal{F})$-space for which we can uniquely solve~\eqref{final nonlinear equations}, and the solution depends continuous on the data and wave speed, as well as the physical parameters.

It is worth highlighting both the superfluous and concrete boundaries of our main theorem. We choose to work in three spatial dimensions for the following two reasons.  First, our methods here simply do not work for the two-dimensional variant of~\eqref{final nonlinear equations}. This is due to the fact that in two spatial dimensions the interface is one dimensional, and hence the container space for the free surface function is degenerate for all choices for tangential integrability parameters $1<r$. The choice $r=1$ would be an adequate replacement, but the harmonic analysis methods employed here are unavailable in that setting.  The second reason we chose to study the three dimensional problem is physical relevance.  The entirety of the theorem can be generalized to handle dimension four and higher, but in this case it would actually be possible to construct solutions in a simpler functional framework utilizing only $L^2-$based spaces.   Another hard boundary in our theorem is seen in the signs of the physical parameters.  We crucially use the strict positivity of the coefficient of gravity $\mathfrak{g}$, the viscosity $\mu$, and the surface tension coefficient $\kappa$ and are  simply unable to relax any of these parameters to zero. On the other hand, we believe that the lower regularity threshold of $\N\ni s>3/r+1$ in Theorem~\ref{the main theorem} is a soft boundary. We chose this numerology in an effort to minimize the complexity of the nonlinear analysis, but it could potentially be improved upon with sufficient additional work.

The final remark is on the qualitative nature of the waves produced by Theorem~\ref{the main theorem}. Thanks to the embedding guaranteed by the first item of this theorem, we see that the free surface perturbation $\eta$ decays to zero at infinity.  This means that the waves we construct are solitary waves, to borrow a phrase from the traveling wave literature. Due to the level of generality of our main result, there is not much more we can say about the qualitative nature of our solutions; however, our well-posedness result opens to the door to more detailed qualitative studies given a fixed wave speed and applied stress and force data.

We now state a couple consequences of the main theorem that formalize the above discussion.  The first clarifies what we know for a fixed choice of physical parameters $(\mathfrak{g},\mu,\kappa)\in(\R^+)^3$.

\begin{customcoro}{2}[Proved in the third item of Corollary~\ref{coro on some further conclusions}]\label{corollary 2}
    Let $r$ and $s$ be as in Theorem~\ref{the main theorem}.  Then for each $(\mathfrak{g},\mu,\kappa)\in(\R^+)^3$ there exists an open set $(0,0,0)\in W_s(\mathfrak{g},\mu,\kappa)\subset\R\times\bf{W}_{s,r}$ with the property that for every triple of wave-speed, stress, and force data $(\gam,\mathcal{T},\mathcal{F})\in W_s(\mathfrak{g},\mu,\kappa)$ there exists a unique $(p,u,\eta)\in V_s(\mathfrak{g},\mu,\kappa)$ such that system~\eqref{final nonlinear equations} is satisfied classically.
\end{customcoro}

The second corollary elucidates how we formulate well-posedness of the stationary wave problem.

\begin{customcoro}{3}[Proved in the fourth item of Corollary~\ref{coro on some further conclusions}]\label{corollary 3}
Let $r$ and $s$ as in Theorem~\ref{the main theorem}.  Then there exists an open set
\begin{equation}\label{open set conditions for Zs}
    (\R^+)^3\times\tcb{0}\subset Z_s\subset(\R^+)^3\times\bf{W}_{s,r}
\end{equation}
and a continuous mapping
\begin{equation}
    Z_s\ni(\mathfrak{g},\mu,\kappa,\mathcal{T},\mathcal{F})\mapsto(p,u,\eta)\in\bigcup_{\bf{v}\in(\R^+)^3}V_s(\bf{v}) \subset\bf{X}_{s,r}
\end{equation}
with the property that for each $(\mathfrak{g},\mu,\kappa,\mathcal{T},\mathcal{F})\in Z_s$ there exists a unique $(p,u,\eta)\in V_s(\mathfrak{g},\mu,\kappa)$ such that the stationary free boundary incompressible Navier-Stokes equations, system~\eqref{final nonlinear equations} with $\gam=0$, is satisfied classically.
\end{customcoro}

We now aim to summarize the principal difficulties in proving Theorem~\ref{the main theorem} and our strategies for overcoming them.  This discussion also serves as an outline of the paper.

As is the case for the traveling wave problems studied in \cite{koganemaru2022traveling,leoni2019traveling,nguyen_tice_2022,MR4337506,stevenson2023wellposedness}, the stationary boundary value problem~\eqref{final nonlinear equations} lies in an unbounded domain of infinite measure and possesses a non-compact free boundary.  The equations are quasilinear and do not enjoy a variational formulation.  Consequently, compactness, Fredholm, and variational techniques are unavailable.  This suggests that the production of solutions ought to proceed via a perturbative argument, such as the implicit function theorem, which has proved successful in the aforementioned work on traveling waves. As such, we begin our discussion by stating the linearization of~\eqref{final nonlinear equations} at zero-wave speed around the equilibrium solution:   
\begin{equation}\label{linearization of the nonlinear problem, introduction friendly}
\begin{cases}
    \grad (p+\mathfrak{g}\eta) -\mu \Delta u=f&\text{in }\Omega,\\
    \grad\cdot u=0&\text{in }\Omega,\\
    -(pI - \mu \mathbb{D} u)e_3-\kappa \Delta_{\|}\eta e_3=k&\text{on }\Sigma,\\
    u\cdot e_3=0&\text{on }\Sigma,\\
    u=0&\text{on }\Sigma_0.
\end{cases}
\end{equation}

The most natural linear theory for system~\eqref{linearization of the nonlinear problem, introduction friendly} lies within $L^2$-based Sobolev spaces. In Section~\ref{section on the basic linear theory} we prove that for every choice of $s\in\N$, $f\in H^s(\Omega;\R^3)$, and $k\in H^{1/2+s}(\Sigma;\R^3)$, there exists a solution $p\in H^{1+s}(\Omega)$, $u\in H^{2+s}(\Omega;\R^3)$, and $\eta\in\tilde{H}^{5/2+s}(\Sigma)$ (meaning $\grad\eta\in H^{3/2+s}(\Sigma;\R^2)$) that is unique up to modifications of $\eta$ by constants.  Moreover, we have an estimate of the solution $(p,u,\eta)$ in terms of the data $(f,k)$.

While this basic $L^2$-based linear theory is encouraging, it is ill-suited for the actual task at hand.  The problem is two-fold.  First, there is no canonical choice of $\eta$, as it is only determined up to a constant, and this is highly problematic in using $\eta$ to generate the set $\Omega[\eta]$ in which the nonlinear problem~\eqref{the stationarity ansatz} is posed.  Second, and more severe, is that  the inclusion $\grad\eta\in H^{3/2+s}(\Sigma;\R^2)$ can never provide an estimate of $\eta\in L^\infty(\Sigma)$ for any choice of $s\in\N$.  This is due to the nature of the critical Sobolev embedding in two dimensions, since $\grad\eta\in L^2(\Sigma;\R^2)$ only guarantees that $\eta\in\m{BMO}(\Sigma)$.  The potential unboundedness of $\eta$ is an even more severe obstruction in building $\Omega[\eta]$.  It is worth noting that for the traveling problem with $\gamma \neq 0$, the papers~\cite{koganemaru2022traveling,leoni2019traveling,nguyen_tice_2022,MR4337506,stevenson2023wellposedness} exploit an essential auxiliary estimate of $\gamma \mathcal{R}_1 \eta \in L^2(\Sigma)$, where $\mathcal{R}_1$ is the Riesz transform in the $e_1$ direction, in order to guarantee $\eta$ belongs to a special anisotropic Sobolev space that embeds into  $C^0_0(\Sigma;\R)$; this then overcomes the criticality problem and allows for the construction of solutions in the anisotropic space.  We see from \eqref{degen_aniso} that we obtain an analogous estimate here when $\gamma \neq 0$, but when $\gamma=0$ there is simply no auxiliary estimate available, and so the anisotropic space is of no use.

Our path forward begins with the following observation.  If we could develop a theory for~\eqref{linearization of the nonlinear problem, introduction friendly} that ensured the inclusion $\grad\eta\in W^{\ell,r}(\Sigma;\R^2)$ for some $1\le r<2$ and $\N\ni\ell>2/r-1$, then the subcritical Sobolev embedding would provide a canonical way to modify $\eta$ by a constant to guarantee the inclusion $\eta\in L^{2r/(2-r)}(\Sigma)$.  Due to the embedding $(L^{2r/(2-r)}\cap\dot{W}^{\ell,r})(\Sigma)\emb C^0_0(\Sigma)$, free surface functions in this space are not only admissible for the nonlinear problem, but also enjoy a wealth of nonlinear properties that permit the transition from~\eqref{linearization of the nonlinear problem, introduction friendly} to the full nonlinear problem~\eqref{the stationarity ansatz}.

We thus arrive at the principal task of  developing a linear well-posedness theory for~\eqref{linearization of the nonlinear problem, introduction friendly} that yields $L^r$-estimates (for $r<2$) on the gradient of the free surface and its derivatives. One possible strategy for this would be to pose the problem in purely $L^r$-based Sobolev spaces. There are existing techniques in the literature (for instance, \cite{AS_2003,AS_2003_2,AY_2010,Abels_2005,Abels_2005_2,Abels_2006,AW_2005}) that provide an $L^r$-based well-posedness theory for the Stokes problems with various boundary conditions but not with the coupling to a free surface function $\eta$. Rather than start with this $L^r$-Stokes theory and attempt to build in a coupling to the free surface function, we have instead identified an alternate approach that is more deeply connected to the symmetries of the equilibrium domain and the natural $L^2$-energy structure of the problem.  This technique allows for the simultaneous construction of the solution triple $(p,u,\eta)$, has clear connections to the relatively simple $L^2$-existence theory, and has the potential for generalization to other problems with similar symmetries.

Our approach aims only to develop the $L^r$-theory in the horizontal variables, while maintaining an $L^2$-theory in the vertical variable.  More concretely, we utilize mixed-type Sobolev spaces modeled on the mixed-type Lebesgue spaces $L_{r,2}(\Omega)=L^r(\R^2;L^2(0,b))$ for the bulk unknowns $p$ and $u$ and the bulk data, and we use Bessel potential Sobolev spaces $H^{s,r}(\R^2)$ for the (gradient of the) boundary unknown $\eta$ and the boundary data.  While mixed-type spaces are often used in parabolic problems, we are unaware of their previous uses in elliptic fluid problems.

At first glance it might seem that the mixed nature of these spaces will make them cumbersome to work with, but in fact they are a natural and streamlined choice of a functional framework to satisfy our stated goals, as we now aim to justify.  First, we observe that the domain $\Omega$ is invariant under translations in the two horizontal variables and that the solution operator to system~\eqref{linearization of the nonlinear problem, introduction friendly}, denoted
\begin{equation}\label{solution operator for linearization of the nonlinear problem, introduction friendly}
    T:L^2(\Omega;\R^3)\times H^{1/2}(\Sigma;\R^3)\to H^1(\Omega)\times H^2(\Omega;\R^3)\times\tp{\tilde{H}^{5/2}(\Sigma)/\R},
\end{equation}
with $(p,u,\eta) = T(f,k)$ solving the PDE, commutes with all horizontal translations.  By making the identification $\Sigma\simeq \R^2$ and employing the factorization (see Lemma~\ref{lemma on equivalent norm on the mixed type spaces})
\begin{equation}\label{the factorization identity}
H^s(\Omega;\R^\ell)=H^s(\R^2;L^2((0,b);\R^\ell))\cap L^2(\R^2;H^s((0,b);\R^\ell)) \text{ for } s,\ell\in\N,
\end{equation}
we see that $T$ is a translation-commuting linear operator acting between certain infinite-dimensional Hilbert-valued Sobolev spaces.  Building on some well-established tools in harmonic analysis (see Section~\ref{appendix on translation commuting linear maps}), we deduce from this that $T$ is diagonalized by the Fourier transform in the two horizontal variables.  More precisely, this grants us the existence of an operator-valued symbol
\begin{equation}
    m:\R^2\to\mathcal{L}(L^2((0,b);\C^3)\times\C^3;H^1((0,b);\C)\times H^2((0,b);\C^3)\times\C)
\end{equation}
such that $T=m(D)$ and the operator norm of $T$ is equivalent to a certain weighted $L^\infty$-type norm on the symbol $m$. Harmonic analysis provides numerous frameworks for extending Fourier multiplication operators, such as $m(D)$, from $L^2$-based spaces to $L^r$-based spaces for $1<r<\infty$.  A celebrated tool in this area is the Mikhlin-H\"ormander multiplier theorem; briefly, this result says that if the derivatives of the symbol obey certain estimates, then the corresponding multiplication operator can be uniquely extended from $L^2$-based spaces to $L^r$-based spaces for every $1<r<\infty$.  In our context, with the symbol $m$ and the map $T$, there is an appropriate vector-valued version of this result to which we appeal and subsequently generalize (see Theorems~\ref{hormander mikhlin multiplier theorem} and~\ref{second HM multiplier theorem}).  Taking for granted, for the moment, that $m$ satisfies the necessary symbol estimates, we then learn that 
\begin{multline}\label{factored mixed type}
    T:L^r(\R^2;L^2((0,b);\R^3))\times H^{1/2,r}(\Sigma;\R^3)\to \tp{H^{1,r}(\R^2;L^2((0,b);\R))\cap L^r(\R^2;H^1((0,b);\R))}\\\times\tp{H^{2,r}(\R^2;L^2((0,b);\R^3))\cap L^r(\R^2;H^2((0,b);\R^3))}\times\tp{\tilde{H}^{5/2,r}(\Sigma;\R)/\R},
\end{multline}
is a bounded linear extension of~\eqref{solution operator for linearization of the nonlinear problem, introduction friendly} for any $1<r<\infty$. The mixed-type Sobolev spaces now simply show up by undoing the factorization~\eqref{the factorization identity}, i.e.
\begin{equation}\label{shove together}
    H^{s,r}(\R^2;L^2((0,b);\R^\ell))\cap L^r(\R^2;H^s((0,b);\R^\ell))=H^s_{r,2}(\Omega;\R^\ell) \text{ for }  s,\ell\in\N,\;r\in(1,\infty),
\end{equation}
which means that~\eqref{factored mixed type} rewrites as 
\begin{equation}
    T:H^0_{r,2}(\Omega;\R^3)\times H^{1/2,r}(\Sigma;\R^3)\to H^1_{r,2}(\Omega)\times H^2_{r,2}(\Omega;\R^3)\times\tp{\tilde{H}^{5/2,r}(\Sigma)/\R}.
\end{equation}
In a similar manner, the mixed-spaces admit a simple Hilbert-valued Littlewood-Paley theory that allows for a rapid development of their properties.

Further evidence of the utility of the mixed-type Sobolev spaces and the tangential-$L^r$ framework is seen in the fact that it allows us to verify that the vector-valued symbol $m$ satisfies the necessary  Mikhlin-H\"ormander hypotheses in a surprisingly effective and efficient manner. Our proof requires no more than the vector-valued harmonic analysis toolbox of Section~\ref{appendix on vector-valued harmonic analysis}, paired with the identification of a certain recursive structure present in the $L^2$ theory for~\eqref{linearization of the nonlinear problem, introduction friendly}. In fact, this technique does not rely on any explicit formula for $m$, nor any specific fluid-dynamical structure of the equations themselves, and so we expect it can serve as a general method for other problems posed in domains with a partial translation symmetry. The main idea of our technique is that derivatives of the symbol are obtained from its difference quotients, which can be computed explicitly in terms of compositions of the solution operator $T$ with modulation operators (see Proposition~\ref{proposition on symbol translation}). The solution operator $T$ interacts with modulation in a very simple manner due to the product rule, and this allows us to deduce differentiability properties of the symbol $m$ by recursively employing $T$ itself and the correspondence between its operator norm and $L^\infty$-type norms of $m$. This not only yields the estimates needed to invoke Mikhlin-H\"ormander, but also yields analyticity of $m$ away from the origin (see Theorem~\ref{thm on analyticity of the symbol}).

Section~\ref{section on novel Sobolev spaces} records a number of properties, linear and nonlinear, about the mixed-type Sobolev spaces and the subcritical gradient spaces. We combine these with the $L^2$-linear theory and the above vector-valued harmonic analysis ideas in Section~\ref{section on linear analysis in the mized type spaces}, which culminates in the linear well-posedness result of Theorem~\ref{thm on well-posedness of the linearization in mixed-type Sobolev spaces, II}. In Section~\ref{section on nonlinear analysis} we then formulate the nonlinear system \eqref{final nonlinear equations} as a nonlinear mapping between appropriate mixed-type spaces and then produce solutions via the implicit function theorem. The proofs of Theorem \ref{the main theorem} and Corollaries \ref{corollary 2} and \ref{corollary 3} are recorded in Section \ref{subsection on well-posedness}.

The above discussion has focused entirely on the stationary ($\gamma =0$) problem, so we conclude with a couple comments about the slowly traveling problem ($\gamma \asymp 0)$.  Previous work on the traveling problem \cite{koganemaru2022traveling,leoni2019traveling,nguyen_tice_2022,MR4337506,stevenson2023wellposedness} considered linearized operators with general $\gamma \in \R\setminus\tcb{0}$, but here we only study the case $\gamma =0$.  This explains how our result only ends up handing slowly traveling waves: the solutions with $\gamma \neq 0$ are obtained perturbatively from the $\gamma=0$ analysis.  Based on the $L^2$ theory, one would expect the free surface function to belong to the obvious $L^r-$analog of the anisotropic $L^2$-based Sobolev spaces mentioned above, and this is indeed the case.  To handle the mismatch between these anisotropic spaces with $\gamma \neq 0$ and the isotropic space with $\gamma=0$, we employ a special $\gamma-$dependent Fourier multiplier (see Definition \ref{the gamma parameterization operator}) that reparameterizes the anisotropic function spaces in terms of the stationary isotropic function space.  Inverting this operator (see Proposition \ref{properties of the anisotroptic paramterization operators}) then shows the anisotropic inclusion that is recorded in \eqref{degen_aniso}.

We emphasize that our work establishes continuity of the solution map into the fixed isotropic space used for the stationary problem, even though the free surface function belongs to a strict subspace (determined by the anisotropic estimate \eqref{degen_aniso})  when $\gamma \neq 0$.  The limit $\gamma \to 0$ can then be understood as a singular limit, in the sense that this extra anisotropic estimate degenerates when $\gamma =0$, resulting in a change in the topology of the container space.  Another impact of this singular limit is that the anisotropic parameterization operators we use are at best continuous with respect to $\gamma$ and not differentiable at $\gamma =0$.

% _+__+_ -_+__+_ -_+__+_ -_+__+_ -_+__+_ -_+__+_ -_+__+_ -_+__+_ -_+__+_ -_+__+_ -_+__+_ -_+__+_ -_+__+_ -
\subsection{Notation}\label{notation subsection}
% _+__+_ -_+__+_ -_+__+_ -_+__+_ -_+__+_ -_+__+_ -_+__+_ -_+__+_ -_+__+_ -_+__+_ -_+__+_ -_+__+_ -_+__+_ -

The set $\tcb{0,1,2,\dots}$ is denoted by $\N$;  $\N^+=\N\setminus\tcb{0}$. The positive real numbers are $\R^+=(0,\infty)$. $\F$ denotes either $\R$ or $\C$. The notation $\al\lesssim\be$ means that there exists $C\in\R^+$, depending only on the parameters that are clear from context, for which $\al\le C\be$. To highlight the dependence of $C$ on one or more particular parameters $a,\dots, b$, we will occasionally write $\al\lesssim_{a,\dots,b}\be$. We also express that two quantities $\al$, $\be$ are equivalent, written $\al\asymp\be$ if both $\al\lesssim\be$ and $\be\lesssim\al$. We shall also employ the bracket notation
\begin{equation}\label{bracket notation}
    \tbr{x}=\sqrt{1+|x_1|^2+\dots+|x_\ell|^2} \text{ for } x\in\C^\ell.
\end{equation}
If $\tcb{X_i}_{i=1}^\ell$ are normed spaces and $X$ is their product, endowed with any choice of product norm $\tnorm{\cdot}_X$, then we shall write
\begin{equation}
    \tnorm{x_1,\dots,x_\ell}_{X}=\tnorm{(x_1,\dots,x_\ell)}_{X}  \text{ for } (x_1,\dots,x_\ell)\in X.
\end{equation}

We identify the dual of a complex Banach space $X$, denoted $X^\ast$, as the set of antilinear and continuous functionals, so that the dual pairing is sesquilinear (i.e. linear in the right argument) and, in the case that $X$ is Hilbert, the Riesz map is linear.

If $\mathcal{H}$ is a separable Hilbert space, we will denote the Fourier and inverse Fourier transforms (normalized to be unitary on $L^2$) in the space of $\mathcal{H}$-valued tempered distributions over $\R^d$, $\mathscr{S}^\ast(\R^d;\mathcal{H})$, via $\mathscr{F}$ and $\mathscr{F}^{-1}$, respectively. For functions defined in the equilibrium domain $\Omega$, we view them as vector-valued tempered distributions on $\R^2$ in the natural way; for example, $L^2(\Omega;\C)=L^2(\R^2;L^2((0,b);\C))\emb\mathscr{S}^\ast(\R^2;L^2((0,b);\C))$. It is in these sense that we are to interpret the Fourier transform acting on functions defined on $\Omega$. We frequently make the natural identification $\Sigma\simeq \R^2$ when performing Fourier analysis for functions defined on $\Sigma$.

We write $\grad=(\pd_1,\dots,\pd_d)$ to denote the gradient on $\R^d$ for $d\in\N^+$. We refer to dimensions $2$ and $3$ simultaneously, in which case the $\R^2$-gradient is denoted by $\grad_{\|}=(\pd_1,\pd_2)$, while the $\R^3$ gradient obeys the aforementioned notation. In $\R^2$ we denote the rotated-gradient operator as $\grad_{\|}^\perp=(-\pd_2,\pd_1)$. We also let $D=\grad/2\pi\ii$ or $D=\grad_{\|}/2\pi\ii$, depending on context. The divergence and tangential divergence operators are written $\grad\cdot f=\sum_{j=1}^3\pd_j(f\cdot e_j)$ and $(\grad_{\|},0)\cdot f=\sum_{j=1}^2\pd_j(f\cdot e_j)$, for appropriate $\R^3$-valued functions $f$.

If $\mathcal{H}$ and $\mathcal{K}$ are Hilbert spaces and $m:\R^d\to\mathcal{L}(\mathcal{H};\mathcal{K})$ is a sufficiently nice symbol, we will write $m(D)$ for the linear operator, acting on certain subspaces of tempered distributions, defined via $\mathscr{F}^{-1}[m\mathscr{F}[\cdot]]$. In other words, $m(D)$ is the Fourier multiplication operator corresponding to the symbol $m$. If $\zeta\in\R^d$, we also let $m(D+\zeta)=(m(\cdot+\zeta))(D)$. The vector of Riesz transforms is $\mathcal{R}=(\mathcal{R}_1,\dots,\mathcal{R}_d)$, where $\mathcal{R}_i=|D|^{-1}\pd_i/2\pi$, $i\in\tcb{1,\dots,d}$.

Given an open set $\Gamma \subseteq \R^d$, $k \in \N$, and a normed vector space $V$, we will write $C^k(\Gamma;V)$ for the $k-$times continuously differentiable maps from $\Gamma$ to $V$.  The notation $C^k_0(\Gamma;V)$ denotes the (possible) subspace of functions $f \in C^k(\Gamma;V)$ for which $\sup_{\abs{\alpha} \le k} \abs{\partial^\alpha f(x_n)} \to 0$ along any sequence $\{x_n\}_{n \in \N} \subset \Gamma$ for which $\abs{x_n} \to \infty$ as $n\to\infty$.

We now turn our attention to the notation for various types of standard Sobolev spaces employed in this paper. First, we address certain negative homogeneous type Spaces. For $1<p<2$ we define the space
\begin{equation}\label{definition of the negative homogeneous Sobolev space}
    \dot{H}^{-1,p}(\R^2;\F)=\tcb{f\in H^{-1,p}(\R^2;\F)\;:\;|D|^{-1}f\in L^p(\R^2;\F)},
\end{equation}
which is equipped with the norm $\tsb{f}_{\dot{H}^{-1,p}}=\tnorm{|D|^{-1}f}_{L^p}$. We also define
\begin{equation}\label{definition of the negative homogeneous Sobolev space, 2}
    \dot{H}^{-1}(\R^2;\F)=\tcb{f\in H^{-1}(\R^2;\F)\;:\;|\cdot|^{-1}\mathscr{F}[f]\in L^2(\R^2;\C)}
\end{equation}
and endow it the norm $\tsb{f}_{\dot{H}^{-1}}=\tnorm{|\cdot|^{-1}\mathscr{F}[f]}_{L^2}$. We will sometimes write $\dot{H}^{-1,2}=\dot{H}^{-1}$. We then denote
\begin{equation}
    \hat{H}^s(\Omega;\F)=\bcb{g\in H^s(\Omega;\F)\;:\;\int_0^bg(\cdot,y)\;\m{d}y\in\dot{H}^{-1}(\R^2;\F)}
\end{equation}
which has the norm $\tnorm{g}_{\hat{H}^s}=\sp{\tnorm{g}_{H^s}^2+\ssb{\int_0^bg(\cdot,y)\;\m{d}y}_{\dot{H}^{-1}}^2}^{1/2}$.

For $\R\ni s\ge 1$ we also define the gradient spaces:
\begin{equation}\label{definition of the critical and subcritical gradient spaces}
    \tilde{H}^{s,p}(\R^2;\F)=\begin{cases}
        \tcb{f\in L^1_{\m{loc}}(\R^2;\F)\;:\;\grad f\in H^{s-1,p}(\R^2;\F^2)}&\text{if }p\ge 2,\\
        \tcb{f\in L^{2p/(2-p)}(\R^2;\F)\;:\;\grad f\in H^{s-1,p}(\R^2;\F^2)}&\text{if }1<p<2.
    \end{cases}
\end{equation}
The norm (seminorm if $p\ge 2$) is given by $\tnorm{f}_{\tilde{H}^{s,p}}=\tnorm{\grad f}_{H^{s-1,p}}$. When $p=2$, we shall again write $\tilde{H}^{s}(\R^2;\F)$ in place of $\tilde{H}^{s,2}(\R^2;\F)$. The spaces $\tilde{H}^{s,p}(\R^2;\F)$ are complete for $p<2$ (see Section \ref{section on properties of subcritical gradient spaces} for this and other properties), while the quotient $\tilde{H}^{s,p}(\R^2;\F)/\F$ is complete for $p\ge 2$.

Now we consider the classical Bessel-potential Sobolev spaces. Given $\mathcal{H}$ a separable Hilbert space and $\R\ni s\ge0$ we write
\begin{equation}
    H^{s,p}(\R^2;\mathcal{H})=\tcb{f\in L^p(\R^2;\mathcal{H})\;:\;\tbr{D}^sf\in L^p(\R^2;\mathcal{H})}
\end{equation}
and equip this space with the standard norm $\tnorm{f}_{H^{s,p}\mathcal{H}}=\tnorm{\tbr{D}^sf}_{L^p\mathcal{H}}$. When $p=2$, we simply write $H^s(\R^2;\mathcal{H})$ in place of $H^{s,2}(\R^2;\mathcal{H})$. Since we are considering the Hilbert-valued case, the theory follows from straightforward adaptations of the scalar theory; for more information on the general Banach-valued cases, we refer the reader to Amann~\cite{amann_1997} or Section 5.6 in Hyt\"onen, Neerven, Veraar, and Weis~\cite{MR3617205}.

The trace operators on to the hypersurfaces $\Sigma$ and $\Sigma_0$, acting on functions defined on $\Omega$, are denoted by $\m{Tr}_{\Sigma}$ and $\m{Tr}_{\Sigma_0}$, respectively. We will utilize the following closed subspace of $H^1(\Omega;\F^3)$:
\begin{equation}
    {_0}H^1(\Omega;\F^3)=\tcb{u\in H^1(\Omega;\F^3)\;:\;\m{Tr}_{\Sigma_0}u=0}.
\end{equation}
For functions like $\eta:\Sigma\to\F$ we can view them as defined on $\Omega$ in the natural way, e.g. $\eta(x,y)=\eta(x)$ for $(x,y)\in\Omega$. In particular, the expression of $\grad\eta$ in the bulk equations of say~\eqref{final nonlinear equations} refers to the $\R^3$-vector $(\pd_1\eta,\pd_2\eta,0)$.

% _+__+_ -_+__+_ -_+__+_ -_+__+_ -_+__+_ -_+__+_ -_+__+_ -_+__+_ -_+__+_ -_+__+_ -_+__+_ -_+__+_ -_+__+_ -
\section{Basic linear theory}\label{section on the basic linear theory}
% _+__+_ -_+__+_ -_+__+_ -_+__+_ -_+__+_ -_+__+_ -_+__+_ -_+__+_ -_+__+_ -_+__+_ -_+__+_ -_+__+_ -_+__+_ -

In this section we are concerned with the well-posedness of the following linear system of equations in the framework of $L^2$-based Sobolev spaces:
\begin{equation}\label{linearization of the nonlinear problem}
\begin{cases}
    \grad(p+\mathfrak{g} \eta)-\mu\grad\cdot\mathbb{D}u=f&\text{in }\Omega,\\
    \grad\cdot u=g&\text{in }\Omega,\\
    -(pI-\mu\mathbb{D}u)e_3-\kappa \Delta_{\|}\eta e_3=k&\text{on }\Sigma,\\
    u\cdot e_3=h&\text{on }\Sigma,\\
    u=0&\text{on }\Sigma_0.
\end{cases}
\end{equation}
Here the (complex) data are $f:\Omega\to\C^3$, $k:\Sigma\to\C^3$, $h:\Sigma\to\C$, and $g:\Omega\to\C$ while the (complex) unknowns are $u:\Omega\to\C^3$, $p:\Omega\to\C$, and $\eta:\Sigma\to\C$. One of the minor technical issues with this system of equations is that the lowest order term appearing for $\eta$ is its gradient thus there is a kernel for the differential operator consisting of $p=0$, $u=0$, and $\eta=\text{constant} \in \C$. We get around this issue in the following two ways. First, in Sections~\ref{section on weak solutions} and~\ref{section on strong solutions}, we work in a semi-normed space functional framework, rather than a normed one. More precisely, we utilize the $\tilde{H}^s(\R^2)$ spaces as in~\eqref{definition of the critical and subcritical gradient spaces} as the containers for the linearized free surface variable.

As it turns out, working in seminormed spaces is not ideally suited for the next stage of our linear analysis, Section~\ref{section on vector-valued symbol calculus for the solution map}, in which we perform operator-valued symbol calculus on a solution operator to the linear problem. Thus, our second way of dealing with the kernel of~\eqref{linearization of the nonlinear problem} is that in the latter half of Section~\ref{section on strong solutions}, we use an equivalent reformulation of~\eqref{linearization of the nonlinear problem} for data and solutions both belonging to normed spaces. The reformulation is given by:
\begin{equation}\label{curl formulation of the linearization}
        \begin{cases}
            \mathfrak{g} (\chi,0)+\grad p-\mu \grad\cdot\mathbb{D}u=f&\text{in }\Omega,\\
            \grad\cdot u=g&\text{in }\Omega,\\
            -(pI-\mu \mathbb{D}u)e_3-\kappa \grad_{\|}\cdot\chi e_3=k&\text{on }\Sigma,\\
            \grad_{\|}^\perp\cdot\chi=\omega&\text{on }\Sigma,\\
            u\cdot e_3=h&\text{on }\Sigma,\\
            u=0&\text{on }\Sigma_0.
        \end{cases}
    \end{equation}
    Here the data $f$, $k$, $h$ are the same as before, and $\omega:\Sigma\to\C$ is a new datum. The solution is $(p,u,\chi)$, with $p$ and $u$ as before and $\chi:\Sigma\to\C^2$.

% _+__+_ -_+__+_ -_+__+_ -_+__+_ -_+__+_ -_+__+_ -_+__+_ -_+__+_ -_+__+_ -_+__+_ -_+__+_ -_+__+_ -_+__+_ -
\subsection{Weak solutions}\label{section on weak solutions}
% _+__+_ -_+__+_ -_+__+_ -_+__+_ -_+__+_ -_+__+_ -_+__+_ -_+__+_ -_+__+_ -_+__+_ -_+__+_ -_+__+_ -_+__+_ -

The strategy for the theory of weak solutions is to prove a priori estimates and then handle existence via a sequence of approximate problems. The initial bounds allow us to deduce that this approximating sequence is Cauchy and has a limit that solves the equations.

The following definition sets the notation for the weak solution theory.
\begin{defn}[Weak formulation operators]\label{defn of weak formulation operators}
We define the linear map
    \begin{equation}
        \mathscr{I}:L^2(\Omega;\C)\times{_0}H^1(\Omega;\C^3)\times\tilde{H}^{3/2}(\Sigma;\C)\to({_0}H^1(\Omega;\C^3))^\ast,
    \end{equation}
    through the action
    \begin{equation}\label{the definition of the weak formulation functional}
        \tbr{\mathscr{I}(p,u,\eta),v}_{({_0}H^1)^\ast,{_0}H^1} = \int_{\Omega}\frac{\mu}{2} \mathbb{D}u: \Bar{\mathbb{D}v}-p \Bar{\grad\cdot v}+ \mathfrak{g}\grad\eta\cdot \Bar{v}- \kappa \tbr{\Delta_{\|}\eta,\m{Tr}_{\Sigma}v\cdot e_3}_{H^{-1/2},H^{1/2}}.
    \end{equation}
    Recall that our notational convention is that the $(\cdot)^\ast$ of a Banach space is its anti-dual and the bracket pairing is antilinear in the right argument.
    
    We also define the following closed subspaces of $H^1(\Omega;\C^3)$:
    \begin{equation}
        {_\natural}H^1(\Omega;\C^3)=\tcb{u\in {_0}H^{1}(\Omega;\C^3)\;:\;\grad\cdot u=0,\;\m{Tr}_{\Sigma}u\cdot e_3=0},
    \end{equation}
    and for $\ep\in(0,1)$ 
    \begin{equation}\label{the weird guy space}
        {_\natural}H^1_\ep(\Omega;\C^3)=\tcb{u\in{_\natural}H^1(\Omega;\C^3)\;:\;\supp\mathscr{F}[u]\subseteq\R^2\setminus B(0,\ep)}.
    \end{equation}
    Note that in the above we are interpreting $\m{supp}\mathscr{F}[u]\subseteq\R^2$ as the support of the vector-valued tempered distribution $\mathscr{F}[u]\in\mathscr{S}^\ast(\R^2;L^2((0,b);\C^3))$.
\end{defn}

We quote the following construction of a solution operator to the divergence equation with Dirichlet boundary conditions. Recall that a linear map $T$ on a vector space of functions defined on $\Omega$ is said to be translation commuting, or tangentially translation commuting, if $(TX)(\cdot+h)=T(X(\cdot+h))$ for all functions $X$ and all $h\in \R^3$ with $h\cdot e_3=0$.

\begin{lem}[Solution operators to divergence equations]\label{lem on right inverse to the divergence operator}The following hold.
\begin{enumerate}
    \item There exists a bounded, linear, and translation commuting map $\mathcal{B}$ such that for $\ell\in\N$ we have
\begin{equation}
    \mathcal{B}:\hat{H}^\ell(\Omega;\C)\to H^1_0(\Omega;\C^3)\cap H^{1+\ell}(\Omega;\C^3)
\end{equation}
and for all $f\in \hat{H}^0(\Omega;\C)$ we have
\begin{equation}
    \grad\cdot\mathcal{B}f=f\quad\text{and}\quad\m{Tr}_{\pd\Omega}\mathcal{B}f=0.
\end{equation}
\item There exists a bounded, linear, and translation commuting map $\Bar{\mathcal{B}}$ such that for $\ell\in\N$ we have
\begin{equation}
    \Bar{\mathcal{B}}:H^\ell(\Omega;\C)\to {_0}H^1(\Omega;\C^3)\cap H^{1+\ell}(\Omega;\C^3)
\end{equation}
and for all $f\in H^0(\Omega;\C)$ we have $\grad\cdot\Bar{\mathcal{B}}f=f$.
\item There exists a bounded, linear, and translation commuting map $\mathcal{B}_0$ such that for $\ell\in\N$ we have
\begin{equation}
    \mathcal{B}_0:H
    ^{1/2+\ell}(\Sigma;\C)\cap\dot{H}^{-1}(\Sigma;\C)\to{_0}H^1(\Omega;\C^3)
\end{equation}
and for all $\varphi\in H^{1/2}(\Sigma;\C)\cap\dot{H}^{-1}(\Sigma;\C)$ we have
\begin{equation}
    \grad\cdot\mathcal{B}_0\varphi=0\quad\text{and}\quad\m{Tr}_{\Sigma}\mathcal{B}_0\varphi=\varphi e_3.
\end{equation}
\end{enumerate}
\end{lem}
\begin{proof}
    The $\R-$valued variants of these operators are constructed in Proposition C.2 and Corollaries C.3 and C.4 in Stevenson and Tice~\cite{stevenson2023wellposedness}. Inspection of the proof shows that the solution operators are indeed translation commuting.  The $\C$-valued assertions above follow from separate considerations of real and imaginary parts.
\end{proof}

We now prove a priori estimates for system~\eqref{linearization of the nonlinear problem} in the reduced case that $g=0$ and $h=0$.

\begin{prop}[A priori estimates for weak solutions]\label{proposition on a priori estimates for weak solutions}
Suppose that
\begin{equation}
    (p,u,\eta)\in L^2(\Omega;\C)\times{_\natural}H^1(\Omega;\C^3)\times\tilde{H}^{3/2}(\Sigma;\C) \text{ and } F\in({_0}H^1(\Omega;\C^3))^\ast
\end{equation}
satisfy the equation
\begin{equation}\label{weak solution identity}
    \mathscr{I}(p,u,\eta)=F,
\end{equation}
or in other words, we have a weak solution to~\eqref{linearization of the nonlinear problem}. Then we have the a priori estimate
\begin{equation}\label{weak solution estimate}
\tnorm{p,u,\eta}_{L^2\times H^1\times\tilde{H}^{3/2}}\lesssim\tnorm{F}_{({_0}H^1)^\ast},
\end{equation}
with an implicit constant depending on $\mathfrak{g}$, $\kappa$, and $\mu$.
\end{prop}
\begin{proof}

Fix $\lambda\in(0,1)$, and let $\eta_\lambda=\mathscr{F}^{-1}[\mathds{1}_{\R^2\setminus B(0,\lambda)}\mathscr{F}[\eta]]\in H^{3/2}(\Sigma;\C)$. Then $(p,u,\eta_\lambda)$ solves the equation
\begin{equation}\label{mod_weak_form}
    \mathscr{I}(p,u,\eta_\lambda)=F-\mathscr{I}(0,0,\eta-\eta_\lambda).
\end{equation}
Testing this with $u$ and integrating by parts, we acquire the identity
\begin{equation}
    \tbr{F-\mathscr{I}(0,0,\eta-\eta_{\lambda}),u}_{({_0}H^1)^\ast,{_0}H^1}=\int_{\Omega}\frac{\mu}{2}|\mathbb{D}u|^2-p  \Bar{\grad\cdot u}+ \mathfrak{g} \grad\eta_\lambda\cdot \Bar{u}=\int_{\Omega}\frac{\mu}{2} |\mathbb{D}u|^2.
\end{equation}
Thus, by applying Korn's inequality (see, for instance, Proposition A.3 in Stevenson and Tice~\cite{stevenson2023wellposedness}) and sending $\lambda\to0$, we obtain the estimate
\begin{equation}\label{the control on w}
    \tnorm{u}_{H^1}\lesssim\tnorm{F}_{({_0}H^1)^\ast}.
\end{equation}

We now derive an estimate on $\eta$. With $\eta_\lambda$ as before, we define  $v_\lambda\in{_0}H^1(\Omega;\C^3)$ via $v_\lambda=-\mathcal{B}_0(\tbr{\grad_{\|}}^{-1}\Delta_{\|}\eta_{\lambda})$, with $\mathcal{B}_0$ from Lemma~\ref{lem on right inverse to the divergence operator}. The lemma provides the bound $\tnorm{v_\lambda}_{H^1}\lesssim\tnorm{\eta}_{\tilde{H}^{3/2}}$. With the understanding that duality pairings are antilinear in the right argument, we then test \eqref{mod_weak_form} with $v_\lambda$ and integrate by parts to learn that
\begin{multline}
    \tbr{F-\mathscr{I}(0,0,\eta-\eta_\lambda),v_\lambda}_{({_0}H^1)^\ast,{_0}H^1}=\int_{\Omega} \frac{\mu}{2} \mathbb{D}u: \Bar{\mathbb{D}v_\lambda} + \mathfrak{g}\grad\eta_{\lambda}\cdot \Bar{v_\lambda} + \kappa \tbr{\Delta_{\|}\eta_\lambda,\tbr{\grad_{\|}}^{-1}\Delta_{\|}\eta_\lambda}_{H^{-1/2},H^{1/2}}\\
    =\int_{\Omega}\frac{\mu}{2} \mathbb{D}u:\Bar{\mathbb{D}v_\lambda} +  \tbr{(\mathfrak{g}-\kappa \Delta_{\|})\eta_{\lambda},-\tbr{\grad_{\|}}^{-1}\Delta_{\|}\eta_\lambda}_{H^{-1/2},H^{1/2}},
\end{multline}
from which we deduce the estimate
\begin{equation}
    \tnorm{\eta_\lambda}_{\tilde{H}^{3/2}}^2\lesssim\tnorm{\eta}_{\tilde{H}^{3/2}}\tp{\tnorm{u}_{H^1}+\tnorm{F-\mathscr{I}(0,0,\eta-\eta_\lambda)}}_{({_0}H^1)^\ast}.
\end{equation}
By sending $\lambda\to0$ and combining with the already established estimate on $u$, we then derive the bound $\tnorm{\eta}_{\tilde{H}^{3/2}}\lesssim\tnorm{F}_{({_0}H^1)^\ast}$.

Finally, we derive an estimate on $p$. For this, we test \eqref{weak solution identity} with $\Bar{\mathcal{B}}p\in{_0}H^1(\Omega;\C^3)$, where $\Bar{\mathcal{B}}$ is again from Lemma~\ref{lem on right inverse to the divergence operator}, to see that
\begin{equation}
    \tbr{F,v}_{({_0}H^1)^\ast,{_0}H^1}=\int_{\Omega}\f{\mu}{2}\mathbb{D}u: \Bar{\mathbb{D}v}+\mathfrak{g}\grad\eta\cdot \Bar{v}-\abs{p}^2,
\end{equation}
which then implies the estimate
\begin{equation}
    \tnorm{p}_{L^2}\lesssim\tnorm{u,\eta,F}_{H^1\times\tilde{H}^{3/2}\times({_0}H^1)^\ast}\lesssim\tnorm{F}_{({_0}H^1)^\ast}.
\end{equation}
Synthesizing the above estimates then completes the proof.
\end{proof}

Our next result examines the existence of weak solutions.  

\begin{prop}[Existence and uniqueness of weak solutions]\label{proposition on existence of weak solutions}
    For any $F\in({_0}H^1(\Omega;\C^3))^\ast$ there exists a $(p,u,\eta)\in L^2(\Omega;\C)\times{_\natural}H^1(\Omega;\C^3)\times\tilde{H}^{3/2}(\Sigma;\C)$ satisfying~\eqref{weak solution identity}.  The triple $(p,u,\eta)$ is unique modulo changes of $\eta$ by constant functions.
\end{prop}
\begin{proof}
    Uniqueness, modulo constants in the linearized free surface variable, is a consequence of estimate~\eqref{weak solution estimate} from Proposition~\ref{proposition on a priori estimates for weak solutions}. To prove existence let $\ep\in(0,1)$ and consider the sesquilinear form $B:{{_\natural}H_\ep^1(\Omega;\C^3)}\times {{_\natural}H_\ep^1(\Omega;\C^3)}\to\C$ given by 
    \begin{equation}
         B(u,v)=\int_{\Omega}\frac{\mu}{2} \mathbb{D}u:\Bar{\mathbb{D}v}.
    \end{equation}
    $B$ is bounded and also coercive thanks to the Korn inequality (see, e.g., Proposition A.3 in~\cite{stevenson2023wellposedness}).  Thus, the Lax-Milgram lemma shows that for every $F\in({_0}H^1(\Omega;\C^3))^\ast$ (which defines an element of $({_\natural}H_\ep^1(\Omega;\C^3))^\ast$ via restriction), there exists a unique $u_\ep\in {{_\natural}H_\ep^1(\Omega;\C^3)}$ such that $B(u_\ep,v)=\tbr{F,v}$ for all $v\in{{_\natural}H_\ep^1(\Omega;\C^3)}$.

    Next, we introduce the Hilbert spaces
    \begin{equation}
    \begin{split}
        L^2_\ep(\Omega;\C) &= \tcb{g\in L^2(\Omega;\C)\;:\;\supp\mathscr{F}[g]\subseteq\R^2\setminus B(0,\ep)}, \\
        {_0}H^1_\ep(\Omega;\C^3) &= \tcb{v\in H^1_0(\Omega;\C^3)\;:\;\supp\mathscr{F}[v]\subseteq\R^2\setminus B(0,\ep)}, \\
        H^s_{\ep}(\Sigma;\C) &= \tcb{h\in H^s(\Sigma;\C)\;:\;\supp\mathscr{F}[h]\subseteq\R^2\setminus B(0,\ep)} \text{ for }  s\in\R,
    \end{split}
    \end{equation}
    which are all endowed with the inner-product from their defining container spaces. Clearly, we have the embeddings $L^2_\ep(\Omega;\C)\emb\hat{H}^0(\Omega;\C)$,  and  $H^s_\ep(\Sigma;\C)\emb\dot{H}^{-1}(\Sigma;\C)$ for any  $s\ge-1$. Hence, Lemma~\ref{lem on right inverse to the divergence operator} allows us to consider the bounded antilinear functional
    \begin{equation}
        L^2_\ep(\Omega;\C)\times H^{1/2}_\ep(\Sigma;\C)\ni(g,h)
        \mapsto G_\ep(g,h) = B(u_\ep,\mathcal{B}g+\mathcal{B}_0h)-\tbr{F,\mathcal{B}g+\mathcal{B}_0h} \in \C
    \end{equation}
    and apply the Riesz-representation theorem to acquire $(q_\ep,\zeta_\ep)\in L^2_\ep(\Omega;\C)\times H_\ep^{1/2}(\Sigma;\C)$ such that 
     \begin{equation} \label{riesz rep id}
        G_\ep(g,h)= (q_\ep,g)_{L^2} +   (\tbr{D}^{1/2}\zeta_\ep, \tbr{D}^{1/2} h)_{L^2(\Sigma)}  \text{ for all } (g,h)\in L^2_\ep(\Omega;\C)\times H^{1/2}_\ep(\Sigma;\C).
    \end{equation}   
    For any $v\in{_0}H^1_\ep(\Omega;\C^3)$ we can use Lemma~\ref{lem on right inverse to the divergence operator} to decompose $v=\mathcal{P}v+\mathcal{Q}v$ via
    \begin{equation}
         \mathcal{P}v=v-\mathcal{B}(\grad\cdot v)-\mathcal{B}_0(\m{Tr}_\Sigma v\cdot e_3)
         \text{ and } \mathcal{Q}v=\mathcal{B}(\grad\cdot v)+\mathcal{B}_0(\m{Tr}_\Sigma v\cdot e_3),
    \end{equation}
    for bounded, linear, and translation commuting maps
    \begin{equation}
        \mathcal{P}:{_0}H^1_\ep(\Omega;\C^3)\to{_\natural}H^1_\ep(\Omega;\C^3)
        \text{ and }
        \mathcal{Q}:{_0}H^1_\ep(\Omega;\C^3)\to{_0}H^1_\ep(\Omega;\C^3).
    \end{equation}

    Now, by the construction of $u_\ep$ we know that for any $v\in{_0}H^1_\ep(\Omega;\C^3)$ we have the identity $B(u_\ep,\mathcal{P}v)-\tbr{F,\mathcal{P}v}=0$, and hence, by the definition of $G_\ep$ and identity~\eqref{riesz rep id}, we have
    \begin{multline}
        B(u_\ep,v)-\tbr{F,v}=B(u_\ep,\mathcal{Q}v)-\tbr{F,\mathcal{Q}v}=G_\ep(\grad\cdot v,\m{Tr}_\Sigma v\cdot e_3)\\
        =\int_{\Omega}q_\ep \Bar{\grad\cdot v} +   (\tbr{D}^{1/2} \zeta_\ep,\tbr{D}^{1/2} \m{Tr}_{\Sigma}v\cdot e_3)_{L^{2}(\Sigma)}.
    \end{multline}
    We then set
    \begin{equation}
        \eta_\ep =  \tbr{D} (-\mathfrak{g} + \kappa \Delta_{\|})^{-1} \zeta_\ep \in H^{3/2}(\Sigma;\C),
    \end{equation}
    $p_\ep=q_\ep-\mathfrak{g}\eta_\ep\in L^2_\ep(\Omega;\C)$, and $F_\ep=\mathds{1}_{\R^2\setminus B(0,\ep)}(D)F\in({_0}H^1(\Omega;\C^3))^\ast$ to learn from this identity and  Definition~\ref{defn of weak formulation operators} that 
    \begin{equation}\label{weak_form_ep}
        \mathscr{I}(p_\ep,u_\ep,\eta_\ep)=F_\ep.
    \end{equation}

    Pick a sequence $\{\ep_n\}_{n\in \N} \subset (0,1)$ such that $\ep_n \to 0$ as $n \to \infty$.  We claim that $\tcb{(p_{\ep_n},u_{\ep_n},\eta_{\ep_n})}_{n \in \N} \subseteq L^2(\Omega;\C)\times{_\natural}H^1(\Omega;\C^3)\times\tilde{H}^{3/2}(\Sigma;\C)$ is Cauchy.       To this end, we first note that~\eqref{weak_form_ep} shows that
    \begin{equation}
        \mathscr{I}(p_{\ep_n}-p_{\ep_m},u_{\ep_n}-u_{\ep_m},\eta_{\ep_n}-\eta_{\ep_m})=(\mathds{1}_{\R^2\setminus B(0,\ep_n)}-\mathds{1}_{\R^2\setminus B(0,\ep_m)})(D)F.
    \end{equation}
    Then the a priori estimates for weak solutions in Proposition~\ref{proposition on a priori estimates for weak solutions} grant the estimate
    \begin{multline}
        \tnorm{p_{\ep_1}-p_{\ep_0},u_{\ep_1}-u_{\ep_0},\eta_{\ep_1}-\eta_{\ep_0}}_{L^2\times H^1\times\tilde{H}^{3/2}}\lesssim\tnorm{(\mathds{1}_{\R^2\setminus B(0,\ep_n)}-\mathds{1}_{\R^2\setminus B(0,\ep_m)})(D)F}_{({_0}H^1)^\ast}\\
        \lesssim\tnorm{(\mathds{1}_{\R^2\setminus B(0,\ep_n)}-\mathds{1}_{\R^2\setminus B(0,\ep_m)})(D)(-\Delta)^{-1}F}_{{_0}H^1},
    \end{multline}
    where $(-\Delta)^{-1}$ is the (translation commuting) inverse to the $\Sigma_0$-Dirichlet $\Sigma$-Neumann Laplacian in $\Omega$.    The claim is then proved by noting that
    \begin{equation}
        \limsup_{n,m \to\infty }\tnorm{(\mathds{1}_{\R^2\setminus B(0,\ep_n)}-\mathds{1}_{\R^2\setminus B(0,\ep_m)})(D)(-\Delta)^{-1}F}_{{_0}H^1}=0,
    \end{equation}
    which follows from Lemma~\ref{lemma on equivalent norm on the mixed type spaces}, Plancherel's theorem, and the monotone convergence theorem.  The claim is proved.

    With the claim in hand, we send $n \to \infty$ to obtain $(p,u,\eta)$ belonging to the same space as the sequence. Testing against $v \in {_0}H^1(\Omega;\C^3)$ in identity~\eqref{weak_form_ep}, and sending $n \to \infty$, we then conclude that the limit $(p,u,\eta)$ satisfies~\eqref{weak solution identity}.
    
\end{proof}

% _+__+_ -_+__+_ -_+__+_ -_+__+_ -_+__+_ -_+__+_ -_+__+_ -_+__+_ -_+__+_ -_+__+_ -_+__+_ -_+__+_ -_+__+_ -
\subsection{Strong solutions}\label{section on strong solutions}
% _+__+_ -_+__+_ -_+__+_ -_+__+_ -_+__+_ -_+__+_ -_+__+_ -_+__+_ -_+__+_ -_+__+_ -_+__+_ -_+__+_ -_+__+_ -

The purpose of this subsection is to obtain estimates in strong norms of solutions to the equations~\eqref{linearization of the nonlinear problem} and~\eqref{curl formulation of the linearization}.  We begin with the former.

\begin{thm}[Analysis of strong solutions, I]\label{theorem on analysis of strong solutions}
Let $s\in\N$. For every
\begin{equation}\label{label data}
    (g,f,k,h)\in H^{1+s}(\Omega;\C)\times H^s(\Omega;\C^3)\times H^{1/2+s}(\Sigma;\C^3)\times H^{3/2+s}(\Sigma;\C)
\end{equation}
satisfying
\begin{equation}\label{divergence compatibility condition}
    h-\int_0^bg(\cdot,y)\;\m{d}y\in\dot{H}^{-1}(\Sigma;\C)
\end{equation}
there exists a unique (again, with the understanding that $\eta$ is only unique modulo constant functions)
\begin{equation}\label{label solution}
    (p,u,\eta)\in H^{1+s}(\Omega;\C)\times H^{2+s}(\Omega;\C^3)\times\tilde{H}^{5/2+s}(\Sigma;\C)
\end{equation}
such that system~\eqref{linearization of the nonlinear problem} is solved with data~\eqref{label data} and solution~\eqref{label solution}; moreover, we have the estimate
\begin{equation}\label{the estimate on strong solutions}
\tnorm{p,u,\eta}_{H^{1+s}\times H^{2+s}\times\tilde{H}^{5/2+s}}\lesssim\tnorm{g,f,k,h}_{H^{1+s}\times H^s\times H^{1/2+s}\times H^{3/2+s}}+\bsb{h-\int_0^bg(\cdot,y)\;\m{d}y}_{\dot{H}^{-1}}.
\end{equation}
\end{thm}
\begin{proof}
     We begin by introducing a useful linear operator.  Let
     \begin{equation}
         \tilde{\mathcal{B}}:\tcb{(g,h)\in H^{1+s}(\Omega;\C)\times H^{3/2+s}(\Omega;\C)\;:\;\eqref{divergence compatibility condition}\text{ is satisfied}}\to{_0}H^{2+s}(\Omega;\C^3)
     \end{equation}
     be defined via
     \begin{equation}
         \tilde{\mathcal{B}}(g,h)=\Bar{\mathcal{B}}g+\mathcal{B}_0(h-\m{Tr}_\Sigma\Bar{\mathcal{B}}g\cdot e_3),
     \end{equation}
    where $\Bar{\mathcal{B}}$ and $\mathcal{B}_0$ are from Lemma~\ref{lem on right inverse to the divergence operator}; $\tilde{\mathcal{B}}$ is well-defined thanks to the lemma and the fact that
     \begin{equation}
         h-\m{Tr}_\Sigma\Bar{\mathcal{B}}g\cdot e_3=\bp{h-\int_0^bg(\cdot,y)\;\m{d}y}+\bp{\int_0^bg(\cdot,y)\;\m{d}y-\m{Tr}_\Sigma\Bar{\mathcal{B}}g\cdot e_3}\in\dot{H}^{-1}(\Sigma;\C).
     \end{equation}
    Given a data tuple~\eqref{label data}, we set $\tilde{f}=f+\mu\grad\cdot\mathbb{D}\tilde{\mathcal{B}}(g,h)\in H^s(\Omega;\C^3)$ and $\tilde{k}=k-\mu \m{Tr}_\Sigma\mathbb{D}\tilde{\mathcal{B}}(g,h)e_3\in H^{1/2+s}(\Sigma;\C^3)$.  Thanks to the mapping properties of $\tilde{\mathcal{B}}$, the reduced data satisfy the estimate
\begin{equation}\label{estimate on the reduced data}
    \tnorm{\tilde{f},\tilde{k}}_{H^s\times H^{1/2+s}}\lesssim\tnorm{g,f,k,h}_{H^{1+s}\times H^s\times H^{1/2+s}\times H^{3/2+s}}+\bsb{h-\int_0^bg(\cdot,y)\;\m{d}y}_{\dot{H}^{-1}}.
\end{equation}

    We then consider the reduced problem of finding 
\begin{equation}\label{dizzy_data}
        (p,w,\eta)\in H^{1+s}(\Omega;\C)\times \tp{ {_\natural}H^1(\Omega;\C^3)\cap H^{2+s}(\Omega;\C^3) }\times\tilde{H}^{5/2+s}(\Sigma;\C)
\end{equation}
solving 
\begin{equation}\label{linearization of the nonlinear problem, reduced}
\begin{cases}
    \grad(p+\mathfrak{g} \eta)-\mu \grad\cdot\mathbb{D}w=\tilde{f}&\text{in }\Omega,\\
    \grad\cdot w=0&\text{in }\Omega,\\
    -(pI-\mu \mathbb{D}w)e_3-\kappa \Delta_{\|}\eta e_3=\tilde{k}&\text{on }\Sigma,\\
w\cdot e_3=0&\text{on }\Sigma,\\
    w=0&\text{on }\Sigma_0,
\end{cases}
\end{equation}
where the reduced data  $(\tilde{f},\tilde{k})$ are determined as above by a data tuple~\eqref{label data}.  We claim that the data $(g,f,k,h)$ uniquely determine solutions to the reduced problem, provided they exist.  Indeed, if $(p,w,\eta)$ solve the reduced problem with data  $(g,f,k,h)=0$, then $(\tilde{f},\tilde{k})=0$ and $(p,w,\eta)$ satisfy  $\mathscr{I}(p,w,\eta)=0$; then the a priori estimates of Proposition~\ref{proposition on a priori estimates for weak solutions} imply that $(p,w,\eta)=0$.  This proves the claim.

The connection between the reduced problem and the original is as follows.  Given $(p,u,\eta)$ as in~\eqref{label solution} solving~\eqref{linearization of the nonlinear problem}, then upon setting  $w=u-\tilde{\mathcal{B}}(g,h)\in {_\natural}H^1(\Omega;\C^3)\cap H^{2+s}(\Omega;\C^3)$ we arrive at a solution $(p,w,\eta)$ to the reduced system.  Conversely, if $(p,w,\eta)$ as in \eqref{dizzy_data} solve the reduced problem, then we obtain a solution to the original problem by setting $u=w+\tilde{\mathcal{B}}(g,h)$; moreover, 
\begin{equation}
    \tnorm{u}_{H^{2+s}}   \lesssim     \tnorm{w}_{H^{2+s}}  + \tnorm{g,f,k,h}_{H^{1+s}\times H^s\times H^{1/2+s}\times H^{3/2+s}}+\bsb{h-\int_0^bg(\cdot,y)\;\m{d}y}_{\dot{H}^{-1}}.
\end{equation}
We thus reduce to solving the reduced problem and deriving the high regularity bounds
    \begin{equation}\label{dig_that_reg}
        \tnorm{p,w,\eta}_{H^{1+s}\times H^{2+s}\times\tilde{H}^{5/2+s}}\lesssim\tnorm{\tilde{f},\tilde{k}}_{H^s\times H^{1/2+s}}.
    \end{equation}

Now, with $(\tilde{f},\tilde{k})$ in hand, we define  $F\in\tp{{_0}H^1(\Omega;\C^3)}^\ast$ via
    \begin{equation}\label{weak forcing}
        \br{F,v}  =  \int_{\Omega}\tilde{f}\cdot \Bar{v}+\int_{\Sigma}\tilde{k}\cdot \Bar{v}\in \C,
    \end{equation}
and use Proposition~\ref{proposition on existence of weak solutions} to obtain a weak solution $(p,w,\eta)$ to the reduced system.  To complete the proof, it is thus sufficient to prove that for every $s\in\N$ and every $(\tilde{f},\tilde{k})\in H^s(\Omega;\C^3)\times H^{1/2+s}(\Sigma;\C)$ the associated unique weak solution $(p,w,\eta)\in L^2(\Omega;\C)\times{_\natural}H^1(\Omega;\C^3)\times \tilde{H}^{3/2}(\Sigma;\C)$ to~\eqref{linearization of the nonlinear problem, reduced} satisfies the higher regularity bounds \eqref{dig_that_reg}.

    We proceed via induction. The case $s=0$ is handled first. We let $\lambda\in(0,1)$ and apply $|D|\mathds{1}_{A_\lambda}(D)$ to weak solution identity (here $D=\grad_{\|}/2\pi\ii$ and $A_\lambda=B(0,\lambda^{-1})\setminus\Bar{B(0,\lambda)}$) and obtain that
    \begin{equation}
        \mathscr{I}(|D|\mathds{1}_{A_\lambda}(D)p,|D|\mathds{1}_{A_\lambda}(D)u,|D|\mathds{1}_{A_\lambda}(D)\eta)=|D|\mathds{1}_{A_\lambda}(D)F,
    \end{equation}
    where $F$ is as in~\eqref{weak forcing}. Thus we may invoke the a priori estimates of Proposition~\ref{proposition on a priori estimates for weak solutions} to bound
    \begin{equation}\label{tangential regularity}
        \tnorm{|D|p,|D|w,|D|\eta}_{L^2\times H^1\times\tilde{H}^{3/2}}\lesssim\limsup_{\lambda\to 0}\tnorm{|D|\mathds{1}_{A_\lambda}(D)F}_{({_0}H^1)^\ast}\lesssim\tnorm{\tilde{f},\tilde{k}}_{L^2\times H^{1/2}},
    \end{equation}
    which is the desired tangential regularity. To establish normal regularity, we note that
    \begin{equation}\label{normal system 1}
    \pd_3p=\mu \Delta_{\|}w\cdot e_3-\mu (\grad_{\|},0)\cdot\pd_3w+\tilde{f}\cdot e_3
\end{equation}
and
\begin{equation}\label{normal system 2}
    \mu \pd_3^2w=-\mu \Delta_{\|}w+\grad(p+\mathfrak{g} \eta)-\tilde{f}.
\end{equation}
Identity~\eqref{normal system 1} (paired with~\eqref{tangential regularity}) establishes that $\pd_3p\in L^2(\Omega;\C)$. Then we use identity~\eqref{normal system 2} to establish that $\pd_3^2w\in L^2(\Omega;\C^3)$ as well. This completes the proof of the base case.

Now suppose that $s\in\N$ and assume the induction hypothesis at $s$. Further suppose
\begin{equation}
    (\tilde{f},\tilde{k})\in H^{1+s}(\Omega;\R^3)\times H^{3/2+s}(\Sigma;\R^3).
\end{equation}
By the induction hypothesis and a tangential regularity argument similar to the one used in the base case, we obtain the estimate
\begin{equation}
    \tnorm{p,w,\eta}_{H^{1+s}\times H^{2+s}\times\tilde{H}^{5/2+s}}+\sum_{j=1}^{2}\tnorm{\pd_jp,\pd_jw,\pd_j\eta}_{H^{1+s}\times H^{2+s}\times\tilde{H}^{5/2+s}}\lesssim\tnorm{\tilde{f},\tilde{k}}_{H^{1+s}\times H^{3/2+s}}.
\end{equation}
To complete the proof, we once more employ identities~\eqref{normal system 1} and~\eqref{normal system 2} to estimate $\pd_3p$ and $\pd_3^2w$ as before, which then proves the induction hypothesis at level $s+1$.  
\end{proof}

Our final result of this subsection reformulates the previous result in an equivalent way that avoids the use of seminormed spaces.  This will be the main take away of our linear analysis for utilization in the next section.

\begin{thm}[Analysis of strong solutions, II]\label{thm on analysis of strong solutions, II}
    Let $s\in\N$. For every 
    \begin{equation}\label{this space is too long dear liza}
        (g,f,k,h,\omega)\in H^{1+s}(\Omega;\C)\times H^s(\Omega;\C^3)\times H^{1/2+s}(\Sigma;\C^3)\times H^{3/2+s}(\Sigma;\C)\times H^{1/2+s}(\Sigma;\C)
    \end{equation}
    satisfying
    \begin{equation}\label{compatibility conditions}
        h-\int_0^bg(\cdot,y)\;\m{d}y\in\dot{H}^{-1}(\Sigma;\C)
        \text{ and }
        \omega\in\dot{H}^{-1}(\Sigma;\C)
    \end{equation}
    there exists a unique
    \begin{equation}
        (p,u,\chi)\in H^{1+s}(\Omega;\C)\times H^{2+s}(\Omega;\C^3)\times H^{3/2+s}(\Sigma;\C^2)
    \end{equation}
    such that the equations~\eqref{curl formulation of the linearization} are satisfied. Moreover, we have the estimate
    \begin{multline}\label{curl form est}
        \tnorm{p,u,\chi}_{H^{1+s}\times H^{2+s}\times H^{3/2+s}}\lesssim\tnorm{g,f,k,h,\omega}_{H^{1+s}\times H^{s}\times H^{1/2+s}\times H^{3/2+s}\times H^{1/2+s}}\\+\bsb{h-\int_0^bg(\cdot,y)\;\m{d}y,\omega}_{\dot{H}^{-1}\times\dot{H}^{-1}}
    \end{multline}
\end{thm}
\begin{proof}

    For any $\chi \in H^{3/2+s}(\Sigma;\C^2)$ we have that
    \begin{equation}
        \chi=\Delta_{\|}^{-1}\grad_{\|}\grad_{\|}\cdot\chi+\Delta_{\|}^{-1}\grad_{\|}^\perp\grad_{\|}^\perp\cdot\chi.
    \end{equation}
    Hence, given a solution to  \eqref{curl formulation of the linearization},   
    we can set $\eta=\Delta_{\|}^{-1}\grad_{\|}\cdot\chi\in\tilde{H}^{5/2+s}(\Sigma;\C)$ and observe that
    \begin{equation}\label{equivalent curl formulation of the linearization}
        \begin{cases}
            \grad(p+\mathfrak{g} \eta)-\mu \grad\cdot\mathbb{D}u=f-\mathfrak{g}(\Delta^{-1}_{\|}\grad_{\|}^\perp\omega,0)&\text{in }\Omega,\\
            \grad\cdot u=g&\text{in }\Omega,\\
            -(pI-\mu \mathbb{D}u)e_3-\kappa \Delta_{\|}\eta e_3=k&\text{on }\Sigma,\\
            u\cdot e_3=h&\text{on }\Sigma,\\
            u=0&\text{on }\Sigma_0.
        \end{cases}
    \end{equation}

    On the other hand, given a data tuple $(g,f,k,h,\omega)$, we may use Theorem~\ref{theorem on analysis of strong solutions} to obtain the existence of  a solution triples $(p,u,\eta)$ to~\eqref{equivalent curl formulation of the linearization}, which is unique modulo constants in the free surface variable.  The solution also obeys the estimate 
    \begin{multline}\label{curl form est eta}
        \tnorm{p,u,\eta}_{H^{1+s}\times H^{2+s}\times \tilde{H}^{5/2+s}}\lesssim\tnorm{g,f,k,h,\omega}_{H^{1+s}\times H^{s}\times H^{1/2+s}\times H^{3/2+s}\times H^{1/2+s}}\\+\bsb{h-\int_0^bg(\cdot,y)\;\m{d}y,\omega}_{\dot{H}^{-1}\times\dot{H}^{-1}}.
    \end{multline}
    We then obtain the unique solution to \eqref{curl formulation of the linearization} upon setting $\chi=\grad_{\|}\eta+\Delta_{\|}^{-1}\grad_{\|}^{\perp}\omega$.  The bound \eqref{curl form est} follows easily from \eqref{curl form est eta}.
\end{proof}

% _+__+_ -_+__+_ -_+__+_ -_+__+_ -_+__+_ -_+__+_ -_+__+_ -_+__+_ -_+__+_ -_+__+_ -_+__+_ -_+__+_ -_+__+_ -
\section{Vector-valued harmonic analysis}\label{appendix on vector-valued harmonic analysis}
% _+__+_ -_+__+_ -_+__+_ -_+__+_ -_+__+_ -_+__+_ -_+__+_ -_+__+_ -_+__+_ -_+__+_ -_+__+_ -_+__+_ -_+__+_ -

This section is a necessary step back from the main PDE line of the story into abstract, vector-valued harmonic analysis. Our goal moving forward is to take the solution operator to the reformulated linear system~\eqref{curl formulation of the linearization} granted by Theorem~\ref{thm on analysis of strong solutions, II} and prove that we can extend it from its domain of $L^2$-based Sobolev spaces to some kind of $L^r$-based Sobolev spaces for an integrability parameter $1<r<2$. The reason for doing so is potentially opaque at this point, but it is exactly this change in integrability parameter to below the threshold $2$ that makes it possible to come back to system~\eqref{linearization of the nonlinear problem} and pose it in normed spaces, rather than seminormed ones in such a way that the linkage with the nonlinear theory of Section~\ref{section on nonlinear analysis} becomes possible.

Implementing the above program requires both old and new ideas in vector-valued harmonic analysis. It is thus the goal of this section of the document to record variants of classical results in harmonic analysis adapted to the vector-valued setting relevant for this paper and to showcase our new tools in the subject, which are Theorem~\ref{thm on translation commuting linear maps}, Corollary~\ref{corc more on translation commuting linear maps}, and Theorem~\ref{second HM multiplier theorem}. We make an effort to include as many abbreviated proofs and external references as possible, striving for a concise treatment.

% _+__+_ -_+__+_ -_+__+_ -_+__+_ -_+__+_ -_+__+_ -_+__+_ -_+__+_ -_+__+_ -_+__+_ -_+__+_ -_+__+_ -_+__+_ -
\subsection{Translation commuting linear maps}\label{appendix on translation commuting linear maps}
% _+__+_ -_+__+_ -_+__+_ -_+__+_ -_+__+_ -_+__+_ -_+__+_ -_+__+_ -_+__+_ -_+__+_ -_+__+_ -_+__+_ -_+__+_ -

This section is devoted to the diagonalization, via the Fourier transform, of vector-valued translation commuting linear maps on $L^2$-based Sobolev spaces. In the finite-dimensional vector-valued case, we have the following formulation.

\begin{thm}[Translation commuting linear maps, finite dimensional case]\label{translation commuting linear maps, finite dimensional case}
    Let $V_0,V_1$ be finite dimensional complex Hilbert spaces. The following are equivalent for a bounded linear map $T:L^2(\R^d;V_0)\to L^2(\R^d;V_1)$.
    \begin{enumerate}
        \item $T$ commutes with translations in the sense that $(Tf)(\cdot+h)=T(f(\cdot+h))$ for all $f\in L^2(\R^d;V_0)$ and all $h\in\R^d$.
        \item There exists $m\in L^\infty(\R^d;\mathcal{L}(V_0,V_1))$ such that $T=m(D)$ in the sense that $\mathscr{F}^{-1}T\mathscr{F}=m$, where the right hand side is simply a multiplication operator.
    \end{enumerate}
    In either case, we have that the operator norm of $T$ coincides with the essential supremum of $m$, i.e.
    \begin{equation}
        \tnorm{T}_{\mathcal{L}(L^2V_0;L^2V_1)}=\tnorm{m}_{L^\infty\mathcal{L}(V_0;V_1)}
    \end{equation}
\end{thm}
\begin{proof}
  The proof with $V_0 = V_1 = \C$ is standard; see, for instance, Theorem 2.5.10 in Grafakos~\cite{MR3243734}.  The general finite dimensional case follows easily from this using orthonormal bases.  
\end{proof}

We require an infinite dimensional generalization of Theorem~\ref{translation commuting linear maps, finite dimensional case}.  To formulate this we first need an appropriate notion of measurable maps taking values in a space of bounded linear operators.  This can be found in Hille and Phillips~\cite{MR0423094} (Definition 3.5.4, the subsequent remark applied for $\sigma-$finite measure spaces, and Definition 3.5.5), and we record it now in the second item in the definition below.

\begin{defn}[Some notions of measurability]\label{definition of strong measurability}
Let $X$ be a complete and $\sig$-finite measure space.
\begin{enumerate}
    \item \textbf{Bochner measurability:} Let $Y$ be a Banach space. We say that a function $g:X\to Y$ is Bochner measurable if it is the almost everywhere limit of a sequence of finitely-valued measurable (simple) functions.
    \item \textbf{Operator-valued strong measurability:} Let $V_0,V_1$ be Banach spaces over $\mathbb{F} \in \{\R,\C\}$. We say that a function $f:X\to\mathcal{L}(V_0;V_1)$ is strongly measurable if for all $v\in V_0$ we have that the map $fh:X\to V_1$ is Bochner measurable.
\end{enumerate}    
\end{defn}

Next we record some essential properties of this notion of measurability.

\begin{thm}[Properties of operator-valued strongly measurable functions]\label{characterization of operator-valued strong measurability}
    Let $V_0$ and $V_1$ be separable Hilbert spaces over $\mathbb{F} \in \{\R,\C\}$ and $X$ be a complete and $\sig$-finite measure space. Let $f:X\to\mathcal{L}(V_0;V_1)$.     
    Then the following hold.
    \begin{enumerate}
        \item $f$ is strongly measurable in the sense of the second item of Definition~\ref{definition of strong measurability} if and only if for all $v_0 \in V_0$ and $v_1 \in V_1$ we have that $\tbr{fv_0,v_1}:X\to\mathbb{F}$ is measurable.
        \item If $f$ is strongly measurable and $g : X \to V_0$ is Bochner measurable in the sense of the first item of Definition~\ref{definition of strong measurability}, then $fg : X \to V_1$ is also Bochner measurable.
        \item If $f$ is strongly measurable and $g : X \to \mathcal{L}(V_0;V_1)$ is such that $f =g$ almost everywhere, then $g$ is strongly measurable.
    \end{enumerate}
\end{thm}
\begin{proof}
The first item follows from Theorem 3.5.5 in~\cite{MR0423094}.  To prove the second item, we note that Theorem 3.5.4 in~\cite{MR0423094} shows that if $\psi : X \to V_0$ is simple, then $f\psi : X \to V_1$ is Bochner measurable.   Thus, $fg$ is the almost everywhere limit of a sequence of Bochner measurable functions and is then measurable, again by Theorem 3.5.4 in~\cite{MR0423094}.  The third item follows by noting that if $h \in V_0$, then $f(x) h = g(x) h$ for almost every $x \in X$, and hence $g(\cdot)h$ is Bochner measurable.
\end{proof}

The notion of operator-valued strong measurability leads us to the following space of essentially bounded functions (see Blasco and van Neerven~\cite{blasco_vanNeerven_2010} for the generalization to $p < \infty$ and applications in the study of multiplication operators between vector-valued Lebesgue spaces).

\begin{defn}[Space of essentially bounded and strongly measurable operator-valued functions]\label{definition of Linfty}
    Let $V_0$ and $V_1$ be separable Hilbert spaces over $\mathbb{F} \in \{\R,\C\}$ and $X$ be a complete and $\sig$-finite a measure space.  We define
    \begin{equation}
        L^\infty_\ast(X;\mathcal{L}(V_0;V_1))\\=\tcb{ [f] \;|\;         
        f:X\to\mathcal{L}(V_0;V_1)\text{ is strongly measurable and essentially bounded}}
        % },\;\exists\; E\subseteq X,\;|E|=0,\;\sup_{X\setminus E}\tnorm{f}_{\mathcal{L}(V_0;V_1)}<\infty},
     \end{equation}
     where $[f]$ denotes the usual equivalence class formed via almost everywhere equality.  Note, though, that as per usual we will dispense with the equivalence class notation in what follows.   For $f\in L^\infty_\ast(X;\mathcal{L}(V_0;V_1))$ we write
    \begin{equation}\label{what we mean by Linfty}
        \tnorm{f}_{L^\infty_\ast \mathcal{L}(V_0;V_1)}= \esssup_{x \in X} \norm{f(x)}_{\mathcal{L}(V_0;V_1)}  = \inf\tcb{C\in\R^+\;:\;\tnorm{f}_{\mathcal{L}(V_0;V_1)}\le C\text{ a.e.}}.
        % =\inf_{\substack{E\subseteq X\\|E|=0}}\sup_{x\in X\setminus E}\tnorm{f(x)}_{\mathcal{L}(V_0;V_1)}.
    \end{equation}
    We emphasize that although the notion of almost everywhere is used in the definition of the essential supremum, it does not actually require the map $f$ to be measurable; consequently, the essential supremum is well-defined even on the space of strongly measurable $\mathcal{L}(V_0;V_1)$-valued maps.
\end{defn}

We can now state and prove our first infinite dimensional generalization of Theorem~\ref{translation commuting linear maps, finite dimensional case}. To the best of our knowledge, the following result does not appear in the literature.

\begin{thm}[Translation commuting linear maps, infinite dimensional case I]\label{thm on translation commuting linear maps}
    Suppose that $V_0$ and $V_1$ are separable infinite dimensional Hilbert spaces over $\C$ and that 
    \begin{equation}
        T:L^2(\R^d;V_0)\to L^2(\R^d;V_1)
    \end{equation}
    is a bounded linear map.  Then the following are equivalent.
    \begin{enumerate}
        \item $T$ commutes with translations in the sense that $(Tf)(\cdot+h)=T(f(\cdot+h))$ for all $h\in\R^d$ and $f\in L^2(\R^d;V_0)$.
        \item  There exists $m\in L^\infty_\ast(\R^d;\mathcal{L}(V_0;V_1))$  such that
    \begin{equation} \label{mult form}
        Tf=\mathscr{F}^{-1}[m\mathscr{F}[f]] \text{ for every } f\in L^2(\R^d;V_0).
    \end{equation}
    \end{enumerate}
    In either case, we have the equality
    \begin{equation}\label{operator norm estimate}
        \tnorm{m}_{L^\infty_\ast\mathcal{L}(V_0,V_1)}=\tnorm{T}_{\mathcal{L}(L^2V_0;L^2V_1)}.
    \end{equation}
\end{thm}
\begin{proof}
    If $T$ is given by \eqref{mult form}, then it is a trivial matter to verify that it commutes with translations, so we only need to prove the converse and~\eqref{operator norm estimate}.   We begin with the proof of the latter, assuming that $T=m(D)$ for $m\in L^\infty_\ast(\R^d;\mathcal{L}(V_0;V_1))$.  
    
    Let $f \in L^2(\R^d;V_0)$.  Since for any null set $E\subset\R^d$ and $\xi\in\R^d\setminus E$ we have $\tnorm{m(\xi)\mathscr{F}[f](\xi)}_{V_1}\le\sup_{\R^d\setminus E}\tnorm{m}_{\mathcal{L}(V_0;V_1)}\tnorm{\mathscr{F}[f](\xi)}_{V_0}$, we are free to integrate the square of this, take the infimum over such $E$, and apply Plancherel's theorem to deduce that $\tnorm{Tf}_{L^2 V_1}^2 =\tnorm{m f}_{L^2V_1}^2\le\tnorm{m}_{L^\infty_\ast}^2\tnorm{f}_{L^2V_0}^2$.  Thus, $\tnorm{T}\le\tnorm{m}$.
        
    For the opposite inequality, we let $\varphi_\lambda =  \mathds{1}_{B(0,\lambda)} / \sqrt{\abs{B(0,\lambda)}} \in L^2(\R^d;\R)$.  Then for any $x\in V_0$, $\xi_0 \in \R^d$, and $\lambda >0$ we have that
    \begin{equation}
    \f{1}{\abs{B(0,\lambda)}}\int_{B(\xi_0,\lambda)} \tnorm{m(\xi)x}_{V_1}^2\m{d}\xi
   =  \int_{\R^d}\tnorm{m(\xi)x}_{V_1}^2|\varphi_\lambda(\xi-\xi_0)|^2\;\m{d}\xi=\tnorm{T\mathscr{F}^{-1}(x\varphi_\lambda(\cdot - \xi_0))}_{L^2V_1}^2\le\tnorm{T}^2\tnorm{x}^2_{V_0}.
    \end{equation}
    Since $m\in L^\infty_\ast(\R^d;\mathcal{L}(V_0;V_1))$, the map $\xi\mapsto \norm{m(\xi)x}_{V_1}^2$ is locally integrable; thus, by Lebesgue's differentiation theorem, for each $x\in V_0$ we obtain a full measure set $E_x\subset\R^d$ such that $\xi_0\in E_x$ implies that $\tnorm{m(\xi_0)x}_{V_1}\le\tnorm{T}\tnorm{x}_{V_0}$. Since $V_0$ is separable, we can let $\tcb{x_n}_{n\in\N}$ be a dense subset of $V_0$ and set $E=\bigcap_{n\in\N}E_{x_n}\subseteq\R^d$. It follows that $E$ has full measure and $\xi_0\in E$ implies that $\tnorm{m(\xi_0)}_{\mathcal{L}(V_0,V_1)}\le\tnorm{T}$. Thus,  $\tnorm{m}\le\tnorm{T}$, and the proof of \eqref{operator norm estimate} is complete.
    
    We now turn to the construction of the multiplier $m$ from the map $T$, assuming it commutes with translations.  Since $V_0$ and $V_1$ are separable Hilbert spaces, we can find sequences $\tcb{\Pi^i_N}_{N=0}^\infty$ of orthogonal projection operators on $V_i$ such that for $i\in\tcb{0,1}$ we have that $\m{dim}(\Pi^i_N V_i)=N$ and $\Pi^i_N V_i\subset\Pi^i_{N+1}V_{i}$ for all $N \in \N$, and $\lim_{N\to\infty}\Pi_N^i x=$ for all $x \in V_i$.
    
    For $N \in \N$, define the  maps $S_N : L^2(\R^d; \Pi^0_N V_0) \to L^2(\R^d; \Pi^1_N V_1)$ via $S_N f= \Pi^1_N T f$.  It is a simple matter to check that each $S_N$ commutes with translations.  Since  $\Pi^0_N V_0$ and $\Pi^1_N V_1$ are finite dimensional, Theorem~\ref{translation commuting linear maps, finite dimensional case} then provides $\mu_N \in L^\infty(\R^d; \mathcal{L}(\Pi^0_N V_0; \Pi^1_N V_1))$ such that $S_N = \mu_N(D)$ and $\norm{S_N}_{\mathcal{L}} = \norm{\mu_N}_{L^\infty}$.  We then define the symbols $\tcb{m_N}_{N\in\N}\subset L^\infty(\R^{d};\mathcal{L}(V_0,V_1))$ via $m_N = \mu_N \Pi^0_N$, which means that
    \begin{equation}
        m_N(D)=S_N \Pi^0_N = \Pi^1_N T \Pi^0_N,
        \quad \tnorm{m_N}_{L^\infty\mathcal{L}(V_0,V_1)} \le \tnorm{T}, \text{ and  } \Pi^1_Nm_{N+1}\Pi^0_N=m_N.
    \end{equation}
    Note that we are free to modify each symbol in the sequence on a set of measure zero and obtain the pointwise inequality $\tnorm{m_N(\xi)}_{\mathcal{L}(V_0,V_1)}\le\tnorm{T}$ for all $\xi\in\R^d$.
    
    Given $\xi_0\in\R^d$, $x\in V_0$ , $y\in V_1$, and $N,M\in\N$, we have that
    \begin{equation}
        \tbr{y,(m_{N+M}(\xi_0)-m_{N}(\xi_0))x}=\tbr{y,m_{N+M}(\xi_0)(1-\Pi^0_N)x}+\tbr{y,(1-\Pi_N^1)m_{N+M}(\xi_0)\Pi_N^0x},
    \end{equation}
    and hence
    \begin{equation}
        \limsup_{N,M\to\infty}\tabs{\tbr{y,(m_{N+M}(\xi_0)-m_{N}(\xi_0))x}}\le\lim_{N\to\infty}\tnorm{T}\tp{\tnorm{y}_{V_1}\tnorm{(1-\Pi_N^0)x}_{V_0}+\tnorm{(1-\Pi_N)y}_{V_1}\tnorm{x}_{V_0}}=0.
    \end{equation}
    Thus, $\tcb{\tbr{y,m_{N}(\xi_0)x}}_{N\in\N}\subset\mathcal{L}(V_0,V_1)$ is Cauchy, and hence convergent.  Using this and the established bounds on $m_N$ together with Theorem VI.1 of Reed and Simon~\cite{MR0493419}, we acquire $m(\xi_0)\in\mathcal{L}(V_0,V_1)$ such that $\tnorm{m(\xi_0)}\le\tnorm{T}$ and $\tbr{y,m_N(\xi_0)x}\to\tbr{y,m(\xi_0)x}$ as $N\to\infty$ for all $x\in V_0$ and $y\in V_1$.  It then follows from Theorem~\ref{characterization of operator-valued strong measurability} that $\xi_0\mapsto m(\xi_0)$ is strongly measurable, since $\tbr{y,mx}$ is the pointwise limit of measurable functions for every $y\in V_1$ and $x\in V_0$.  Synthesizing this information, we find that $m\in L^\infty_\ast(\R^d;\mathcal{L}(V_0;V_1)).$

    To complete the proof, it only remains to check that $m(D) =T$.  For this, we  use Parseval's theorem for fixed $f\in L^2(\R^d;V_0)$ and $g\in L^2(\R^d;V_1)$ to write 
        \begin{equation}
        \tbr{g,Tf}=\int_{\Omega}\tbr{\mathscr{F}[g](\xi),m_{N}(\xi)\mathscr{F}[f](\xi)}\;\m{d}\xi+\tbr{(1-\Pi_N^1)g,T\Pi_N^0f}+\tbr{g,T(1-\Pi_N^0)f}.
    \end{equation}
    We then apply the dominated convergence theorem while sending $N \to \infty$ to deduce that $\tbr{g,Tf}=\tbr{\mathscr{F}[g],m\mathscr{F}[f]}$.   Hence, $m(D) =T$.
\end{proof}

The following generalization will also be of use to us. In fact, this  corollary is the main workhorse of Section~\ref{section on vector-valued symbol calculus for the solution map} in that it is the dictionary that allows us to transfer operator bounds derived from solving PDEs to bounds on the derivatives of the symbol of a special vector-valued Fourier multiplication operator.

\begin{coro}[Translation commuting linear maps, infinite dimensional case II]\label{corc more on translation commuting linear maps}
Suppose that $W,V_0,V_1$ are separable Hilbert spaces over $\C$ with $V_1\emb V_0$, and suppose that for some $s_0,s_1,s\in\R$ we have a translation commuting and continuous linear map
\begin{equation}\label{TILO BOUNDS}
    T:H^{s_0}(\R^d;V_0)\cap H^{s_1}(\R^d;V_1)\to H^{s}(\R^d;W).
\end{equation}
Then, there exists a unique (up to modification on sets of measure zero) locally essentially bounded and strongly measurable (in the sense of the second item of Definition~\ref{definition of strong measurability}) function $m:\R^d\to\mathcal{L}(V_1;W)$ such that $T=m(D)$; moreover, $m$ obeys the estimate
\begin{equation}\label{estimate we need on the multiplier}
    \tbr{\xi}^s\tnorm{m(\xi)x}_{W}\lesssim\tnorm{T}\tp{\tbr{\xi}^{s_0}\tnorm{x}_{V_0}+\tbr{\xi}^{s_1}\tnorm{x}_{V_1}},\quad \forall\;x\in V_1
\end{equation}
for almost every $\xi\in\R^d$. The above implicit constant only depends on $d$, $s$, $s_0$, and $s_1$.
\end{coro}
\begin{proof}
     For $\ell\in\N$ we let
     \begin{equation}
         A_\ell=\begin{cases}
             B(0,1)&\text{if }\ell=0,\\
             B(0,2^{\ell})\setminus B(0,2^{\ell-1})&\text{if }\ell\ge 1,
         \end{cases}
     \end{equation}
     $T_\ell=T\mathds{1}_{A_\ell}(D)$, and 
     \begin{equation}
         \tnorm{y}_{W^{(\ell)}}=\tbr{2^\ell}^s\tnorm{y}_W,\quad\tnorm{x}_{V_1^{(\ell)}}=\sqrt{\tbr{2^\ell}^{2s_0}\tnorm{x}_{V_0}^2+\tbr{2^\ell}^{2s_1}\tnorm{x}_{V_1}^2}
     \end{equation}
     for $y\in W$ and $x\in V_1$. We consider $W$ equipped with the norm $\tnorm{\cdot}_{W^{(\ell)}}$, denoted $W^{(\ell)}$, and $V_1$ equipped with the norm $\tnorm{\cdot}_{V_1^{(\ell)}}$, similarly denoted $V_1^{(\ell)}$, and apply Theorem~\eqref{thm on translation commuting linear maps} to $T_\ell$, viewed  as a map $T_\ell:L^2(\R^d;W^{(\ell)}_1)\to L^2(\R^d;W^{(\ell)})$. The  hypothesis~\eqref{TILO BOUNDS} and the usual Fourier characterization of $H^r$ Sobolev norms then provide the estimate
     \begin{equation}
         \tnorm{T_\ell}_{\mathcal{L}(L^2V_1^{(\ell)};L^2W^{(\ell)})}\lesssim\tnorm{T}_{\mathcal{L}(H^{s_0}V_0\cap H^{s_1}V_1;H^s W)},
     \end{equation}
     where the implicit constant depends on $d$, $s$, $s_0$, and $s_1$ but not on $\ell$.
     Then the associated multiplier $m_\ell\in L_\ast^\infty(\R^d;\mathcal{L}(V_1^{(\ell)},W^{(\ell)}))$ granted from Theorem~\ref{thm on translation commuting linear maps} obeys the bounds
     \begin{equation}\label{annular bounds}
         \tbr{2^\ell}^s\tnorm{m_\ell(\xi)x}_{W}\lesssim\tnorm{T}\tp{\tbr{2^\ell}^{s_0}\tnorm{x}_{H^{s_0}}+\tbr{2^\ell}^{s_1}\tnorm{x}_{V_1}} 
     \end{equation}
     for every $x\in V_1$ and  almost every $\xi\in A_\ell$. Note that $m_\ell$ can be modified on a set of measure zero and made to have support contained in $A_\ell$. To conclude, we take $m=\sum_{\ell=0}^\infty m_\ell$. It is then straightforward to check that $m(D)=T$ and, by using~\eqref{annular bounds}, that estimate~\eqref{estimate we need on the multiplier} holds. 
\end{proof}

% _+__+_ -_+__+_ -_+__+_ -_+__+_ -_+__+_ -_+__+_ -_+__+_ -_+__+_ -_+__+_ -_+__+_ -_+__+_ -_+__+_ -_+__+_ -
\subsection{Classical results in vector-valued Harmonic analysis}\label{section on classical results in vector-valued Harmonic analysis}
% _+__+_ -_+__+_ -_+__+_ -_+__+_ -_+__+_ -_+__+_ -_+__+_ -_+__+_ -_+__+_ -_+__+_ -_+__+_ -_+__+_ -_+__+_ -

We now turn our attention to a collection of classical results in vector-valued harmonic analysis.  We showcase these here for two reasons.  First, in the subsequent subsection we will develop some variants and generalizations that will play a crucial role in our study of the linear PDEs~\eqref{linearization of the nonlinear problem} and~\eqref{curl formulation of the linearization}.  Second, we will need them  in our development of some Sobolev-type function spaces in Section~\ref{appendix on properties of mixed-type Sobolev spaces}. 

We begin by recording a pair of well-known multiplier theorems. The first up is the scalar-valued Marcinkiewicz theorem, for which a proof can be found in Corollary 6.25 of Grafakos~\cite{MR3243734}.

\begin{thm}[Marcinkiewicz]\label{thm on Marcinkiewicz multiplier theorem}
    Let $m : \R^d \to \C$ be a bounded function that is $d$-times continuously differentiable away from the coordinate axes in $\R^d$.  Assume that there exists a constant $A\ge 0$ such that for all $k\in\tcb{1,\dots,d}$, each choice of distinct $j_1,\dots,j_k\in\tcb{1,\dots,d}$, and every $\xi\in \R^d$ such that $\xi_r \neq 0$ for $r \notin \tcb{j_1,\dotsc,j_k}$ we have that
    \begin{equation}\label{Marcinkiewicz ineqs}
        |(\pd_{j_1}\cdots\pd_{j_k}m)(\xi)| \le A |\xi_{j_1}|^{-1}\cdots|\xi_{j_k}|^{-1}.
    \end{equation}
    Then the map $m(D) = L^2(\R^d;\C) \to L^2(\R^d;\C)$ uniquely extends to a bounded linear map  $m(D) : L^p(\R^d;\C) \to L^p(\R^d;\C)$ for every $1 < p < \infty$, and 
    \begin{equation}
        \norm{m(D)}_{\mathcal{L}(L^p)} \le C_{p,d}(A + \norm{m}_{L^\infty})
    \end{equation}
    for a constant $C_{p,d} >0$ depending only on $d$ and $p$. If, in addition, we have that $m(-\xi)=\Bar{m(\xi)}$ for a.e. $\xi\in\R^d$, then $m$ is reality preserving in the sense that $m(D):L^p(\R^d;\R)\to L^p(\R^d;\R)$.
\end{thm}

We next record a vector-valued version of the celebrated Mikhlin-H\"ormander multiplier theorem, originally due to  Schwartz \cite{schwartz_1961}.  For a proof of the following formulation we refer to Proposition 6.16 in Bergh and L\"ofstr\"om~\cite{MR0482275}.

\begin{thm}[Mikhlin-H\"ormander]\label{hormander mikhlin multiplier theorem}
    Let $V_0$ and $V_1$ be two separable complex Hilbert spaces and let $\N\ni\ell>d/2$. Suppose that $m\in C^\ell(\R^d\setminus\tcb{0};\mathcal{L}(V_0,V_1))$ satisfies
    \begin{equation}
        \max_{|\al| \le \ell } \sup_{\xi\neq 0}|\xi|^{|\al|}\tnorm{\pd^\al m(\xi)}_{\mathcal{L}(H_0,H_1)}\le C_\ell
    \end{equation}
    for a constant $C_\ell\in\R^+$.  Then the map $m(D) : L^2(\R^d;V_0)\to L^2(\R^d;V_1)$ uniquely extends to a bounded linear map $m(D):L^p(\R^d;V_0)\to L^p(\R^d;V_1)$, and 
    \begin{equation}
        \tnorm{m(D)}_{\mathcal{L}(L^p V_0,L^pV_1)}\lesssim C_\ell,
    \end{equation}
     where the implicit constant depends only on $d$ and $p$.
\end{thm}

 We now record an important maximal inequality due to Fefferman and Stein~\cite{MR284802} (see also Theorem 1 in Chapter 2 of Stein~\cite{MR1232192}).  In what follows $\mathcal{M}$ denotes the usual Hardy-Littlewood maximal function.

\begin{thm}[Fefferman-Stein maximal inequality]\label{thm on fefferman stein maximal inequality}
    Suppose that $\tcb{f_\ell}_{\ell\in\Z}\subseteq L^1_\m{loc}(\R^d)$.  Then for $1<p<\infty$ we have the inequality 
    \begin{equation}
        \bnorm{\bp{\sum_{\ell\in\Z}|\mathcal{M}(f_\ell)|^2}^{1/2}}_{L^p}\lesssim\bnorm{\bp{\sum_{\ell\in\Z}|f_\ell|^2}^{1/2}}_{L^p}.
    \end{equation}
\end{thm}

Next, we record some vector-valued adaptations of well-known estimates from Littlewood-Paley theory. The first is based on Theorem 6.1.2 and Proposition 6.1.4 of Grafakos~\cite{MR3243734}.

\begin{thm}[Annular Littlewood-Paley, I]\label{thm on annular littlewood paley, I}
    Let $V$ be a separable complex Hilbert space, and let $s\in[0,\infty)$, $1<p<\infty$, and $m\in\N$.  Suppose that $\tcb{\varphi_j}_{j\in\Z}\in C^\infty_c(\R^d)$ are such that for all $j\in\Z$ we have that $\mathscr{F}[\varphi_j]$ is supported in the annulus $B(0,2^{m+j})\setminus\Bar{B(0,2^{-m+j})}$ and there exists a constants $\tcb{C_\al}_{\al \in \N^d}$ such that
    \begin{equation}
        |\pd^\al\varphi_j(\xi)|\le C_\al2^{-j|\al|} \text{ for all }\xi\in\R^d 
        \text{ and } \alpha \in \N^d.
    \end{equation}
    Then  for every $f\in H^{s,p}(\R^d;E)$ we have the inequality
        \begin{equation}
            \bnorm{\bp{\sum_{j\in\Z}\tbr{2^j}^{2s}\tnorm{\varphi_j(D)f}_V^2}^{1/2}}_{L^p}\lesssim\tnorm{f}_{H^{s,p}V},
        \end{equation}
        with implicit constants depending only on $\varphi$, $s$, $p$, $d$, $m$, and finitely many of the $\tcb{C_\al}_{\al}$.
\end{thm}
\begin{proof}
    It suffices to prove that for $f\in L^p(\R^d;V)$ we have the inequality
    \begin{equation}
        \bnorm{\bp{\sum_{j\in\Z}\tbr{2^j}^{2s}\tnorm{\varphi_j(D)\tbr{D}^{-s}f}_V^2}^{1/2}}_{L^p}\lesssim\tnorm{f}_{L^{p}V}.
    \end{equation}
    But this is a direct application of the Mikhlin-H\"ormander multiplier theorem,  Theorem~\ref{hormander mikhlin multiplier theorem}, with the smooth multiplier $m:\R^d\to\mathcal{L}(V;\ell^2(\Z;V))$ given by 
    \begin{equation}
         m(\xi)x=\tcb{\tbr{2^j}^s\varphi_j(\xi)x/\tbr{\xi}^s}_{j\in\Z} \text{ for } \xi\in\R^d \text{ and } x\in V.
    \end{equation} 
\end{proof}

Our next two results are loosely based on Lemmas 2.1.F and  2.1.G from Taylor~\cite{MR1121019}.

\begin{thm}[Annular Littlewood-Paley, II]\label{thm on annular littlewood paley, II}
    Let $V$ be a separable complex Hilbert space, and fix $s\in[0,\infty)$, $1<p<\infty$, and $m\in\N^+$. For $j\in\Z$ let $A_j=B(0,2^{j+m})\setminus\Bar{B(0,2^{j-m})}$. Suppose that $\tcb{f_j}_{j\in\Z}\subset L^p(\R^d;V)$ satisfy $\supp\mathscr{F}[f_j]\subset A_j$ for every $j\in\Z$. Then
    \begin{equation}\label{littlewood paley estimate 2}
        \sup_{ \substack{F \subset \Z \\ F \text{ finite}}  } \bnorm{\sum_{j \in F} f_j}_{H^{s,p}V}\lesssim\bnorm{\bp{\sum_{j\in\Z}\tbr{2^j}^{2s}\tnorm{f_j}_V^2}^{1/2}}_{L^p},
    \end{equation}
    with the implicit constant depending only on $d$, $s$, $p$, and $m$. Moreover, if the right hand side is finite, then the series $\sum_{j \in \Z} f_j$ converges unconditionally in $H^{s,p}(\R^d;V)$.
\end{thm}
\begin{proof}
    Let $1 < q < \infty$ satisfy $1/p + 1/q =1$, and let $\varphi\in C^\infty_c(A_{-1}\cup A_0\cup A_1)$ be such that $\varphi=1$ on $\Bar{A_0}$. Given $g\in L^{q}(\R^d;V)$ and a finite set $F \subset \Z$, we compute
    \begin{equation}
        \int_{\R^d}\bbr{\tbr{D}^s\sum_{ j \in F}f_j,g}=\sum_{j \in F}\int_{\R^d}\sbr{\tbr{2^j}^sf_j,\tilde{\varphi}_j(D)g}
    \end{equation}
    where $\tilde{\varphi}_j(\xi)=\tbr{\xi}^s\tbr{2^j}^{-s}\varphi(\xi/2^j)$. Hence, we can apply Cauchy-Schwartz, H\"older, Theorem~\ref{thm on annular littlewood paley, I}, and duality (See, e.g., Theorem 5 of Section 4 in Chapter 12 of Dinculeanu~\cite{MR0206190}) to acquire the bound
    \begin{equation}
        \bnorm{\sum_{j \in F} f_j}_{H^{s,p}V}
        =\sup_{\tnorm{g}_{L^{q}V}\le 1}\babs{\int_{\R^d}\bbr{\tbr{D}^s\sum_{j \in F}f_j,g}}
        \lesssim\bnorm{\bp{\sum_{j \in \Z}\tbr{2^j}^{2s}\tnorm{f_j}_V^2}^{1/2}}_{L^p}.
    \end{equation}
    Estimate~\eqref{littlewood paley estimate 2} follows.  A similar strategy shows that if the right hand side of \eqref{littlewood paley estimate 2} is finite, and $\{F_n\}_{n \in \N}$ is any increasing sequence of finite subsets of $\Z$ such that $\Z = \bigcup_{n \in \N} F_n$, then $\{\sum_{j \in F_n} f_j \}_{n \in \N}$ is Cauchy (and hence convergent) in $H^{s,p}(\R^d;V)$, and the limit is independent of the choice of the sequence.  This implies the unconditional convergence of the series.  We omit further details for the sake of brevity.
\end{proof}

Finally, we record a non-annular Littlewood-Paley estimate.  Note that in this result it is important that we are considering spaces of positive regularity.

\begin{thm}[Ball Littlewood-Paley]\label{thm on ball littlewood paley}
    Let $V$ be a separable complex Hilbert space, and fix $s\in\R^+$, $1<p<\infty$, and $m\in\N^+$. For $j\in\N$ set $B_j=B(0,2^{j+m})$. Suppose that $\tcb{f_j}_{j\in\N}\subset L^p(\R^d;V)$ satisfy $\supp\mathscr{F}[f_j]\subset B_j$ for every $j\in\N$. Then
    \begin{equation}\label{littlewood paley estimate 3}
        \sup_{ \substack{F \subset \N \\ F \text{ finite}}  } \bnorm{\sum_{j\in F} f_j}_{H^{s,p}V}\lesssim\bnorm{\bp{\sum_{j\in\N}4^{js}\tnorm{f_j}_V^2}^{1/2}}_{L^p},
    \end{equation}
    with the implicit constant depending only on $d$, $s$, $p$, and $m$. Moreover, if the right hand side is finite, then the series $\sum_{j \in \N} f_j$ converges unconditionally in $H^{s,p}(\R^d;V)$.
\end{thm}
\begin{proof}
     Let $\tcb{\varphi_k}_{k=0}^\infty$ be an inhomogeneous Littlewood-Paley partition of unity with $\sum_{k=0}^\infty\varphi_k=1$, $\varphi_k=\varphi_1(\cdot/2^k)$ for $k\ge 1$, $\supp\varphi_0\subseteq B(0,2)$, $\supp\varphi_k\subseteq B(0,2^{k+2})\setminus\Bar{B(0,2^{k-2})}$.  We also write $\tilde{\varphi}_k=2^{-ks}\tbr{\cdot}^s\varphi_k$. 

    Let $1 < q < \infty$ satisfy $1/p + 1/q =1$, and let $g\in L^{q}(\R^d;V)$.  Then for  any finite $F \subset \N$ we have that
    \begin{equation}
        \int_{\R^d}\bbr{\tbr{D}^s\sum_{j\in F} f_j,g} 
        = \sum_{j\in F} \sum_{k=0}^{j+3+m}\int_{\R^d}\tbr{\tbr{D}^sf_j,\varphi_k(D)g}
        =\sum_{j \in F} \sum_{k=0}^{j+3+m}\int_{\R^d}\tbr{2^{js}f_j,2^{(k-j)s}\tilde{\varphi}_k(D)g}.
    \end{equation}
    Hence, we may use Cauchy-Schwarz, Young's convolution inequality, and H\"older's inequality to bound
    \begin{multline}\label{applies just as well}
        \babs{\int_{\R^d}\bbr{\tbr{D}^s\sum_{j\in F}f_j,g}}
        \le\int_{\R^d}\bp{\sum_{j\in F} 4^{js}\tnorm{f_j}_V^2}^{1/2}\bp{\sum_{j\in F}\bp{\sum_{k=0}^{j+3+m}2^{s(k-j)}\tnorm{\tilde{\varphi}_k(D)g}_V}^2}^{1/2}\\
        \lesssim\bnorm{\bp{\sum_{j=0}^\infty 4^{js}\tnorm{f_j}_V^2}^{1/2}}_{L^p}\bnorm{\bp{\sum_{k=0}^\infty\tnorm{\tilde{\varphi}_k(D)g}_V^2}^{1/2}}_{L^{q}}.
    \end{multline}
    Theorem~\ref{thm on annular littlewood paley, I} provides the bound
    \begin{equation}
        \bnorm{\bp{\sum_{k=0}^\infty\tnorm{\tilde{\varphi}_k(D)g}_V^2}^{1/2}}_{L^{q}}\lesssim\tnorm{g}_{L^{q}V}.
    \end{equation}
    Taking the supremum over $g$ such that $\tnorm{g}_{L^{q}V}\le 1$ in~\eqref{applies just as well} then gives the desired estimate.  The unconditional convergence of the series then follows from a variant of this bound as in the proof of Theorem \ref{thm on annular littlewood paley, II}. 
\end{proof}

% _+__+_ -_+__+_ -_+__+_ -_+__+_ -_+__+_ -_+__+_ -_+__+_ -_+__+_ -_+__+_ -_+__+_ -_+__+_ -_+__+_ -_+__+_ -
\subsection{On a novel variant of the Mikhlin-H\"ormander multiplier theorem}\label{section on the new MH multiplier theorem}
% _+__+_ -_+__+_ -_+__+_ -_+__+_ -_+__+_ -_+__+_ -_+__+_ -_+__+_ -_+__+_ -_+__+_ -_+__+_ -_+__+_ -_+__+_ -

We now return to the topic of multipliers, with the aim of deriving a generalized version of Mikhlin-H\"ormander (the classical vector-valued version is Theorem~\ref{hormander mikhlin multiplier theorem}).  The task of generalizing Mikhlin-H\"ormander is, of course, not new, and there are many works in the existing literature that do so.  We pause here briefly to review these.  In unpublished work, Pisier showed that if a vector-valued Mikhlin-H\"ormander theorem holds for an $\mathcal{L}(E)$-valued symbol, then $E$ is isomorphic to a Hilbert space.  Note that Lancien, Lancien, and Le Merdy \cite{lancien_lancien_lemerdy_1998} provide a proof of Pisier's unpublished result in Remark 6.4, as a consequence of another result.  The takeaway is that there is an obstruction to versions of Mikhlin-H\"ormander for maps valued in $\mathcal{L}(B_0;B_1)$ if the $B_i$ spaces are only Banach.  The work of Bourgain \cite{bourgain_1983}, Burkholder \cite{burkholder_1981}, McConnell \cite{mcconnell_1984},  Zimmermann \cite{zimmermann_1989} found a workaround for scalar-valued symbol multipliers if the Banach space is an unconditional martingale difference (UMD) space and the Mikhlin-H\"ormander hypotheses  are strengthened a bit.  However, it is known that UMD spaces are reflexive.  The scalar UMD extension was later strengthened to operator-valued symbols  by Amann \cite{amann_1997}, Hieber \cite{hieber_1999},  Haller, Heck, and Noll \cite{haller_heck_noll_2002},  and  Giardi and Weis \cite{girardi_weis_2003_lp}.  We also refer also to Chapter 4 of Pr\"uss and Simonett \cite{pruss_simonett_2016} and the survey of Giardi and Weis \cite{girardi_weis_2003_survey} for more information.  

The Banach obstruction can also be overcome by changing from $L^p$ spaces to others.   Amann \cite{amann_1997} developed a version of Mikhlin-H\"ormander in the context of vector-valued Besov spaces, which was subsequently extended by Giardi and Weis \cite{girardi_weis_2003_besov}.  The setting of Triebel-Lizorkin spaces was considered by Bu and Kim \cite{bu_kim_2005}.

For our purposes in this paper, we can restrict to multipliers that take values in $\mathcal{L}(V;W)$, where $V$ and $W$ are separable Hilbert spaces.  However, we need a version of Mikhlin-H\"ormander  that allows us to replace a single $L^p$ space with the more general setting of $H^{s_0,p}(\R^d;V_0)\cap H^{s_1,p}(\R^d;V_1)$, the intersection of different Bessel potential spaces with values in different Hilbert spaces such that $V_1 \emb V_0$.  We will prove this new variant by combining  three main ingredients. The first is the annular Littlewood-Paley results of Theorems~\ref{thm on annular littlewood paley, I} and~\ref{thm on annular littlewood paley, II}, while the second is the Fefferman-Stein maximal inequality of Theorem~\ref{thm on fefferman stein maximal inequality}.  The final ingredient we need, which we record in the next result, gives pointwise bounds via maximal functions.  It is a generalization of Lemma 2.2 in Bahouri, Chemin, and Danchin~\cite{MR2768550}.

\begin{thm}[Pointwise bounds for spectrally localized Fourier multipliers]\label{thm on pointwise bounds for spectrally localized Fourier multipliers}
    Suppose $V$, $V_0$, and $V_1$ are separable complex Hilbert spaces such that $V_1\emb V_0$.  Let $s,s_0,s_1\in\R$ and $\mu, \ell \in \Z$ with $\mu \le 0$ and $d < \ell$.  Suppose that $m\in C^\ell(\R^d\setminus\tcb{0};\mathcal{L}(V_1,V))$ satisfies  
    \begin{equation}\label{the bounds on the multilpier}
        \tbr{\xi}^s  \max_{|\alpha|\le \ell}  |\xi|^{|\al|}\tnorm{\pd^\al m(\xi)x}_V\le C_\ell|\xi|^\mu\tp{\tbr{\xi}^{s_0}\tnorm{x}_{V_0}+\tbr{\xi}^{s_1}\tnorm{x}_{V_1}}
    \end{equation}
    for all $\xi\in\R^d\setminus\tcb{0}$ and $x \in V_1$. Let $\varphi,\tilde{\varphi}\in C^\infty_c(B(0,16)\setminus\Bar{B(0,1/16)})$ be such that $\tilde{\varphi}=1$ on the support of $\varphi$. Then for $f\in\mathscr{S}(\R^d;V_1)$, $\lambda\in\R^+$, and $z\in\R^d$ we have the pointwise estimate
    \begin{equation}
        \tbr{\lambda}^s\tnorm{m(D)\varphi(D/\lambda)f(z)}_{V}
        \lesssim C_\ell|\lambda|^\mu \left[\tbr{\lambda}^{s_0}\mathcal{M}(\tnorm{\tilde{\varphi}(D/\lambda)f}_{V_0})(z)+\tbr{\lambda}^{s_1}\mathcal{M}(\tnorm{\tilde{\varphi}(D/\lambda)f}_{V_1})(z) \right].
    \end{equation}
\end{thm}
\begin{proof}
    We let $K_\lambda:\R^d\to\mathcal{L}(V_1;V)$ be defined via
    \begin{equation}
        K_\lambda(z)=\int_{\R^d}e^{2\pi\ii z\cdot\xi}\varphi(\xi)m(\lambda\xi)\;\m{d}\xi,
    \end{equation}
    and note that 
    \begin{equation}\label{the convolution formula}
        m(D)\varphi(D/\lambda)f=\lambda^dK_\lambda(\lambda(\cdot))\ast(\tilde{\varphi}(D/\lambda)f)
    \end{equation}
    for all $f\in\mathscr{S}(\R^d;V_1)$.    

    Fix $x \in V_1$.   Hypothesis~\eqref{the bounds on the multilpier} implies the rescaled form
    \begin{equation}
        \tnorm{\pd^\be(m(\lambda\cdot))(\xi)x}_{V}\le C_\ell|\lambda|^\mu\tbr{\lambda\xi}^{-s}|\xi|^{\mu-|\be|}\tp{\tbr{\lambda\xi}^{s_0}\tnorm{x}_{V_0}+\tbr{\lambda\xi}^{s_1}\tnorm{x}_{V_1}}
    \end{equation}
    for every $|\beta|\le \ell$.   This, integration by parts, and an application of the Leibniz rule  then provide the following kernel bound for $|\alpha|\le \ell$: 
    \begin{equation}\label{the kernel bounds}
        \tnorm{(2\pi\ii z)^\al K_\lambda(z)x}_{V}=\int_{\R^d}\tnorm{\pd^\al(\varphi m(\lambda(\cdot)))(\xi)x}^2_V\;\m{d}\xi 
        \lesssim C_\ell\tbr{\lambda}^{-s}|\lambda|^\mu\tp{\tbr{\lambda}^{s_0}\tnorm{x}_{V_0}+\tbr{\lambda}^{s_1}\tnorm{x}_{V_1}}.
    \end{equation}
    Summing over $|\al|\le\ell$, we deduce from this that 
    \begin{equation}
        \sup_{z \in \R^d} \tbr{z}^\ell \tnorm{K_\lambda(z)x}_{V} \lesssim C_\ell\tbr{\lambda}^{-s}|\lambda|^\mu\tp{\tbr{\lambda}^{s_0}\tnorm{x}_{V_0}+\tbr{\lambda}^{s_1}\tnorm{x}_{V_1}}.
    \end{equation}

    Now we return to formula~\eqref{the convolution formula} to obtain the claimed pointwise bounds. Let  $f\in\mathscr{S}(\R^d;V_1)$, and set $g=\tilde{\varphi}(D/\lambda)f$.   Write
    \begin{equation}
        A_j(y,\lambda)=
        \begin{cases}
            B(y,1/\lambda)&\text{if }j=0,\\
            B(y,2^j/\lambda)\setminus B(y,2^{j-1}/\lambda)&\text{if }j\ge 1.
        \end{cases}
    \end{equation}
    Then the above estimates allow us to bound
    \begin{multline}
        \tnorm{(\lambda^dK_\lambda(\lambda\cdot)\ast g)(y)}_{V}
        \le\sum_{j=0}^\infty\int_{A_j(y,\lambda)}\lambda^d\tnorm{K_\lambda(\lambda(y-z))g(z)}_{V}\;\m{d}z\\
        \lesssim C_\ell\sum_{j=0}^\infty\int_{A_j(y,\lambda)}\f{\lambda^d}{\tbr{\lambda(y-z)}^\ell}\tbr{\lambda}^{-s}|\lambda|^\mu\tp{\tbr{\lambda}^{s_0}\tnorm{g(z)}_{V_0}+\tbr{\lambda}^{s_1}\tnorm{g(z)}_{V_1}}\;\m{d}z\\
        \lesssim C_\ell\tbr{\lambda}^{-s}|\lambda|^\mu\sum_{j=0}^\infty \tbr{2^j}^{d-\ell}\bp{\f{1}{|B(y,2^j/\lambda)|}\int_{B(y,2^j/\lambda)}\tp{\tbr{\lambda}^{s_0}\tnorm{g(z)}_{V_0}+\tbr{\lambda}^{s_1}\tnorm{g(z)}_{V_1}}\;\m{d}z}\\
        \lesssim C_\ell\tbr{\lambda}^{-s}|\lambda|^\mu \left[\tbr{\lambda}^{s_0}\mathcal{M}(\tnorm{g}_{V_0})(y)+\tbr{\lambda}^{s_1}\mathcal{M}(\tnorm{g}_{V_1})(y) \right].
    \end{multline}
    This yields the desired bound upon substituting in $g=\tilde{\varphi}(D/\lambda)f$.
\end{proof}

We now have all the tools needed to prove our generalization of the Mikhlin-H\"ormander multiplier theorem.

\begin{thm}[Mikhlin-H\"ormander, novel form]\label{second HM multiplier theorem}
Suppose $V$, $V_0$, and $V_1$ are separable complex Hilbert spaces such that $V_1\emb V_0$.  Let $s,s_0,s_1\in\R$ and $\mu, \ell \in \Z$ with $\mu \le 0$ and $d < \ell$.  Suppose that $m\in C^\ell(\R^d\setminus\tcb{0};\mathcal{L}(V_1;V))$ satisfies  
    \begin{equation}\label{ziggy_returns}
        \tbr{\xi}^s  \max_{|\beta|\le \ell}  |\xi|^{|\beta|}\tnorm{\pd^\beta m(\xi)x}_V\le C_\ell|\xi|^\mu\tp{\tbr{\xi}^{s_0}\tnorm{x}_{V_0}+\tbr{\xi}^{s_1}\tnorm{x}_{V_1}}
    \end{equation}
    for all $\xi\in\R^d\setminus\tcb{0}$ and $x \in V_1$.  Let  $\al\in \N^d$ satisfy $|\al|=-\mu$. Then the bounded linear map
        \begin{equation}
            m(D)D^\al:H^{s_0}(\R^d;V_0)\cap H^{s_1}(\R^d;V_1)\to H^{s}(\R^d;V),
        \end{equation}
        which is well-defined and bounded in light of the $\beta=0$ estimate from \eqref{ziggy_returns},  uniquely extends to a bounded linear map 
        \begin{equation}
            m(D)D^\al:H^{s_0,p}(\R^d;V_0)\cap H^{s_1,p}(\R^d;V_1)\to H^{s,p}(\R^d;V)
        \end{equation}
        for every $1<p<\infty$.
\end{thm}
\begin{proof}
    Let $\varphi,\tilde{\varphi},\tilde{\tilde{\varphi}}\in C^\infty_c(\R^d)$ be such that $\supp\varphi\subset B(0,8)\setminus \Bar{B(0,1/8)}$, $\supp\tilde{\varphi}, \supp\tilde{\tilde{\varphi}}\subset B(0,16)\setminus\Bar{B(0,1/16)}$, $\sum_{j\in\Z}\varphi(\cdot/2^j)=\mathds{1}_{\R^d\setminus\tcb{0}}$, $\tilde{\varphi}=1$ on $\supp\varphi$, and $\tilde{\tilde{\varphi}}=1$ on $\supp\tilde{\varphi}$.
    
    Suppose that $f\in\mathscr{S}(\R^d;V_1)$. We first use the Littlewood-Paley estimate of Theorem~\ref{thm on annular littlewood paley, II} to bound
    \begin{equation}\label{the first step}
        \tnorm{m(D)D^\al f}_{H^{s,p}V}\lesssim\bnorm{\bp{\sum_{j\in\Z}\tbr{2^j}^{2s}\tnorm{m(D)D^\al\varphi(D/2^j)f}^2_{V}}^{1/2}}_{L^p}.
    \end{equation}
    Next, for each $j\in\Z$ we apply Theorem~\ref{thm on pointwise bounds for spectrally localized Fourier multipliers} with $\lambda=2^j$ to see that
    \begin{equation}\label{combine 1}
        \tbr{2^j}^s\tnorm{m(D)D^\al\varphi(D/2^j)f}_V\lesssim C_\ell2^{-j|\al|}\tp{\tbr{2^j}^{s_0}\mathcal{M}(\tnorm{D^\al\tilde{\varphi}(D/2^j)f}_{V_0})+\tbr{2^j}^{s_1}\mathcal{M}(\tnorm{D^\al\tilde{\varphi}(D/2^j)f}_{V_1})}.
    \end{equation}
    For $i\in\tcb{0,1}$, we may apply Theorem~\ref{thm on pointwise bounds for spectrally localized Fourier multipliers} again (using the trivial multiplier $m=1$ and parameters $s = s_0 = s_1 =0$ with $V_i$ used for all three Hilbert spaces), to acquire the bound
    \begin{equation}\label{combine 2}
        \tnorm{D^\al\tilde{\varphi}(D/2^j)f}_{V_i}\lesssim2^{j|\al|}\mathcal{M}(\tnorm{\tilde{\tilde{\varphi}}(D/2^j)f}_{V_i}).
    \end{equation}
    We then combine the estimates~\eqref{combine 1} and~\eqref{combine 2}:
    \begin{equation}
        \f{1}{C_\ell}\tbr{2^j}^s\tnorm{m(D)D^\al\varphi(D/2^j)f}_V
        \lesssim
        \tbr{2^j}^{s_0}\mathcal{M}\mathcal{M}(\tnorm{\tilde{\tilde{\varphi}}(D/2^j)f}_{V_0})+\tbr{2^j}^{s_1}\mathcal{M}\mathcal{M}(\tnorm{\tilde{\tilde{\varphi}}(D/2^j)f}_{V_1}).
    \end{equation}
    From this, estimate~\eqref{the first step},  two applications of  the Fefferman-Stein maximal inequality of Theorem~\ref{thm on fefferman stein maximal inequality}, and an application of the Littlewood-Paley estimate of Theorem~\ref{thm on annular littlewood paley, I} we  deduce that 
    \begin{multline}
        \f{1}{C_\ell} \tnorm{m(D)D^\al f}_{H^{s,p}V} \\ \lesssim 
        \bnorm{ \bp{ \sum_{j \in \Z}\tbr{2^j}^{2s_0}[\mathcal{M}\mathcal{M}(\tnorm{\tilde{\tilde{\varphi}}(D/2^j)f}_{V_0})]^2 }^{1/2} }_{L^p}
        +
        \bnorm{ \bp{ \sum_{j \in \Z}\tbr{2^j}^{2s_1}[\mathcal{M}\mathcal{M}(\tnorm{\tilde{\tilde{\varphi}}(D/2^j)f}_{V_1})]^2 }^{1/2} }_{L^p} \\
        \lesssim 
        \bnorm{ \bp{ \sum_{j \in \Z}\tbr{2^j}^{2s_0}\tnorm{\tilde{\tilde{\varphi}}(D/2^j)f}_{V_0}^2 }^{1/2} }_{L^p}
        +
        \bnorm{ \bp{ \sum_{j \in \Z}\tbr{2^j}^{2s_1}\tnorm{\tilde{\tilde{\varphi}}(D/2^j)f}_{V_1}^2 }^{1/2} }_{L^p} \\
        \lesssim
         \tnorm{f}_{H^{s_0,p}V_0} + \tnorm{f}_{H^{s_1,p}V_1}.
    \end{multline}
    This estimate then allows us to extend $m(D) D^\al$ as stated.  
\end{proof}

% _+__+_ -_+__+_ -_+__+_ -_+__+_ -_+__+_ -_+__+_ -_+__+_ -_+__+_ -_+__+_ -_+__+_ -_+__+_ -_+__+_ -_+__+_ -
\section{Vector-valued symbol calculus for the solution map}\label{section on vector-valued symbol calculus for the solution map}
% _+__+_ -_+__+_ -_+__+_ -_+__+_ -_+__+_ -_+__+_ -_+__+_ -_+__+_ -_+__+_ -_+__+_ -_+__+_ -_+__+_ -_+__+_ -

The goal of this section is to prove that the solution operator associated to the linear system~\eqref{curl formulation of the linearization} is given by an operator-valued Fourier multiplier.  Once this is established, we then show that the symbol of the operator is a smooth function of frequency away from zero and satisfies bounds of the type appearing in the hypotheses of the Mikhlin-H\"ormander multiplier theorem (see Theorems~\ref{hormander mikhlin multiplier theorem} and~\ref{second HM multiplier theorem}).

% _+__+_ -_+__+_ -_+__+_ -_+__+_ -_+__+_ -_+__+_ -_+__+_ -_+__+_ -_+__+_ -_+__+_ -_+__+_ -_+__+_ -_+__+_ -
\subsection{Preliminaries}
% _+__+_ -_+__+_ -_+__+_ -_+__+_ -_+__+_ -_+__+_ -_+__+_ -_+__+_ -_+__+_ -_+__+_ -_+__+_ -_+__+_ -_+__+_ -

The following definition enumerates the translation commuting maps we are interested in extending to an $L^p$ theory. That the following is well-defined is a consequence of Theorem~\ref{thm on analysis of strong solutions, II}.

\begin{defn}[Spaces and translation commuting maps]\label{defn of spaces and TILOs}
We set the following notation.
\begin{enumerate}
    \item For $s\in\N$ we define the spaces
    \begin{equation}
        \X_s=H^{1+s}(\Omega;\C)\times H^{2+s}(\Omega;\C^3)\times H^{3/2+s}(\Sigma;\C^2),
    \end{equation}
    \begin{equation}
        \Y_s=H^s(\Omega;\C^3)\times H^{1/2+s}(\Sigma;\C^3)\times H^{5/2+s}(\Sigma;\C^2),
    \end{equation}
    \begin{equation}
        \quad\tilde{\Y}_s=\Big\{(g,f,k,h,\omega)\text{ as in~\eqref{this space is too long dear liza}}\;:\;\omega,\;h-\int_0^bg(\cdot,y)\;\m{d}y\in\dot{H}^{-1}(\Sigma;\C)\Big\}.
    \end{equation}

    \item We define the bounded linear map $\Phi:\tilde{\Y}_s\to\X_s$ via $\Phi(g,f,k,h,\omega)=(p,u,\chi)$, where the latter tuple is the unique solution to~\eqref{curl formulation of the linearization} with data $(g,f,k,h,\omega)$ provided by Theorem~\ref{thm on analysis of strong solutions, II}.  We also define the bounded linear map  $\Psi:\Y_s\to\X_s$ via  $\Psi(f,k,H)=\Phi(0,f,k,\grad_{\|}\cdot H,0)$.
\end{enumerate}
\end{defn}

It is clear from the existence and uniqueness result of Theorem~\ref{thm on analysis of strong solutions, II} and the fact that the equations~\eqref{curl formulation of the linearization} have constant coefficients that the operators $\Phi$ and $\Psi$ of the previous definition commute with translations of the tangential variables. Thus, we expect that they are given by  Fourier multipliers.  The following definition gives the class of the vector-valued Fourier multipliers with which we are concerned. As we will show later, $\Psi$ has a multiplier belonging to this admissible class.

\begin{defn}[Admissible class of vector-valued Fourier multipliers]\label{defn of admissable class of symbols}
We make the following definitions for $s\in\N$.
\begin{enumerate}
    \item We say that a symbol
    \begin{equation}\label{the multipliers in the admissiable class}
        m:\R^2\to\mathcal{L}\tp{H^s((0,b);\C^3)\times\C^3\times\C^2;H^{1+s}((0,b);\C)\times H^{2+s}((0,b);\C^3)\times\C^2}.
    \end{equation}
    belongs to the admissible class $\mathfrak{A}(s)$ if it is strongly measurable in the sense of the second item of Definition~\ref{definition of strong measurability}, and upon writing $m$ in `matrix form'    
    \begin{equation}
        m=\bpm m_{11}&m_{12}&m_{13}\\ m_{21}&m_{22}&m_{23}\\m_{31}&m_{32}&m_{33}\epm,
    \end{equation}
    we have that there exists a constant $C<\infty$ and a full measure set $E\subseteq\R^2$ such that for $\xi\in E$:
    \begin{enumerate}
        \item $m_{11}(\xi)\in\mathcal{L}(H^s((0,b);\C^3);H^{1+s}((0,b);\C))$ obeys the bound 
        \begin{equation}\label{m11}
            \tnorm{m_{11}(\xi)\phi}_{H^{1+s}}+\tbr{\xi}^{1+s}\tnorm{m_{11}(\xi)\phi}_{L^2}\le C\tp{\tnorm{\phi}_{H^s}+\tbr{\xi}^s\tnorm{\phi}_{L^2}}
        \end{equation}
        for $\phi\in H^s((0,b);\C^3)$;
        \item $m_{12}(\xi)\in\mathcal{L}(\C^3;H^{1+s}((0,b);\C))$ obeys the bound
        \begin{equation}\label{m12}
            \tnorm{m_{12}(\xi)}_{H^{1+s}}+\tbr{\xi}^{1+s}\tnorm{m_{12}(\xi)}_{L^2}\le C\tbr{\xi}^{1/2+s};
        \end{equation}
        \item $m_{13}\in\mathcal{L}(\C^2;H^{1+s}((0,b);\C))$
        obeys the bound
        \begin{equation}\label{m13}
            \tnorm{m_{13}(\xi)}_{H^{1+s}}+\tbr{\xi}^{1+s}\tnorm{m_{13}(\xi)}_{L^2}\le C\tbr{\xi}^{5/2+s};
        \end{equation}
        \item $m_{21}(\xi)\in\mathcal{L}(H^s((0,b);\C^3),H^{2+s}((0,b);\C^3))$ obeys the bound 
        \begin{equation}\label{m21}
            \tnorm{m_{21}(\xi)\phi}_{H^{2+s}}+\tbr{\xi}^{2+s}\tnorm{m_{21}(\xi)\phi}_{L^2}\le C\tp{\tnorm{\phi}_{H^s}+\tbr{\xi}^s\tnorm{\phi}_{L^2}}
        \end{equation}
        for $\phi\in H^s((0,b);\C^3)$;
        \item $m_{22}(\xi)\in\mathcal{L}(\C^3;H^{2+s}((0,b);\C^3))$ obeys the bound
        \begin{equation}\label{m22}
            \tnorm{m_{22}(\xi)}_{H^{2+s}}+\tbr{\xi}^{2+s}\tnorm{m_{22}(\xi)}_{L^2}\le C\tbr{\xi}^{1/2+s};
        \end{equation}
        \item $m_{23}(\xi)\in\mathcal{L}(\C^2;H^{2+s}((0,b);\C^3))$ obeys the bound
        \begin{equation}\label{m23}
            \tnorm{m_{23}(\xi)}_{H^{2+s}}+\tbr{\xi}^{2+s}\tnorm{m_{23}(\xi)}_{L^2}\le C\tbr{\xi}^{5/2+s};
        \end{equation}
        \item $m_{31}(\xi)\in\mathcal{L}(H^s((0,b);\C^3);\C^2)$ obeys the bound 
        \begin{equation}\label{m31}
            \tbr{\xi}^{3/2+s}|m_{31}(\xi)\phi|\le C\tp{\tnorm{\phi}_{H^s}+\tbr{\xi}^s\tnorm{\phi}_{L^2}}
        \end{equation}
        for  $\phi\in H^s((0,b);\C^3)$;
        \item $m_{32}(\xi)\in\mathcal{L}(\C^3;\C^2)$ obeys the bound        \begin{equation}\label{m32}
            \tbr{\xi}^{3/2+s}|m_{32}(\xi)|\le C\tbr{\xi}^{1/2+s};
        \end{equation}
        \item $m_{33}(\xi)\in\mathcal{L}(\C^2;\C^2)$ obeys the bound
        \begin{equation}\label{m33}
            \tbr{\xi}^{3/2+s}|m_{33}(\xi)|\le C\tbr{\xi}^{5/2+s}.
        \end{equation}
    \end{enumerate}
    \item If $m\in\mathfrak{A}(s)$, then we write $\jump{m}_s\in[0,\infty)$ to be the infimum over the constants $C$ for which the bounds~\eqref{m11}, \eqref{m12}, \eqref{m13}, \eqref{m21}, \eqref{m22}, \eqref{m23}, \eqref{m31}, \eqref{m32}, and~\eqref{m33} hold over a full measure set of frequencies. This makes $\mathfrak{A}(s)$ into a Banach space.
\end{enumerate}
\end{defn}

Before continuing, we remark that the $\tjump{\cdot}_s$-norm on $\mathfrak{A}(s)$ locally controls essential supremum norm.

\begin{lem}[Local essentially uniform control]\label{lemma on local esssup control}
    Fix $s\in\N$. For any $R\in\R^+$ there exists a constant $C_{R}\in\R^+$, depending only on $s$ and $R$, such that for all $m\in\mathfrak{A}(s)$ we have the estimate
    \begin{equation}\label{R-bound bound I R bound yeah R bound bound bound whoo}
    \tnorm{\mathds{1}_{B(0,R)}m}_{L^\infty_\ast\mathcal{L}}\le C_{R}\tjump{m}_{s},
    \end{equation}
    where here $\mathcal{L}$ refers to the space of linear operators on the right hand side of~\eqref{the multipliers in the admissiable class} and $L^\infty_\ast\mathcal{L}$ is the norm from Definition~\ref{definition of Linfty}.  
    Moreover, the map $R\mapsto C_R$ is bounded on bounded sets. 
\end{lem}
\begin{proof}
    An inspection of the $\jump{\cdot}_s$ norm in Definition~\ref{defn of admissable class of symbols} shows that we can replace the $\xi$-dependent bounds on the right hand sides of~\eqref{m11}--\eqref{m33} by their value at $R$ to derive~\eqref{R-bound bound I R bound yeah R bound bound bound whoo}.
\end{proof}

The utility of Definition~\ref{defn of admissable class of symbols} is seen in the next result.

\begin{prop}[Symbols and translation commuting maps]\label{prop on symbols and translation commuting maps}
    The following are equivalent for $s\in\N$ and a bounded linear map $T:\Y_s\to\X_s$:
    \begin{enumerate}
        \item $T$ commutes with the collection of tangential translation operators in the sense that for all $Y\in\Y_s$ and all $h\in\R^3$ satisfying $h\cdot e_3=0$, it holds that $(TY)(\cdot+h)=T(Y(\cdot+h))$.
        \item There exists $m\in\mathfrak{A}(s)$ such that $T=m(D)$.
    \end{enumerate}
    In either case, $m$ is unique up to modification on a set of measure zero, and  we have that 
    \begin{equation}
        \tjump{m}_s\asymp\tnorm{T}_{\mathcal{L}(\Y_s;\X_s)}
    \end{equation}
    with implicit constants depending only on $s$.
\end{prop}
\begin{proof}
    Thanks to the norm equivalence of Lemma~\ref{lemma on equivalent norm on the mixed type spaces} with $p=2$, each of the nine components of the `matrix' of $T$ satisfies the hypotheses of Corollary~\ref{corc more on translation commuting linear maps}. Thus, by enumerating the estimates on each component, we find that we are granted a multiplier $m\in\mathfrak{A}(s)$ such that $T=m(D)$. Estimate~\eqref{estimate we need on the multiplier} implies that $\tjump{m}_s\lesssim\tnorm{T}_{\mathcal{L}(\Y_s;\X_s)}$. The opposite inequality and the fact that the second item implies the first are immediate from Plancherel's theorem.
\end{proof}

Our next result relates translations of the symbol to conjugation of the operator by complex exponentials. This is the key that allows us to recast questions of smoothness for multipliers in the language of PDE on the spatial side.

\begin{prop}[Symbol translation]\label{proposition on symbol translation}
    Let $s \in \N$, and suppose that $T:\Y_s\to\X_s$ is a tangentially translation commuting bounded linear map with associated symbol $m\in\mathfrak{A}(s)$. Then for all $Y=(f,k,H)\in\Y_s$ and all $\zeta\in\R^2$ we have the identity
    \begin{equation}\label{the translation identity}
        m(D+\zeta)Y=\bf{e}_{-\zeta}(T(\bf{e}_{\zeta}Y)),
    \end{equation}
    as an equality in the space $\X_s$, where $\bf{e}_\zeta(\xi)=e^{2\pi\ii\xi\cdot\zeta}$ for $\xi\in\R^2$, and $m(D+\zeta)=(m(\cdot+\zeta))(D)$.
\end{prop}
\begin{proof}
With $Y$ and $\zeta$ as in the statement, we have that the second item of Proposition~\ref{prop on symbols and translation commuting maps} applies and we have $T(\bf{e}_\zeta Y)=m(D)(\bf{e}_{\zeta}Y)$. On the other hand, it is clear by properties of the Fourier transform
\begin{equation}
    \mathscr{F}[\bf{e}_\zeta m(D+\zeta)Y]=\mathscr{F}[m(D+\zeta)Y](\cdot-\zeta)=m\mathscr{F}[Y](\cdot-\zeta)=m\mathscr{F}[\bf{e}_\zeta Y].
\end{equation}
By taking inverse Fourier transforms, we reveal identity~\eqref{the translation identity}.
\end{proof}
% _+__+_ -_+__+_ -_+__+_ -_+__+_ -_+__+_ -_+__+_ -_+__+_ -_+__+_ -_+__+_ -_+__+_ -_+__+_ -_+__+_ -_+__+_ -
\subsection{Derivative estimates for the symbol}
% _+__+_ -_+__+_ -_+__+_ -_+__+_ -_+__+_ -_+__+_ -_+__+_ -_+__+_ -_+__+_ -_+__+_ -_+__+_ -_+__+_ -_+__+_ -

The main symbols of concern in this subsection are given by the following definition. The notation here is from Definitions~\ref{defn of spaces and TILOs} and~\ref{defn of admissable class of symbols}.  

\begin{defn}[Main symbol]\label{defn of the main symbol}
    Let $s\in\N$. We denote by $\bf{m}\in\mathfrak{A}(s)$ the symbol associated with the translation commuting bounded linear map $\Psi:\Y_s\to\X_s$ whose existence is guaranteed by Proposition~\ref{prop on symbols and translation commuting maps}.
\end{defn}

Our aim now is to study the differentiability of $\bf{m}$ in $\R^2 \backslash \{0\}$.  In doing so, it will be very convenient to introduce some minor abuse of notation in order to make various expressions easier to read.  This abuse, the use of which we confine to this subsection, is to view any vector $\theta \in \C^2$ as being contained in $\C^3$ via $(\theta,0) \in \C^3$.  For example, this shorthand means that for $v \in \C^3$ and $\theta \in \C^2$,
\begin{equation}
    \theta \cdot v = (\theta,0)\cdot v = \sum_{j=1}^2 \theta_j v_j,
\end{equation}
and so on.  This abuse will only be used in equations that are naturally understood to be posed in $\C^3$.

Our next result explores the manifestations of symbol translation on the PDE side.

\begin{prop}[Spatial realization of symbol translation]\label{prop on spatial realization on symbol realization}
    Suppose that $s\in\N$, $(f,k,H)\in\Y_s$, and $\zeta\in\R^2$. Then $\bf{m}(D+\zeta)(f,k,H)=(p_\zeta,u_\zeta,\chi_\zeta)\in\X_s$
    obeys the following equations
\begin{equation}\label{zeta translated symbol equations}
    \begin{cases}
        \mathfrak{g} \chi_\zeta+\grad p_\zeta +  2\pi\ii\zeta p_\zeta-\mu \grad\cdot\mathbb{D}u_\zeta - \mu 2\pi\ii\grad(\zeta\cdot u_\zeta) - \mu 4\pi\ii\zeta\cdot\grad u_\zeta + \mu 4\pi^2|\zeta|^2 u_\zeta=f&\text{in }\Omega,\\
        \grad\cdot u_\zeta+2\pi\ii\zeta\cdot u_\zeta=0&\text{in }\Omega,\\
        -(p_\zeta I -\mu \mathbb{D} u_\zeta)e_3 + \mu 2\pi\ii\zeta(u_\zeta\cdot e_3) - \kappa \grad_{\|}\cdot\chi_\zeta e_3 - \kappa 2\pi\ii\zeta\cdot \chi_\zeta e_3=k&\text{on }\Sigma,\\
        \grad_{\|}^\perp\cdot\chi_\zeta+2\pi\ii\zeta^\perp\cdot\chi_\zeta=0&\text{on }\Sigma,\\
        u_\zeta\cdot e_3=\grad_{\|}\cdot H+2\pi\ii\zeta\cdot H&\text{on }\Sigma,\\
        u_\zeta=0&\text{on }\Sigma_0.
    \end{cases}
\end{equation}
\end{prop}
\begin{proof}
    We invoke Proposition~\ref{proposition on symbol translation} and Definition~\ref{defn of the main symbol} to see that $(p_\zeta,u_\zeta,\chi_\zeta)=\bf{e}_{-\zeta}\Psi(\bf{e}_\zeta f,\bf{e}_\zeta k,\bf{e}_\zeta H)$, and hence $(\bf{e}_\zeta p_\zeta,\bf{e}_\zeta u_\zeta,\bf{e}_\zeta\chi_\zeta)$ is a solution to the equations~\eqref{curl formulation of the linearization} with data $(0,\bf{e}_\zeta f,\bf{e}_\zeta k,\grad_{\|}\cdot(\bf{e}_\zeta H),0)$.  We can then derive equations for $(p_\zeta,u_\zeta,\chi_\zeta)$ by expanding with the Leibniz rule and then multiplying by $\bf{e}_{\zeta}$.  This results in the system~\eqref{zeta translated symbol equations}.
\end{proof}

The following lemma gives a simple estimate for the solution to~\eqref{zeta translated symbol equations} and, more importantly, a relationship between the translated symbol and the operators $\Phi$ and $\Psi$ from Definition~\ref{defn of spaces and TILOs}.

\begin{lem}[Estimate and identity for the translated symbol PDE]\label{lem on estimates for the translated symbol PDE}
The following hold for $s\in\N$, $(f,k,H)\in\Y_s$, $\zeta\in\R^2$, and $(p_\zeta,u_\zeta,\chi_\zeta)=\bf{m}(D+\zeta)(f,k,H)$.
\begin{enumerate}
    \item We have the estimate
    \begin{equation}\label{soft arguments}
        \tnorm{p_\zeta,u_\zeta,\chi_\zeta}_{\X_s}\le C_\zeta \tnorm{\Psi}_{\mathcal{L}}\tnorm{f,k,H}_{\Y_s},
    \end{equation}
    where the constant $\R^2\ni \zeta\mapsto C_\zeta\in\R^+$ is bounded on bounded sets.
    \item If $0\not\in\supp\mathscr{F}(f,k,H)$, then 
    \begin{equation}\label{the starting point for the jet expansion}
            (p_\zeta,u_\zeta,\chi_\zeta)
            = \Psi(f,k,H)
            + \Phi\bpm -2\pi\ii\zeta\cdot u_\zeta \\ 
            -2\pi\ii\zeta p_\zeta + \mu 2\pi\ii\grad(\zeta\cdot u_\zeta) + \mu 4\pi\ii\zeta\cdot\grad u_\zeta - \mu 4\pi^2|\zeta|^2u_\zeta \\
            \kappa 2\pi\ii\zeta\cdot\chi_\zeta e_3 - \mu 2\pi\ii\zeta(u_\zeta\cdot e_3) \\
            2\pi\ii\zeta\cdot H \\
            -2\pi\ii\zeta^\perp\cdot\chi_\zeta\epm.
    \end{equation}
\end{enumerate}
\end{lem}
\begin{proof}
For the first item, we note that $(\bf{e}_\zeta p_\zeta,\bf{e}_\zeta u_\zeta,\bf{e}_\zeta\chi_\zeta)=\Psi(\bf{e}_\zeta f,\bf{e}_\zeta k,\bf{e}_\zeta H)$ and hence
\begin{equation}
    \tnorm{p_\zeta,u_\zeta,\chi_\zeta}_{\X_s}\lesssim_\zeta \tnorm{\bf{e}_\zeta p_\zeta,\bf{e}_\zeta u_\zeta,\bf{e}_\zeta\chi_\zeta}_{\X_s} \le \tnorm{\Psi}_{\mathcal{L}}\tnorm{\bf{e}_\zeta f,\bf{e}_\zeta k,\bf{e}_\zeta H}_{\Y_s}\lesssim_\zeta\tnorm{\Psi}_{\mathcal{L}}\tnorm{f,k,H}_{\Y_s}.
\end{equation}
    Next, we use the system~\eqref{zeta translated symbol equations} from Proposition~\ref{prop on spatial realization on symbol realization} along with Definition~\ref{defn of spaces and TILOs} to derive the equation~\eqref{the starting point for the jet expansion}. Note that the hypotheses $0\not\in\supp\mathscr{F}(f,k,H)$ ensures that the argument of $\Phi$ in~\eqref{the starting point for the jet expansion} belongs to its domain $\tilde{\Y}_s$. 
\end{proof}

Motivated by formula~\eqref{the starting point for the jet expansion}, we now construct the translation commuting linear maps whose symbols will turn out to be the derivatives of $\bf{m}$. In what follows $\mathcal{S}_\ell$ denotes the symmetric group on the set $\tcb{1,\dots,\ell}$.

\begin{defn}[Iterative derivative construction]\label{defn iterative jet construction}
    Let $(f,k,H)\in\Y_s$ satisfy $0\not\in\supp\mathscr{F}(f,k,H)$ and set $(p,u,\chi)=\Psi(f,k,H)\in\X_s$. For $j\in\N^+$ we define the $j$-multilinear and symmetric maps
    \begin{equation}
        (\R^2)^j\ni(\zeta_1,\dots,\zeta_j)\mapsto(p^{(j)},u^{(j)},\chi^{(j)})[\zeta_1,\dots,\zeta_j]\in\X_s
    \end{equation}
    via the following inductive procedure. If $j=1$, we set
    \begin{equation}\label{derivative definition}
        (p^{(1)},u^{(1)},\chi^{(1)})[\zeta]=\Phi\bpm -2\pi\ii\zeta\cdot u\\
        -2\pi\ii\zeta p + \mu 2\pi\ii\grad(\zeta\cdot u) + \mu 4\pi\ii\zeta\cdot\grad u\\
        \kappa 2\pi\ii\zeta\cdot\chi e_3 - \mu 2\pi\ii\zeta(u\cdot e_3)\\
        2\pi\ii\zeta\cdot H\\
        -2\pi\ii\zeta^\perp\cdot\chi
        \epm.
    \end{equation}
    If $j=2$, we set
    \begin{multline}
        (p^{(2)},u^{(2)},\chi^{(2)})[\zeta_1,\zeta_2]\\ =\sum_{\sig\in\mathcal{S}_2}\Phi\bpm -2\pi\ii\zeta_{\sig1}\cdot u^{(1)}[\zeta_{\sig2}]\\
        -2\pi\ii\zeta_{\sig 1}p^{(1)}[\zeta_{\sig 2}] + \mu 2\pi\ii\grad(\zeta_{\sig 1}\cdot u^{(1)}[\zeta_{\sig2}]) + \mu 4\pi\ii\zeta_{\sig 1}\cdot\grad u^{(1)}[\zeta_{\sig 2}] - \mu 4\pi^2(\zeta_{\sig1}\cdot\zeta_{\sig 2})u\\
        \kappa 2\pi\ii\zeta_{\sig1}\cdot\chi^{(1)}[\zeta_{\sig 2}] e_3 - \mu 2\pi\ii\zeta_{\sig 1}(u^{(1)}[\zeta_{\sig 2}]\cdot e_3)\\
        0\\
        -2\pi\ii\zeta^\perp_{\sig 1}\cdot{\chi^{(1)}}[\zeta_{\sig 2}]
        \epm
    \end{multline}
    and if $j\ge 3$, we take
    \begin{multline}
        (p^{(j)},u^{(j)},\chi^{(j)})[\zeta_1,\dots,\zeta_j] \\ = \f{1}{(j-1)!}\sum_{\sig\in\mathcal{S}_j}\Phi\bpm -2\pi\ii\zeta_{\sig1}\cdot u^{(j-1)}\\
        -2\pi\ii\zeta_{\sig 1}p^{(j-1)}+ \mu  2\pi\ii\grad(\zeta_{\sig 1}\cdot u^{(j-1)}) + \mu 4\pi\ii\zeta_{\sig 1}\cdot\grad u^{(j-1)}\\
        \kappa 2\pi\ii\zeta_{\sig1}\cdot\chi^{(j-1)} e_3 - \mu 2\pi\ii\zeta_{\sig1}(u^{(j-1)}\cdot e_3)\\0\\-2\pi\ii\zeta^\perp_{\sig1}\cdot\chi^{(j-1)}
        \epm[\zeta_{\sig 2},\dots,\zeta_{\sig j}]\\
        +\f{1}{(j-2)!}\sum_{\sig\in\mathcal{S}_j}\Phi\bpm0\\-
        \mu 4\pi^2(\zeta_{\sig 1}\cdot\zeta_{\sig 2})u^{(j-2)}\\
        0\\0\\0
        \epm[\zeta_{\sig 3},\dots,\zeta_{\sig j}].
    \end{multline}
\end{defn}

Our next result studies the previous construction more carefully.

\begin{prop}[Properties of the derivative construction]\label{proposition on jet verification}
    The following hold.
    \begin{enumerate}
        \item For every $j\in\N^+$ the map
        \begin{equation}
            \tcb{(f,k,H)\in\Y_s\;:\;0\not\in\supp\mathscr{F}(f,k,H)}\ni(f,k,H)\mapsto|D|^j(p^{(j)},u^{(j)},\chi^{(j)})\in\mathcal{L}^j(\R^2;\X_s)
        \end{equation}
        is continuous and tangentially translation commuting, and hence extends uniquely to a bounded linear map defined on all of $\Y_s$.
        \item For every $j\in\N^+$ there exists a unique multilinear mapping into the space of symbols, $(\zeta_1,\dots,\zeta_j)\mapsto\bf{m}^{(j)}[\zeta_1,\dots,\zeta_j]$, such that for all $(f,k,H)\in\Y_s$ satisfying $0\not\in\supp\mathscr{F}(f,k,H)$ we have the identity
        \begin{equation}\label{symbols yay}
            \bf{m}^{(j)}[\zeta_1,\dots,\zeta_j](D)(f,k)=(p^{(j)},u^{(j)},\chi^{(j)})[\zeta_1,\dots,\zeta_j].
        \end{equation}
        \item For all $j\in\N$ and all $(\zeta_1,\dots,\zeta_j)\in(\R^2)^j$ we have that the map  $\xi\mapsto|\xi|^j\bf{m}^{(j)}[\zeta_1,\dots,\zeta_j](\xi)$ belongs to the space $\mathfrak{A}(s)$ from Definition~\ref{defn of admissable class of symbols}; moreover, we have the estimate
        \begin{equation}\label{Hormander-Mikhlin bounds}
            \tjump{|\cdot|^j\bf{m}^{(j)}[\zeta_1,\dots,\zeta_j]}_s\lesssim j!\cdot C^j\prod_{i=1}^j|\zeta_i|,
        \end{equation}
        with an implicit constant independent of $j$.
    \end{enumerate}
\end{prop}
\begin{proof}
    We begin with the first item in the case $j=1$. The key point is that the operator norm of $\Phi:\tilde{\Y}_s\to\X_s$ depends on a few quantities belonging to $\dot{H}^{-1}(\Sigma;\C)$.  The appearance of the operator $|D|$ is thus crucial, as it permits the bounds
    \begin{equation}
        \bsb{2\pi\ii\zeta\cdot|D|H+2\pi\ii\int_0^b\zeta\cdot(|D|u)(\cdot,y)\;\m{d}y}_{\dot{H}^{-1}}+\tsb{2\pi\ii\zeta^\perp\cdot|D|\chi}_{\dot{H}^{-1}}\lesssim|\zeta|\tnorm{\chi,u,H}_{L^2\times L^2\times L^2}.
    \end{equation}
    With this observation in hand, we appeal  directly to Definition~\ref{defn iterative jet construction} and the mapping properties of $\Phi$ established in Theorem~\ref{thm on analysis of strong solutions, II} and Definition~\ref{defn of spaces and TILOs}, to verify that
    \begin{equation}
        \tnorm{|D|(p^{(1)},u^{(1)},\chi^{(1)})[\zeta]}_{\X_s}\lesssim\tp{\tnorm{p,u,\chi}_{\X_s}+\tnorm{f,k,H}_{\Y_s}}|\zeta|\lesssim\tnorm{f,k,H}_{\Y_s}|\zeta|.
    \end{equation}
   
    In a similar manner to the above, we may derive the estimate
    \begin{equation}
        \sup_{|\zeta_1|,|\zeta_2|\le 1}\tnorm{|D|^2(p^{(2)},u^{(2)},\chi^{(2)})[\zeta_1,\zeta_2]}_{\X_s}\lesssim\sup_{|\zeta|\le 1}\tnorm{|D|(p^{(1)},u^{(1)},\chi^{(1)})[\zeta]}_{\X_s}+\tnorm{p,u,\chi}_{\X_s}\lesssim\tnorm{f,k}_{\Y_s}
    \end{equation}
    and, for $j\ge 3$,
    \begin{multline}
        \sup_{|\zeta_1|,\dots,|\zeta_j|\le 1}\tnorm{|D|^j(p^{(j)},u^{(j)},\chi^{(j)})[\zeta_1,\dots,\zeta_j]}_{\X_s}\lesssim\\ j\sum_{\sig=1}^2((j-2)\sig+3-j)\sup_{|\zeta_1|,\dots,|\zeta_{j-\sig}|\le 1}\tnorm{|D|^{j-\sig}(p^{(j-\sig)},u^{(j-\sig)},\chi^{(j-\sig)})[\zeta_1,\dots,\zeta_{j-\sig}]}_{\X_s}.
    \end{multline}
    Upon iterating these estimates, we deduce the boundedness assertion of the first item.

    As a consequence of the first item and Proposition~\ref{prop on symbols and translation commuting maps}, we find that for each $j\in\N^+$ and $(\zeta_1,\dots,\zeta_j)\in(\R^2)^j$, the map $(f,k,H)\mapsto|D|^j(p^{(j)},u^{(j)},\chi^{(j)})[\zeta_1,\dots,\zeta_j]$ is given by a symbol $\tilde{\bf{m}}^{(j)}[\zeta_1,\dots,\zeta_j]\in\mathfrak{A}(s)$. By uniqueness, we must have that $\tilde{\bf{m}}^{(j)}$ is a $j$-multilinear and symmetric function of the $\zeta_1,\dots,\zeta_j$. We then set $\bf{m}^{(j)}=|\cdot|^{-j}\tilde{\bf{m}}^{(j)}$, which implies  that~\eqref{symbols yay} holds, and hence the second item is satisfied.

    The third item is follows since $|\cdot|^j\bf{m}^{(j)}=\tilde{\bf{m}}^{(j)}$, with the symbol on the right a $j$-multilinear function of the $\zeta_1,\dots,\zeta_j$ into $\mathfrak{A}(s)$.  
\end{proof}

We pause to note that estimate~\eqref{Hormander-Mikhlin bounds} would tell us that $\bf{m}$ 
obeys estimates of Mikhlin-H\"ormander type if we knew that $\bf{m}$ were smooth away from the origin with the $j^{\m{th}}$-derivative equal to $\bf{m}^{(j)}$.   It is thus our goal now to prove that these are, in fact the derivatives of the symbol $\bf{m}$.  We first require a definition and a technical lemma about remainders.  Note that the following is well-defined thanks to Proposition~\ref{proposition on jet verification}.

\begin{defn}[Remainders]\label{defn of the remainder terms}
     Given $(f,k,H)\in\Y_s$ satisfying $0\not\in\supp\mathscr{F}(f,k,H)$ and $\zeta\in\R^2$, we define the following elements of $\X_s$:
    \begin{equation}
        \mathcal{R}_0(f,k,H)[\zeta]=(\bf{m}(D+\zeta)-\bf{m}(D))(f,k),
    \end{equation}
    and for $j\in\N^+$
    \begin{equation}
        \mathcal{R}_j(f,k,H)[\zeta]=\mathcal{R}_0(f,k,H)[\zeta]-\sum_{i=1}^j\f{1}{i!}\bf{m}^{(i)}[\zeta^{\otimes i}](D)(f,k,H),
    \end{equation}
    where in the above we write $\zeta^{\otimes i}\in(\R^2)^i$ to refer to the $i$-tuple of vectors with each entry equal to $\zeta$.
\end{defn}

Now we derive estimates for these remainder terms.

\begin{lem}[Remainder estimates]\label{lemma on recursive jet identity}
    There exists a constant $C\ge 0$ such that for all $(f,k,H)\in\Y_s$ satisfying $0\not\in\supp\mathscr{F}(f,k,H)$ and $\zeta\in\R^2$ satisfying $|\zeta|<\min\tcb{\m{dist}(0,\supp\mathscr{F}(f,k,H)),1}$, we have the following estimates for $j\in\N$:
    \begin{equation}\label{recursive}
        \tnorm{|D|^{j+1}\mathcal{R}_j(f,k,H)[\zeta]}_{\X_s}\le C^{j+1}|\zeta|^{j+1}\tnorm{f,k,H}_{\Y_s}.
    \end{equation}
\end{lem}
\begin{proof}
    Throughout the proof, we will employ the following convenient notation:  we set $(p_\zeta,u_\zeta,\chi_\zeta)=\bf{m}(D+\zeta)(f,k,H)$, and for $j\in\N$ we set
    \begin{equation}
        \mathcal{R}_j(f,k,H)[\zeta]=(\mathcal{R}_jp[\zeta],\mathcal{R}_ju[\zeta],\mathcal{R}_j\chi[\zeta]).
    \end{equation}
    
    We first establish~\eqref{recursive} when $j=0$. From identity~\eqref{the starting point for the jet expansion}, we find that
    \begin{equation}\label{starting to look recursive}
        \mathcal{R}_0(f,k,H)[\zeta]=\Phi\bpm -2\pi\ii\zeta\cdot u_\zeta\\ -2\pi\ii\zeta p_\zeta + \mu 2\pi\ii\grad(\zeta\cdot u_\zeta) + \mu 4\pi\ii\zeta\cdot\grad u_\zeta - \mu 4\pi^2|\zeta|^2u_\zeta\\ \kappa 2\pi\ii\zeta\cdot\chi_\zeta e_3 - \mu 2\pi\ii\zeta(u_\zeta\cdot e_3)\\2\pi\ii\zeta\cdot H\\-2\pi\ii\zeta^\perp\cdot\chi_\zeta\epm.
    \end{equation}
    We then apply $|D|$ to the above and take the norm in $\X_s$. Thanks to the continuity properties of $\Phi:\tilde{\Y}_s\to\X_s$ (see the proof of Proposition~\ref{proposition on jet verification}) and Lemma~\ref{lem on estimates for the translated symbol PDE}, we acquire the bound
    \begin{equation}\label{the j=0 estimate}
        \tnorm{|D|\mathcal{R}_0(f,k,H)[\zeta]}_{\X_s}
        \lesssim|\zeta|\tp{\tnorm{p_\zeta,u_\zeta,\eta_\zeta}_{\X_s}+\tnorm{f,k,H}_{\Y_s}}\lesssim|\zeta|\tnorm{f,k,H}_{\Y_s},
    \end{equation}
    which completes the proof in the $j=0$ case.  
    
    Next up is $j=1$. We subtract $\bf{m}^{(1)}[\zeta](D)(f,k,H)$ from both sides of~\eqref{starting to look recursive} and recall~\eqref{derivative definition}. This yields the equation
    \begin{equation}\label{starting to look even more recursive}
        \mathcal{R}_1(f,k,H)[\zeta]=\Phi\bpm -2\pi\ii\zeta\cdot \mathcal{R}_0u[\zeta]\\ -2\pi\ii\zeta\mathcal{R}_0p[\zeta]+\mu 2\pi\ii\grad(\zeta\cdot \mathcal{R}_0u[\zeta]) + \mu 4\pi\ii\zeta\cdot\grad\mathcal{R}_0u[\zeta] -\mu 4\pi^2|\zeta|^2u_\zeta\\
        \kappa 2\pi\ii\zeta\cdot\mathcal{R}_0\chi[\zeta] e_3-\mu 2\pi\ii\zeta(\mathcal{R}_0u[\zeta]\cdot e_3)\\0\\-2\pi\ii\zeta^\perp\cdot\mathcal{R}_0\chi[\zeta]\epm.
    \end{equation}
    We now apply $|D|^2$ to~\eqref{starting to look even more recursive} and then take the norm in $\X_s$.  The mapping properties of $\Phi$, together with Lemma~\ref{lem on estimates for the translated symbol PDE} and estimate~\eqref{the j=0 estimate}, then show that
    \begin{equation}\label{r1 bound}
        \tnorm{|D|^2\mathcal{R}_1(f,k,H)[\zeta]}_{\X_s}\lesssim|\zeta|\tnorm{|D|\mathcal{R}_0(f,k,H)[\zeta]}_{\X_s}+|\zeta|^2\tnorm{|D|^2u_\zeta}_{H^s}
        \lesssim|\zeta|^2\tnorm{f,k}_{\Y_s}.
    \end{equation}
    This completes the proof in the case $j=1$.

    Now we claim the following identity holds for all $j\ge 2$:
    \begin{equation}\label{the most recursive identity}
        \mathcal{R}_j(f,k,H)[\zeta]=\Phi\bpm -2\pi\ii\zeta\cdot \mathcal{R}_{j-1}u[\zeta]\\ -2\pi\ii\zeta\mathcal{R}_{j-1}p[\zeta]+ \mu 2\pi\ii\grad(\zeta\cdot \mathcal{R}_{j-1}u[\zeta]) + \mu 4\pi\ii\zeta\cdot\grad\mathcal{R}_{j-1}u[\zeta] - \mu 4\pi^2|\zeta|^2\mathcal{R}_{j-2}u[\zeta]\\ \kappa 2\pi\ii\zeta\cdot\mathcal{R}_{j-1}\chi[\zeta] e_3 - \mu 2\pi\ii\zeta(\mathcal{R}_{j-1}u[\zeta]\cdot e_3)\\0\\-2\pi\ii\zeta^\perp\cdot\mathcal{R}_{j-1}\chi[\zeta]\epm.
    \end{equation}
    We prove this via induction. For the base case $j=2$, we look to Definition~\ref{defn iterative jet construction} to see that
    \begin{equation}
        \f{1}{2}\bf{m}^{(2)}[\zeta^{\otimes 2}](D)(f,k,H)=\bpm-2\pi\ii\zeta\cdot u^{(1)}[\zeta]\\
        -2\pi\ii\zeta p^{(1)}[\zeta]+\mu 2\pi\ii\grad(\zeta\cdot u^{(1)}[\zeta]) + \mu 4\pi\ii\zeta\cdot\grad u^{(1)}[\zeta] - \mu 4\pi^2|\zeta|^2u\\
        \kappa 2\pi\ii\zeta\cdot\chi^{(1)}[\zeta] e_3-\mu 2\pi\ii\zeta(u^{(1)}[\zeta]\cdot e_3)\\0\\-2\pi\ii\zeta^\perp\cdot\chi^{(1)}[\zeta]
        \epm.
    \end{equation}
    We subtract the above from~\eqref{starting to look even more recursive} to see that~\eqref{the most recursive identity} is true for $j=2$. Proceeding inductively, suppose that this identity holds for some $\N\ni j\ge 2$. We again look to Definition~\ref{defn iterative jet construction} to acquire the identity
    \begin{multline}\label{the above}
        \f{1}{(j+1)!}\bf{m}^{(j+1)}[\zeta^{\otimes (j+1)}](D)(f,k,H) \\
        = \f{1}{j!}\Phi\bpm -2\pi\ii\zeta\cdot u^{(j)}[\zeta^{\otimes j}]\\
        -2\pi\ii\zeta p^{(j)}[\zeta^{\otimes j}]+\mu 2\pi\ii\grad(\zeta\cdot u^{(j)}[\zeta^{\otimes(j)}])+\mu 4\pi\ii\zeta\cdot\grad u^{(j)}[\zeta^{\otimes(j)}]\\
       \kappa 2\pi\ii\zeta\cdot\chi^{(j)}[\zeta^{\otimes(j)}] e_3-\mu 2\pi\ii\zeta(u^{(j)}[\zeta^{\otimes j}]\cdot e_3)\\0\\-2\pi\ii\zeta^\perp\cdot\chi^{(j)}[\zeta^{\otimes j}]e_3
        \epm\\+\f{1}{(j-1)!}\Phi\bpm0\\ -\mu 4\pi^2|\zeta|^2u^{(j-1)}[\zeta^{\otimes(j-1)}]\\0\\0\\0\epm.
    \end{multline}
    We then simply subtract~\eqref{the above} from the induction hypothesis~\eqref{the most recursive identity} to prove the stated identity in the $j+1$ case. Thus~\eqref{the most recursive identity} holds for all $\N\ni j\ge 2$, and the claim is proved.
    
    With the claim in hand, we apply $|D|^{j+1}$ in the case $\N\ni j\ge 2$ to identity~\eqref{the most recursive identity}, take the norm in $\X_s$, and utilize the mapping properties of $\Phi$ to deduce the inductive estimate
    \begin{equation}\label{the recursive bound}
        \tnorm{|D|^{j+1}\mathcal{R}_j(f,k,H)[\zeta]}_{\X_s}\lesssim\sum_{\sig=1}^2|\zeta|^{\sig}\tnorm{|D|^{j-\sig+1}\mathcal{R}_{j-\sig}(f,k,H)[\zeta]}_{\X_s}.
    \end{equation}
    Iteratively applying~\eqref{the recursive bound} and employing~\eqref{r1 bound} and~\eqref{the j=0 estimate}, we then readily conclude that~\eqref{recursive} holds.
\end{proof}

At last, we prove analyticity of $\bf{m}$ away from the origin.

\begin{thm}[Analyticity of the symbol]\label{thm on analyticity of the symbol}
    The symbol $\bf{m}\in\mathfrak{A}(s)$ from Definition~\ref{defn of the main symbol} has a representative that is analytic as a mapping
    \begin{equation}\label{where the symbol lives}
        \bf{m}:\R^2\setminus\tcb{0}\to\mathcal{L}\tp{H^s((0,b);\C^3)\times\C^3\times\C^2;H^{1+s}((0,b);\C)\times H^{2+s}((0,b);\C^3)\times\C^2}.
    \end{equation}
    Moreover, the derivatives of the above symbol obey the following Mikhlin-H\"ormander type bounds: for every $\al\in \N^d$ there exists $C_\al<\infty$ such that
    \begin{equation}\label{the estimate on the dude}
       \tjump{|\cdot|^{|\al|}\pd^\al\bf{m}}_s\le C_\al.
    \end{equation}
\end{thm}
\begin{proof}

It suffices to prove that $\bf{m}$ is analytic as claimed and that $\grad^j\bf{m}=\bf{m}^{(j)}$ for $j \in \N^+$. Indeed, once this is established, the bound~\eqref{the estimate on the dude} is a consequence of the third item of Proposition~\ref{proposition on jet verification} and Definition~\ref{defn of the main symbol}. We divide the proof into several steps.

\emph{Step 1 - Setup and basic estimates:} 

For the sake of brevity, define the Hilbert spaces $V_0 = H^s((0,b);\C^3)\times\C^3\times\C^2$ and $V_1 = H^{1+s}((0,b);\C)\times H^{2+s}((0,b);\C^3)\times\C^2$.  We will write $\mathcal{L} = \mathcal{L}(V_0;V_1)$ throughout the proof.

Fix $\ep\in(0,1/4)$ and define the annuli $A_{2}=B(0,2\ep^{-1})\setminus\Bar{B(0,\ep/2)}$, $A_{1}=B(0,\ep^{-1})\setminus\Bar{B(0,\ep)} \subset A_2$, and $A_0=B(0,\ep^{-1}/2)\setminus\Bar{B(0,2\ep)}\subset A_1$.  Thanks to Definition~\ref{defn of the remainder terms}, for $j\in\N$ and $|\zeta|<\ep/2$ the linear map
    \begin{equation}
        \Y_s\ni(f,k,H)\overset{T_j^\ep[\zeta]}{\mapsto}\mathcal{R}_j(f,k,H)[\zeta] \mathds{1}_{A_2}(D)(f,k)\in\X_s
    \end{equation}
    is bounded, translation commuting, and satisfies $T_j^\ep[\zeta]=\bf{m}_j^\ep[\zeta](D)$ for the symbol
    \begin{equation}
        \bf{m}_j^\ep[\zeta](\xi)=\bp{\bf{m}(\xi+\zeta) - \sum_{i=0}^j\f{1}{i!}\bf{m}^{(i)}[\zeta^{\otimes i}](\xi)}\mathds{1}_{A_2}(\xi),
    \end{equation}
    with the understanding that $\f{1}{0!}\bf{m}^{(0)}[\zeta^{\otimes 0}](\xi) = \bf{m}(\xi)$.      Lemma~\ref{lemma on recursive jet identity} provides the estimate
    \begin{equation}
        \tnorm{|D|^{j+1}T^\ep_j[\zeta]}_{\mathcal{L}(\Y_s;\X_s)}\le C^{j+1}|\zeta|^{j+1},
    \end{equation}
    and so Proposition~\ref{prop on symbols and translation commuting maps} then yields multiplier bound
    \begin{equation}
        \tjump{|\cdot|^{j+1}\bf{m}^\ep_j[\zeta]}_s\le C^{j+1}|\zeta|^{j+1} \text{ for all } \abs{\zeta} < \ep/2.
    \end{equation}

    Now, since the multiplier $\bf{m}^\ep_j$ is supported in the set $A_2$, we are free to apply Lemma~\ref{lemma on local esssup control} and find that there exists a constant $c_\ep$, depending only $\ep$ and $s$, such that
    \begin{equation}\label{the all important identity}
        \esssup_{\xi \in A_2}|\xi|^{j+1}\bnorm{\bf{m}(\xi+\zeta)  -\sum_{i=0}^j\f{1}{i!}\bf{m}^{(i)}[\zeta^{\otimes i}](\xi)}_{\mathcal{L}}\le c_\ep \tjump{|\cdot|^{j+1}\bf{m}_j^\ep[\zeta]}_s\le c_\ep C^{j+1}|\zeta|^{j+1},
    \end{equation}
    where we recall that $\mathcal{L} = \mathcal{L}(V_0;V_1)$, and  we  deduce from this that if $\abs{\zeta} < \ep/2$, then 
    \begin{equation}\label{remainder estimate is bananas nabandanas while the whale of a time is a watched pot}
       \esssup_{\xi \in A_2}\bnorm{\bf{m}(\xi+\zeta) - \sum_{i=0}^j\f{1}{i!}\bf{m}^{(i)}[\zeta^{\otimes i}](\xi)}_{\mathcal{L}}\le c_\ep\bp{\f{2C|\zeta|}{\ep}}^{j+1}.
        %\le c_\ep 2^{-(j+1)}.
    \end{equation}
    A similar application of Definition~\ref{defn of the main symbol} and Lemma~\ref{lemma on local esssup control} shows that 
    \begin{equation}\label{basic m bound}
        \esssup_{\xi \in A_2}\norm{\bf{m}(\xi)}_{\mathcal{L}} \le c_\ep.
    \end{equation}

\emph{Step 2 -  Lipschitz continuity in $A_1$:}

We now aim to show that $\bf{m} : A_1 \to \mathcal{L}$ has a Lipschitz continuous representative.  To this end, let $\varphi \in C^\infty_c(\R^n)$ be such that $\varphi \ge 0$, $\int \varphi =1$, and $\supp(\varphi) \subseteq B(0,1)$. For $0 < \delta < \ep/4$ and $\xi \in A_1$ we define $\bf{n}_\ep(\xi) : V_0 \to V_1$ via 
\begin{equation}
    \bf{n}_\delta(\xi) v = \int_{B(0,\delta)} \frac{1}{\delta^2} \varphi(\omega/\delta) \bf{m}(\xi-\omega) v\;\m{d}\omega,
\end{equation}
which is well-defined since $B(0,\delta) \ni \omega \mapsto \bf{m}(\xi-\omega)v \in V_2$ is measurable and essentially bounded for each $v \in V_0$ and $\xi \in A_1$.  From this it's easy to see that $\bf{n}_\delta(\xi) \in \mathcal{L}$ for each $\xi \in A_1$ and that
the induced map $\bf{n}_\delta : A_1 \to \mathcal{L}$ is continuous.  In fact,  the mollified sequence is uniformly Lipschitz in $\del$.  Indeed, from~\eqref{remainder estimate is bananas nabandanas while the whale of a time is a watched pot} with $j=0$ 
we know that there exists a null set $N \subset A_2$ such that if $\xi \in A_1$ and $\omega \in B(0,\delta)$ are such that $\xi - \omega \in A_2 \backslash N$, then 
\begin{equation}
   \sup_{\norm{v}_{V_0} \le 1} \norm{\bf{m}(\xi-\omega + \eta)v - \bf{m}(\xi-\omega)v}_{V_1} 
    \le \sup_{\theta \in A_2\backslash N} \norm{\bf{m}(\theta + \eta) - \bf{m}(\theta)}_{\mathcal{L}}  
    \le \f{2c_\ep C}{\ep}|\eta|.
\end{equation}
Consequently, for any given $\xi \in A_1$ and $\abs{\eta} < \ep/2$ we have that 
\begin{equation}
    \sup_{\norm{v}_{V_0} \le 1} \norm{\bf{m}(\xi-\omega + \eta)v - \bf{m}(\xi-\omega)v}_{V_1}   \le \f{2c_\ep C}{\ep}|\eta| \text{ for a.e. }\omega \in B(0,\delta),
\end{equation}
and hence 
\begin{equation}
    \tnorm{\bf{n}_\del(\xi+\eta)-\bf{n}_\del(\xi)}_{\mathcal{L}} \le \f{2c_\ep C}{\ep}|\eta| \int_{\R^2 } \frac{1}{\delta^2} \varphi(\omega/\delta) \;\m{d}\omega = \f{2c_\ep C}{\ep}|\eta|.
\end{equation}
Since $\xi \in A_1$ was arbitrary, we deduce that 
\begin{equation}\label{lip estimate}
    \sup_{\xi \in A_1} \tnorm{\bf{n}_\del(\xi+\eta)-\bf{n}_\del(\xi)}_{\mathcal{L}}\le  \f{2c_\ep C}{\ep}|\eta|
\end{equation}
for all $\abs{\eta} < \ep/2$.  By similar considerations, we may estimate 
\begin{equation}\label{Cauchy estimate}
    \tnorm{\bf{n}_\del-\bf{m}}_{L^\infty_\ast(A_1;\mathcal{L})}\le
     \f{2c_\ep C}{\ep}\int_{B(0,\del)}\f{1}{\del^2}\varphi(\omega/\del)|\omega|\;\m{d}\omega\le\f{2c_\ep C}{\ep}\del.
\end{equation}
Sending $\delta \to 0$ and appealing to \eqref{lip estimate} and \eqref{basic m bound}, we find that the restriction of $\bf{m}$ to $A_1$ is almost everywhere equal to a Lipschitz continuous function from $A_1$ to $\mathcal{L}$. From now on, we shall use the continuous representative of $\bf{m}$ in $A_1$.

\emph{Step 3 - Local convergence of the power series in $A_2$:}

We now claim that there exists a constant $C_1 \ge 2$ and a set $\mathfrak{E}\subseteq A_2$ with $|A_2\setminus\mathfrak{E}|=0$ such that  for any $\xi\in\mathfrak{E}$ the power series
    \begin{equation}\label{surfin US and A}
        \R^2\supset B(\xi,\ep /C_1)\ni\zeta\mapsto \sum_{j=0}^\infty\bf{m}^{(j)}[\zeta^{\otimes j}](\xi)\in\mathcal{L}
    \end{equation}
    converges uniformly absolutely and thus defines an analytic $\mathcal{L}$-valued function in $B(\xi,\ep/N)$.

    To prove the claim, we first appeal to the bounds~\eqref{Hormander-Mikhlin bounds} from Proposition~\ref{proposition on jet verification} and Lemma~\ref{lemma on local esssup control} to find a constant $R>0$ such that for every $j\in\N$ we have the bound
    \begin{equation}\label{power series bananas}
        \f{1}{j!} \esssup_{\xi \in A_2} |\xi|^j\tnorm{\bf{m}^{(j)}[\zeta^{\otimes j}](\xi)}_{\mathcal{L}}\le c_\ep R^j\tabs{\zeta}^j 
         \text{ for every }|\zeta|<\f{\ep}{2}.
    \end{equation}
    This estimate can be improved by virtue of multilinearity. To see how, set $C_1 = 2(1+2R)$ and let  $\tcb{\zeta_n}_{n\in\N}\subset B(0,\ep/C_1)$ be dense. Then by~\eqref{power series bananas}, the fact that a countable union of null sets is again null,  the pointwise continuity of the multilinear maps involved, and the fact that $\m{dist}(0,A_2)=\ep/2$, we have that
\begin{equation}\label{improved bound}
    \esssup_{\xi\in A_2} \sup_{|\zeta|< \ep/C_1}
     \tnorm{\bf{m}^{(j)}[\zeta^{\otimes j}](\xi)}_{\mathcal{L}}
    \le \esssup_{\xi\in A_2} \sup_{n\in\N} \frac{2^j|\xi|^j}{\ep^j} \tnorm{\bf{m}^{(j)}[\zeta_n^{\otimes j}](\xi)}_{\mathcal{L}}\le j!\cdot c_\ep \left(\frac{2R}{C_1}\right)^j \le j! \cdot c_\ep 2^{-j}.
\end{equation}
     Summing over $j\in\N$, we may thus bound 
    \begin{equation}\label{estimates on the radius of convergence of the power series bananas}
        \sum_{j=0}^\infty\f{1}{j!}\esssup_{\xi\in A_2}\sup_{|\zeta|<\ep/C_1}\tnorm{\bf{m}^{(j)}[\zeta^{\otimes j}](\xi)}_{\mathcal{L}}\le 2 c_\ep,
    \end{equation}
    and the claim now follows directly from this.

\emph{Step 4 - Analyticity in $A_0$:}

Finally, we aim to prove that the multiplier $\bf{m}$ in~\eqref{where the symbol lives} is analytic in the annulus $A_0$.  The strategy is to combine estimates~\eqref{estimates on the radius of convergence of the power series bananas} and~\eqref{remainder estimate is bananas nabandanas while the whale of a time is a watched pot}.  Let $C_2 =\max\{C_1, 2(1+2C)\}$, where $C_1$ is from Step 3 and $C>0$ is the constant in ~\eqref{remainder estimate is bananas nabandanas while the whale of a time is a watched pot}.  By letting $\tcb{\zeta_n}_{n\in\N}\subset B(0,\ep/C_2)$ be dense, we may use~\eqref{remainder estimate is bananas nabandanas while the whale of a time is a watched pot} to produce a set $\mathfrak{F}\subseteq A_2$ with the property that $|A_2\setminus\mathfrak{F}|=0$ and if $\xi\in\mathfrak{F}$ and $j\in\N$ then
\begin{equation}
    \sup_{n\in\N}\bnorm{\bf{m}(\xi+\zeta_n)-\sum_{i=0}^j\bf{m}^{(i)}[\zeta_n^{\otimes i}](\xi)}_{\mathcal{L}}\le c_\ep 2^{-(j+1)}.
\end{equation}
In particular, if $\xi\in A_0\cap\mathfrak{E}$, then the continuity assertion of Step 2 allows us to bound 
\begin{equation}\label{get us some ice cubes please its getting cold in here}
    \sup_{|\zeta|<\ep/C_2}\bnorm{\bf{m}(\xi+\zeta)-\sum_{i=0}^j\bf{m}^{(i)}[\zeta^{\otimes i}](\xi)}_{\mathcal{L}}\le c_\ep 2^{-(j+1)}.
\end{equation}

 Let $\mathfrak{F}$ be the set from Step 3 and note that $A_0\cap\mathfrak{E}\cap\mathfrak{F}$ is a set of full measure in $A_0$.  We can then send $j \to \infty$ in \eqref{get us some ice cubes please its getting cold in here} to see that
 \begin{equation}\label{powerzig}
     \bf{m}(\xi+\zeta)= \sum_{i=0}^\infty\bf{m}^{(i)}[\zeta^{\otimes i}](\xi) \text{ for } \xi \in A_0\cap\mathfrak{E}\cap\mathfrak{F} \text{ and } \zeta\in B(0,\ep/C_2).
 \end{equation}
 In light of the continuity of $\bf{m}$ in $A_1$, we learn from this that the power series produced in Step 3 agree on the intersection of their balls of convergence.  Consequently, we may produce a single $\mathcal{L}-$valued analytic function on $A_0$ that is equal to $\bf{m}$ a.e., and since $\ep>0$ was arbitrary we conclude that $\bf{m}$ has an analytic representative in $\R^2 \backslash \{0\}$.  Finally, we learn from \eqref{powerzig} that $\nabla^j \bf{m} = \bf{m}^{(j)}$ almost everywhere.
     
    \end{proof}

% _+__+_ -_+__+_ -_+__+_ -_+__+_ -_+__+_ -_+__+_ -_+__+_ -_+__+_ -_+__+_ -_+__+_ -_+__+_ -_+__+_ -_+__+_ -
\section{On some Sobolev-type spaces}\label{section on novel Sobolev spaces}
% _+__+_ -_+__+_ -_+__+_ -_+__+_ -_+__+_ -_+__+_ -_+__+_ -_+__+_ -_+__+_ -_+__+_ -_+__+_ -_+__+_ -_+__+_ -

In Section~\ref{section on vector-valued symbol calculus for the solution map} we constructed an operator-valued symbol $\bf{m}$ such that the corresponding Fourier multiplication operator $\bf{m}(D)$ is a particular solution map for the PDE~\eqref{curl formulation of the linearization}.  We know from Theorem~\ref{thm on analyticity of the symbol} that $\bf{m}$ obeys certain inequalities of Mikhlin-H\"ormander type, and so we expect to be able to employ Theorem \ref{second HM multiplier theorem} to obtain the boundedness of $\bf{m}(D)$ on certain vector-valued Sobolev spaces.  The first goal of this section is to define and study these mixed-type spaces for use in this manner.  This is done in Section~\ref{appendix on properties of mixed-type Sobolev spaces}. In Section \ref{appendix on some nonlinear analysis in mized type Sobolev spaces} we record a number of nonlinear tools that we will use in working with the mixed-type spaces.

The second purpose of this section, which is the content of Section~\ref{section on properties of subcritical gradient spaces}, is to study what we call subcritical gradient spaces.  These are spaces of functions whose distributional derivatives belong to $H^{s-1,p}(\R^d)$ for $1 < p < d$ and $s \in \N^+$, and they arise naturally in our analysis of the free surface function in \eqref{final nonlinear equations}.  Their properties will play a crucial role in our subsequent PDE analysis.  

Throughout this section, we have phrased the results in a more general manner than what is precisely needed in Sections~\ref{section on linear analysis in the mized type spaces} and~\ref{section on nonlinear analysis}, as we believe that the analysis here may be of independent interest.

% _+__+_ -_+__+_ -_+__+_ -_+__+_ -_+__+_ -_+__+_ -_+__+_ -_+__+_ -_+__+_ -_+__+_ -_+__+_ -_+__+_ -_+__+_ -
\subsection{Mixed-type Sobolev spaces}\label{appendix on properties of mixed-type Sobolev spaces}
% _+__+_ -_+__+_ -_+__+_ -_+__+_ -_+__+_ -_+__+_ -_+__+_ -_+__+_ -_+__+_ -_+__+_ -_+__+_ -_+__+_ -_+__+_ -

Throughout this subsection we consider a generic  open interval $I\subseteq\R$ and define the set $U=\R^d\times I$ for $d \in \N^+$.   

\begin{defn}[Mixed-type Sobolev spaces]\label{mixed-type sobolev def}
Let $1<p < \infty$ and $V$ be a finite dimensional normed space  over $\F \in \{\R,\C\}$.  Let  $U=\R^d\times I$ for $d \in \N^+$.
\begin{enumerate}
    \item We define the mixed type Lebesgue space
\begin{equation}
    L_{p,2}(U;V)=L^p(\R^d;L^2(I;V)),\quad \tnorm{f}_{L_{p,2}}=\bp{\int_{\R^d}\bp{\int_I\tnorm{f(x,y)}_V^2\;\m{d}y}^{p/2}\;\m{d}x}^{1/p},
\end{equation}
which is a Banach space when endowed with the obvious norm.  Moreover, the Fubini-Tonelli theorem shows that $L_{p,2}(U;V) \hookrightarrow L^{\min\{2,p\}}_{\loc}(U;V)$. We model Sobolev spaces on these mixed Lebesgue-spaces in the natural way.
 \item For $s\in\N$ we define
\begin{equation}
    H^s_{p,2}(U;V)=\tcb{f\in L_{p,2}(U;V)\;:\; \pd^\al f\in L_{p,2}(U;V) \text{ for all } \al\in\N^3 \text{ with }|\al|\le s },
\end{equation}
and endow this space with the norm
\begin{equation}
    \norm{f}_{H^s_{p,2}} = \bp{ \sum_{\abs{\alpha} \le s} \tnorm{\partial^\alpha f}_{L_{p,2}}^p }^{1/p}.
\end{equation}
Minor variants of the usual Sobolev-theoretic arguments apply to show that these spaces are Banach and that the restrictions to $U$ of elements of $C^\infty_c(\R^{d+1};V)$ form a dense subspace.  When we write $H^s_{p,2}(U)$ the understanding is that $V = \R$.
\end{enumerate}

\end{defn}

Our first lemma about these spaces provides a useful equivalent norm obtained via factorization.

\begin{lem}[Equivalent norm on the mixed type Sobolev spaces]\label{lemma on equivalent norm on the mixed type spaces}
    Let $s \in \N$, $1 < p < \infty$, and $V$ be a finite dimensional normed space over $\F \in \{\R,\C\}$.  Then we have that
    \begin{equation}\label{the identity of funny Sobolev spaces}
        H^s_{p,2}(U;V)=L^p(\R^d;H^s(I;V))\cap H^{s,p}(\R^d;L^2(I;V))
    \end{equation}
    with norm equivalence
\begin{equation}
    \norm{f}_{H^s_{p,2}} \asymp \norm{f}_{L^p H^{s}} + \norm{f}_{H^{s,p} L^2}.
\end{equation}    
\end{lem}
\begin{proof}
    Since we can always complexify a real normed space $V$, it suffices to prove the result when $\F =\C$.  Assume this. The continuous embedding of the space of the left of~\eqref{the identity of funny Sobolev spaces} into the right is obvious;  the reverse inclusion requires work. Suppose that $f$ belongs to the space on the right and $k\in\tcb{1,\dots,s-1}$.    Let $\varphi \in C^\infty_c(\R^d)$ be a generator for a homogeneous Littlewood-Paley partition of unity.  We then use Gagliardo-Nirenberg interpolation on the space $H^k(I;V)\emb H^s(I;V)$, together with  Young's inequality for products, namely $a^{1-k/s}b^{k/s}\lesssim a+b$,  the triangle inequality, and Theorem \ref{thm on annular littlewood paley, I} to estimate
    \begin{multline}
     \bnorm{\bp{\sum_{j\in\Z}\tbr{2^{j }}^{2(s-k))} \tnorm{\varphi(D/2^j)f}^2_{H^k}}^{1/2}}_{L^p}\\\lesssim\bnorm{\bp{\sum_{j\in\Z}\tbr{2^{j}}^{2(s-k)}\tnorm{\varphi(D/2^j)f}^{2(1-k/s)}_{L^2}\tnorm{\varphi(D/2^j)f}_{H^s}^{2s/k}}^{1/2}}_{L^p} \\
     \lesssim 
     \bnorm{\bp{\sum_{j\in\Z}\tbr{2^{j}}^{2s}\tnorm{\varphi(D/2^j)f}^{2}_{L^2}}^{1/2}}_{L^p}
     +
     \bnorm{\bp{\sum_{j\in\Z}  \tnorm{\varphi(D/2^j)f}_{H^s}^{2}}^{1/2}}_{L^p}
     \lesssim \tnorm{f}_{H^{s,p}L^2} + \tnorm{f}_{L^pH^s}.
    \end{multline}
    In turn, we may use  Theorem~\ref{thm on annular littlewood paley, II} to conclude that $f = \sum_{j \in \Z} \varphi(D/2^j)f$ and
    \begin{equation}
        \tnorm{f}_{H^{s-k,p}H^k}\lesssim\bnorm{\bp{\sum_{j\in\Z}\tbr{2^{j}}^{2(s-k))}\tnorm{\varphi(D/2^j)f}^2_{H^k}}^{1/2}}_{L^p}
        \lesssim  \tnorm{f}_{H^{s,p}L^2} + \tnorm{f}_{L^pH^s} .
    \end{equation}
    Summing over $k \in \{1,\dotsc,s-1\}$ then  completes the proof of the reverse embedding in~\eqref{the identity of funny Sobolev spaces}.
\end{proof}

Next we turn our attention to extension operators.

\begin{prop}[Extensions on mixed-type Sobolev spaces]\label{proposition on stein extensions}
    Let $s \in \N$, $1 < p < \infty$, and $V$ be a finite dimensional normed space over $\F \in \{\R,\C\}$.    There exists a bounded linear extension operator
    \begin{equation}\label{mapping of the extension operator}
        \mathfrak{E}_U:H^s_{p,2}(U;V)\to H^s_{p,2}(\R^{d+1};V)
    \end{equation}
    such that $\mathfrak{R}_{U}\mathfrak{E}_U=\m{id}$ on $H^s_{p,2}(U)$, where $\mathfrak{R}_U$ denotes the restriction operator.
\end{prop}
\begin{proof}
    Since $V$ is assumed to be finite dimensional over $\F \in \{\R,\C\}$, we may use a basis to reduce to the case $V = \F$.     The notion of a Stein extension operator is given in Section 3.1 in Chapter VI of Stein~\cite{MR0290095}. For the domain $U=\R^d\times I$, one can select a Stein-extension operator that is tangentially translation commuting. In fact, we only need the Stein-extension operator $I\to\R$ and view it as acting on functions on $\R^d\times I$ via carrying along the first $d$-variables as parameters. It is then immediate that
    \begin{equation}
        \mathfrak{E}_U:L^p(\R^d;H^s(I;\F))\to L^p(\R^d;H^s(\R;\F))
    \end{equation}
    is a bounded linear map.  From tangential translation invariance, we get
    \begin{equation}
        \tnorm{\mathfrak{E}_Uf}_{H^{s,p}L^2}=\tnorm{\mathfrak{E}_{U}\tbr{D}^sf}_{L^pL^2}\lesssim\tnorm{\tnorm{\tbr{D}^sf}_{L^2(I)}}_{L^p(\R^d)}=\tnorm{f}_{H^{s,p}L^2} \text{ for } f\in H^{s,p}(\R^d;L^2(I;\F)),
    \end{equation}
    and hence $\mathfrak{E}_U:H^{s,p}(\R^d;L^2(I;\F))\to H^{s,p}(\R^d;L^2(\R;\F))$ is bounded. Thus, \eqref{mapping of the extension operator} follows from Lemma~\ref{lemma on equivalent norm on the mixed type spaces}.
\end{proof}

Next, we discuss traces.  Note that in this result the regularity loss caused by taking a trace is $1/2$ rather than $1/p$.  This is due to the $L^2-$based regularity spaces used in the `normal' direction.

\begin{prop}[Traces of mixed-type Sobolev spaces]\label{prop on traces of mixed type Sobolev spaces}
    Let $s \in \N^+$ and $1<p< \infty$, and $V$ be a finite dimensional normed space over $\F \in \{\R,\C\}$.  For $b \in \R^+$ define $\Omega = \R^d \times (0,b)$ and $\Sigma = \R^d \times \{b\}$.  Then there exists a bounded and linear trace map
    \begin{equation}
        \m{Tr}_{\Sigma}:H^{s}_{p,2}(\Omega;V)\to H^{s-1/2,p}(\Sigma;V).
    \end{equation}
\end{prop}
\begin{proof}
    Using a basis of $V$, we reduce to proving the result with $V= \F$.      The key observation is the following interpolation inequality for functions $\phi\in H^{s}((0,b))$:
    \begin{equation}
        |\phi(b)|\lesssim\tnorm{\phi}_{L^2}^{1-1/2s}\tnorm{\phi}_{H^s}^{1/2s},
    \end{equation}
    where the implicit constant depends on $b$ and $s$.    A proof may be found, for instance, in Lemmas 4.9 and 4.10 of Constantin and Foias~\cite{MR972259}.

    Let $\varphi \in C^\infty_c(\R^d)$ generate a homogeneous Littlewood-Paley partition of unity.  Then, given $f\in H^{s}_{p,2}(\Omega;\F)$ we use  the above interpolation inequality with Young's inequality, namely $a^{1-1/2s}b^{1/2s}\lesssim a+b$, and Theorem~\ref{thm on annular littlewood paley, I} to bound
    \begin{multline}
    \bnorm{\bp{\sum_{j\in\Z}\tbr{2^{j}}^{2s -1}|\varphi(D/2^j)\m{Tr}_{\Sigma}f|^2}^{1/2}}_{L^p} 
        \lesssim \bnorm{\bp{\sum_{j\in\Z}\tbr{2^{j}}^{2s-1} \tnorm{\varphi(D/2^j)f}_{L^2}^{2-1/s}\tnorm{\varphi(D/2^j)f}_{H^s}^{1/s}}^{1/2}}_{L^p} \\
        \lesssim
        \bnorm{\bp{\sum_{j\in\Z}\tbr{2^{j}}^{2s} \tnorm{\varphi(D/2^j)f}_{L^2}^{2}}^{1/2}}_{L^p}
        +
        \bnorm{\bp{\sum_{j\in\Z} \tnorm{\varphi(D/2^j)f}_{H^s}^{2}}^{1/2}}_{L^p} 
        \lesssim\tnorm{f}_{L^pH^s}+\tnorm{f}_{H^{s,p}L^2}.
    \end{multline}
    This, Lemma~\ref{lemma on equivalent norm on the mixed type spaces}, and Theorem~\ref{thm on annular littlewood paley, II} then provide the estimate
    \begin{equation}
        \tnorm{\m{Tr}_\Sigma f}_{H^{s-1/2,p}}\lesssim\tnorm{f}_{L^pH^s}+\tnorm{f}_{H^{s,p}L^2} \lesssim \norm{f}_{H^s_{p,2}},
    \end{equation}
    so the trace operator is bounded as stated.
\end{proof}

Now we discuss lifting maps that complement the trace map.

\begin{prop}[Lifting in mixed-type Sobolev spaces]\label{prop on lifting mixed type Sobolev spaces}
    Let $s \in \N^+$, and $1 < p < \infty$, and $V$ be a finite dimensional normed space over $\F\in \{\R,\C\}$.  For $b \in \R^+$ let $\Omega=\R^d \times (0,b)$ and $\Sigma = \R^d \times \{b\}$.   There exists a bounded linear extension map $\m{L}_\Omega:H^{s-1/2,p}(\Sigma;V)\to H^{s}_{2,p}(\Omega;V)$ such that  $\m{Tr}_{\Sigma}\m{L}_{\Omega}=\m{id}_{H^{s-1/2,p}(\Sigma)}$.
\end{prop}
\begin{proof}
    We will prove the result when $V = \C$; the general case can be deduced from this.       Given $f\in\mathscr{S}(\Sigma;\C)$, we define $\m{L}_\Omega f : \Omega \to \C$ via
    \begin{equation}
        \mathscr{F}[\m{L}_\Omega f](\xi,y)=\exp(\tbr{\xi}(y-b))\mathscr{F}[f](\xi) \text{ for } \xi\in\R^d \text{ and } y\in(0,b).
    \end{equation}
    
    Now define $m \in C^\infty(\Omega;\C)$ via $m(\xi,y) = \exp(\tbr{\xi}(y -b))$.  Thanks to the Leibniz rule and Fa\`a di Bruno's formula, we have that
\begin{multline}
    |\pd_y^j D^k_\xi m(\xi,y)|=|D^k_\xi\tp{\tbr{\xi}^j\exp(\tbr{\xi}(y-b))}|\lesssim\sum_{i=0}^k\tbr{\xi}^{j-k+i}|D^i_\xi(\exp(\tbr{\xi}(y-b)))|\\
    \lesssim\sum_{i=0}^k\tbr{\xi}^{j-k+i}\sum_{\ell=1}^i\tbr{\xi}^{-i+\ell}|y-b|^\ell\exp(\tbr{\xi}(y-b))\lesssim\tbr{\xi}^{j-k}\max_{1\le \ell\le k }|y-b|^\ell\tbr{\xi}^\ell\exp(\tbr{\xi}(y-b)).
\end{multline}
We readily deduce from this that for any $\alpha \in \N^d$ there exists a constant $C_\alpha >0$ such that
\begin{equation}\label{curry bounds}
    \sup_{\xi\in\R^d}\tbr{\xi}^{|\al|+1/2}\tnorm{\pd^\al_\xi m(\xi,\cdot)}_{L^2}
    +
    \sup_{\xi\in\R^d}\tbr{\xi}^{|\al|+1/2-s}\tnorm{\pd^\al_\xi m(\xi,\cdot)}_{H^s}\le C_{\al}.
\end{equation}

Let $\ell \in \{0,s\}$.  Employing the canonical isometric identification $H^\ell((0,b);\C) = \mathcal{L}(\C;H^\ell((0,b);\C))$, we define the vector-valued Fourier multiplier with symbol $\mu \in C^\infty(\R^d; \mathcal{L}(\C;H^\ell((0,b);\C)))$ given by $\mu(\xi) = m(\xi,\cdot)$.  In light of \eqref{curry bounds}, we can then invoke Theorem~\ref{second HM multiplier theorem} twice to deduce that
\begin{equation}
    \mu(D):H^{s-1/2,p}(\Sigma;\C)\to H^{s,p}(\R^d;L^2((0,b);\C))\cap L^p(\R^d;H^s((0,b);\C)) = H^s_{2,p}(\Omega;\C)
\end{equation}
is a bounded linear map (in the last equality we have employed Lemma~\ref{lemma on equivalent norm on the mixed type spaces}).  To conclude, we simply note that for  $f\in\mathscr{S}(\Sigma;\C)$ we have that $\mu(D) f(x) = \m{L}_\Omega f(x,\cdot)$ for all $x \in \R^d$.
\end{proof}

Finally, we record a mixed-type Sobolev spaces variant of Proposition C.1 in Stevenson and Tice~\cite{stevenson2023wellposedness}.

\begin{prop}[Divergence compatibility in mixed-type Sobolev spaces]\label{the divergence compatibility condition is here for you to check it out}
    Let $1<p<\infty$, $b \in \R^+$, and $\Omega=\R^d\times(0,b)$.  Then there exists a constant $C$ such that for all $u\in H^1_{p,2}(\Omega;\F^{d+1})$ we have the estimate
    \begin{equation}
        \bsb{\int_0^b(\grad\cdot u)(\cdot,y)\;\m{d}y-\m{Tr}_{\Sigma}u\cdot e_{d+1}+\m{Tr}_{\Sigma_0}u\cdot e_{d+1}}_{\dot{H}^{-1,p}}\le C\tnorm{u}_{L_{p,2}}.
    \end{equation}
\end{prop}
\begin{proof}
By density, it suffices to consider the case that $0\not\in\m{supp}\mathscr{F}[u]$. Thanks to the fundamental theorem of calculus, we have
    \begin{equation}
        \int_0^b(\grad\cdot u)(\cdot,y)\;\m{d}y-\m{Tr}_{\Sigma}u\cdot e_{d+1}+\m{Tr}_{\Sigma_0}u\cdot e_{d+1}=(\pd_1,\dots,\pd_d,0)\cdot\int_0^b u(\cdot,y)\;\m{d}y
    \end{equation}
    Applying $|D|^{-1}$ and using the definition of $\dot{H}^{-1,p}$ from~\eqref{definition of the negative homogeneous Sobolev space} and~\eqref{definition of the negative homogeneous Sobolev space, 2} along with the boundedness of Riesz transforms yields
    \begin{equation}\label{sigma grindset}
        \bsb{\int_0^b(\grad\cdot u)(\cdot,y)\;\m{d}y-\m{Tr}_{\Sigma}u\cdot e_{d+1}+\m{Tr}_{\Sigma_0}u\cdot e_{d+1}}_{\dot{H}^{-1,p}}\lesssim\bnorm{\int_0^bu(\cdot,y)\;\m{d}y}_{L^p}.
    \end{equation}
    We conclude after noting that the embedding $L^2((0,b))\emb L^1((0,b))$ allows us to bound the right hand side of~\eqref{sigma grindset} by $\tnorm{u}_{L_{p,2}}$.
\end{proof}

% _+__+_ -_+__+_ -_+__+_ -_+__+_ -_+__+_ -_+__+_ -_+__+_ -_+__+_ -_+__+_ -_+__+_ -_+__+_ -_+__+_ -_+__+_ -
\subsection{Some nonlinear analysis in mixed-type Sobolev spaces}\label{appendix on some nonlinear analysis in mized type Sobolev spaces}
% _+__+_ -_+__+_ -_+__+_ -_+__+_ -_+__+_ -_+__+_ -_+__+_ -_+__+_ -_+__+_ -_+__+_ -_+__+_ -_+__+_ -_+__+_ -

The goal of this subsection is to record a series of useful results related to the  nonlinear use of mixed-type spaces. As in Section \ref{appendix on properties of mixed-type Sobolev spaces}, we will let $I\subseteq\R$ be an open interval and set $U=\R^d\times I$. 

Our first result gives a product estimates for the mixed-type Sobolev spaces.

\begin{lem}[Product estimates in mixed type Sobolev spaces]\label{product estimates in mixed type Sobolev spaces}
    Suppose $s \in \N^+$, $1<p < \infty$, and $V$ is a finite dimensional normed space over $\F \in \{\R,\C\}$.  Then we have the estimate  
    \begin{equation}
        \tnorm{fg}_{H^s_{p,2}}\lesssim\tnorm{f}_{L^\infty\cap L_{p,2}}\tnorm{g}_{H^s_{p,2}}+\tnorm{f}_{H^s_{p,2}}\tnorm{g}_{L^\infty\cap L_{p,2}}
    \end{equation}
    for all $f\in H^s_{p,2}(U;\F)\cap L^\infty(U;\F)$ and $g\in H^s_{p,2}(U;V)\cap L^\infty(U;V)$.
\end{lem}
\begin{proof}
    It suffices to prove the result for $\F = \C$, so we will assume this in the proof. 
    To check the product belongs to the correct space, we will use the norm from Lemma~\ref{lemma on equivalent norm on the mixed type spaces}.  We first recall the well-known (see e.g. Theorem D.6 in~\cite{stevenson2023wellposedness}) high-low product estimate
    \begin{equation}
        \norm{F G}_{H^s(I;V)} \lesssim   \norm{F}_{L^\infty(I;\F)}   \norm{ G}_{H^s(I;V)} + \norm{F}_{H^s(I;\F)}   \norm{ G}_{L^\infty(I;V)}
    \end{equation}
    for all $F \in H^s(I;\F) \cap L^\infty(I;\F)$ and $G \in  H^s(I;V) \cap L^\infty(I;V)$.  Applying this almost everywhere in $\R^d$ and integrating, we then derive the bound
    \begin{equation}
        \tnorm{fg}_{L^pH^s}
        % \le \tnorm{\tnorm{f}_{L^\infty(0,b)}\tnorm{g}_{H^s(0,b)}}_{L^p}+\tnorm{\tnorm{f}_{H^s(0,b)}\tnorm{g}_{L^\infty(0,b)}}_{L^p}
        \lesssim\tnorm{f}_{L^\infty}\tnorm{g}_{L^pH^s}+\tnorm{f}_{L^pH^s}\tnorm{g}_{L^\infty}.
    \end{equation}
    The bounds for tangential derivatives are more involved.  To handle them we let $\varphi \in C^\infty_c(\R^d)$ be generator for a homogeneous Littlewood-Paley partition of unity, we set $\Phi=\sum_{j\le 0}\varphi(\cdot/2^j)$, and we introduce the  (tangential, homogeneous) paraproduct decomposition
    \begin{multline}
        fg=\sum_{j\in\Z}\varphi(D/2^j)f\;\Phi(D/2^{j-3})g+\sum_{j\in\Z}\sum_{k=j-2}^{j+2}\varphi(D/2^j)f\;\varphi(D/2^k)g+\sum_{k\in\Z}\Phi(D/2^{k-3})f\;\varphi(D/2^k)g\\=\pi_{\m{hl}}(f,g)+\pi_{\m{hh}}(f,g)+\pi_{\m{lh}}(f,g).
    \end{multline}

    For the $\pi_{\m{hl}}$ term, we use the annular Littlewood-Paley estimates from Theorems~\ref{thm on annular littlewood paley, I} and~\ref{thm on annular littlewood paley, II} together with the bound 
    \begin{equation}
        \sup_{j\in\N}\tnorm{\Phi(D/2^j)g}_{L^\infty(\Omega)}\lesssim\tnorm{g}_{L^\infty},
    \end{equation}
    which follows from Young's convolution inequality, in order to estimate
    \begin{equation}
        \tnorm{\pi_{\m{hl}}(f,g)}_{H^{s,p}L^2}\lesssim\bnorm{\bp{\sum_{j\in\Z} \tbr{2^{j}}^{2s}\tnorm{\varphi(D/2^j)f\;\Phi(D/2^{j-3})g}^2_{L^2}}^{1/2}}_{L^p}\lesssim\tnorm{f}_{H^{s,p}L^2}\tnorm{g}_{L^\infty}.
    \end{equation}
    By an entirely symmetric argument, we have that $\tnorm{\pi_{\m{lh}}(f,g)}_{H^{s,p}L^2}\lesssim\tnorm{f}_{L^\infty}\tnorm{g}_{H^{s,p}L^2}$. For the remaining high-high paraproduct, we split again:
    \begin{equation}
        \pi_{\m{hh}}(f,g)=\bp{\sum_{j\in\N}\sum_{k=j-2}^{j+2}+\sum_{j<0}\sum_{k=j-2}^{j+2}}\varphi(D/2^j)f\varphi(D/2^k)g=\pi_{\m{hhh}}(f,g)+\pi_{\m{hhl}}(f,g).
    \end{equation}
    
    We now use the ball Littlewood-Paley estimate of Theorem~\ref{thm on ball littlewood paley} to handle $\pi_{\m{hhh}}$:
    \begin{equation}
        \tnorm{\pi_{\m{hhh}}(f,g)}_{H^{s,p}L^2}\lesssim\sum_{|c|\le 2}\bnorm{\bp{\sum_{j=0}^\infty 4^{sj}\tnorm{\varphi(D/2^j)f\;\varphi(D/2^{j+c})g}^2_{L^2}}^{1/2}}_{L^p} 
        \lesssim \tnorm{f}_{H^{s,p}L^2}\tnorm{g}_{L^\infty},
    \end{equation}
    where in the final inequality we have pulled $g$ out in $L^\infty(U;V)$ and estimated $f$ in $H^{s,p}(\R^d;L^2(I;\F))$ as above.      To handle $\pi_{\m{hhl}}$ we first note that it is band limited; indeed, it is supported in $B(0,16)$ by construction.  Hence the $L^p L^2$ norm controls $H^{s,p}L^2$:
    \begin{equation}
        \tnorm{\pi_{\m{hhl}}(f,g)}_{H^{s,p}L^2}\lesssim\tnorm{\pi_{\m{hhl}}(f,g)}_{L^pL^2}\le\sum_{|c|\le 2}\bnorm{\sum_{j<0}\varphi(D/2^j)f\varphi(D/2^{j+c})g}_{L^pL^2}.
    \end{equation}
    For the final term above, we then use Cauchy-Schwartz, H\"older, and Theorem~\ref{thm on annular littlewood paley, I}:
    \begin{multline}
        \bnorm{\sum_{j<0}\varphi(D/2^j)f\varphi(D/2^{j+c})g}_{L^pL^2}\lesssim\bnorm{\bp{\sum_{j<0}\tnorm{\varphi(D/2^j)f}^2_{H^1}}^{1/2}\bp{\sum_{j<0}\tnorm{\varphi(D/2^{j+c})g}_{L^2}^2}^{1/2}}_{L^p}\\
        \lesssim\bnorm{\bp{\sum_{j<0}\tnorm{\varphi(D/2^j)\Phi(D/2^3)f}_{H^1}^2}^{1/2}}_{L^{2p}}\bnorm{\bp{\sum_{j<0}\tnorm{\varphi(D/2^j)\Phi(D/2^{3+c})g}_{L^2}^2}^{1/2}}_{L^{2p}}\\\lesssim\tnorm{\Phi(D/2^3)f}_{L^{2p}H^1}\tnorm{\Phi(D/2^{3+c})g}_{L^{2p}L^2}.
    \end{multline}
    Now we invoke Young's inequality and the fact that $\Phi(D/2^3)f$ and $\Phi(D/2^{3+c})g$ are given via (tangential) convolution with a Schwartz function to bound
    \begin{equation}
        \tnorm{\Phi(D/2^3)f}_{L^{2p}H^1}\lesssim\tnorm{f}_{L^pH^1}
        \text{ and }
        \tnorm{\Phi(D/2^{3+c})g}_{L^{2p}L^2}\lesssim\tnorm{g}_{L^pL^2}.
    \end{equation}
    Upon synthesizing these estimates, we arrive at
    \begin{equation}
        \tnorm{\pi_{\m{hhl}}(f,g)}_{H^{s,p}L^2}\lesssim\tnorm{f}_{L^pH^1}\tnorm{g}_{L^pL^2}.
    \end{equation}
    Since $\N\ni s\ge 1$, we see that this completes the proof.
\end{proof}

As a consequence of the previous result and a supercritical embedding, we find sufficient conditions under which the mixed-type Sobolev spaces are an algebra.

\begin{prop}[Supercritical Sobolev embeddings in mixed-type Sobolev spaces]\label{the embedding into the Linfty space}
    The following hold.
    \begin{enumerate}
        \item $H^s_{p,2}(U;\F)\emb C^k_0(U;\F)$ for $s>k+(d+1)/\min\tcb{2,p}$, $k\in\N$.
        \item For $s>(d+1)/\min\tcb{2,p}$, the mixed type Sobolev $H^{s}_{p,2}(U;\F)$ space is an algebra.
    \end{enumerate}
\end{prop}
\begin{proof}
    For the first item we note that the restriction $H^{s}_{p,2}(U;\F)\to W^{s,\min\tcb{2,p}}(\R^d\times J;\F)$ is continuous, for every $J\subset I$ of finite length and hence the claim follows from standard Sobolev embeddings. The second item follows from the first and Lemma~\ref{product estimates in mixed type Sobolev spaces}.
\end{proof}

We now rapidly record three useful consequences of the previous results. 

\begin{prop}[Products, I]\label{proposition on smoothness of a certain product map}
    Suppose that $1<p<\infty$ and $\N\ni s>(d+1)/\min\tcb{2,p}$. Then the pointwise product map
    \begin{equation}
        H^s_{p,2}(U)\times H^s_{p,2}(U)\times W^{s,\infty}(U)\ni(f,g,h)\mapsto f\cdot (g+h)\in H^s_{p,2}(U)
    \end{equation}
    is well-defined and smooth.
\end{prop}
\begin{proof}
    That $H^s_{p,2}(U)\times W^{s,\infty}(U)\ni (f,h)\mapsto fh\in H^s_{p,2}(U)$ is well-defined and smooth is clear from bilinearity and the Leibniz rule. On the other hand, that these properties as also true for the assignment $H^s_{p,2}(U)\times H^s_{p,2}(U)\ni (f,g)\mapsto fg\in H^s_{p,2}(U)$ is a consequence of Lemma~\ref{product estimates in mixed type Sobolev spaces} and Proposition~\ref{the embedding into the Linfty space}.
\end{proof}

\begin{prop}[Products, II]\label{corollary Banach algebra of the sum space}
    Suppose that $1<p<\infty$ and $\N\ni s>(d+1)/\min\tcb{2,p}$. Then the Banach sum space
    \begin{equation}\label{space}
        (H^s_{p,2}+W^{s,\infty})(U)=\tcb{f\in L^{1}_{\loc}(U)\;:\;f=f_0+f_1,\;f_0\in H^s_{p,2}(U),\; f_1\in W^{s,\infty}(U)},
    \end{equation}
    equipped with the norm
    \begin{equation}\label{norm}
    \tnorm{f}_{H^s_{2,p}+W^{s,\infty}}=\inf\tcb{\tnorm{f_0}_{H^s_{2,p}}+\tnorm{f_1}_{W^{s,\infty}}\;:\;f=f_0+f_1},
    \end{equation}
    is a Banach algebra under pointwise multiplication.
\end{prop}
\begin{proof}
    That this space is Banach is straightforward to see, so we only prove that it is an algebra. Let $f,g\in(H^s_{p,2}+W^{s,\infty})(U)$ and decompose $f=f_0+f_1$, $g=g_0+g_1$ with $f_0,g_0\in H^{s}_{p,2}(U)$ and $f_1,g_1\in W^{s,\infty}(U)$. Then we use that $W^{s,\infty}(U)$ is an algebra along with Proposition~\ref{proposition on smoothness of a certain product map} to estimate
    \begin{equation}
        \tnorm{fg}_{H^s_{p,2}+W^{s,\infty}}\le\tnorm{f_0g_0+f_1g_0+f_0g_1}_{H^{s}_{p,2}}+\tnorm{f_1g_1}_{W^{s,\infty}}\lesssim\tp{\tnorm{f_0}_{H^s_{p,2}}+\tnorm{f_1}_{W^{s,\infty}}}\tp{\tnorm{g_0}_{H^s_{p,2}}+\tnorm{g_1}_{W^{s,\infty}}}.
    \end{equation}
    The result follows by taking the infimum over all decompositions of $f$ and $g$.
\end{proof}

\begin{rmk}[Products, III]\label{remark on products 3}
    For $1<p<\infty$ and $\N\ni s>(d+1)/\min\tcb{2,p}$, we have that the pointwise product map
    \begin{equation}
        H^s_{p,2}(U)\times(H^{s}_{p,2}+W^{s,\infty})( U)\ni(f,g)\mapsto fg\in H^s_{p,2}(U)
    \end{equation}
    is smooth. This follows directly from Proposition~\ref{proposition on smoothness of a certain product map} and definitions~\eqref{space} and~\eqref{norm}.
\end{rmk}

Our next result is meant to handle the reciprocal Jacobian of the flattening map.

\begin{coro}[Smoothness of pointwise inversion]\label{smoothness of inversion}
    For $1<p<\infty$ and $\N\ni s>(d+1)/\min\tcb{2,p}$, we have that there exists a constant $\rho\in\R^+$, depending on $U$, $s$, $d$, and $p$, such that the map
    \begin{equation}
        (H^s_{p,2}+W^{s,\infty})(U)\supset B(0,\rho)\ni f\mapsto (1+f)^{-1}\in (H^s_{p,2}+W^{s,\infty})(U)
    \end{equation}
    is well-defined and smooth.
\end{coro}
\begin{proof}
    Proposition~\ref{corollary Banach algebra of the sum space} established that $(H^s_{p,2}+W^{s,\infty})(U)$ is a Banach algebra.  Consequently, we can use the usual theory of power series on Banach algebras to pick $\rho$ sufficiently small such that the power series
    \begin{equation}
        B(0,\rho)\ni f\mapsto \sum_{j=0}^\infty(-1)^jf^j
    \end{equation}
    is uniformly absolutely convergent and defines an analytic map.  It is then elementary to verify that this power series is pointwise equal to $f\mapsto (1+f)^{-1}$.
\end{proof}

Now we study the composition appearing in the interaction between the flattening map and the data. The following is a modification of the main argument presented in Inci, Kappeler, and Topalov~\cite{MR3135704}.

\begin{prop}[On composition]\label{proposition on composition}
    Assume that $2 \le d \in \N$, and let $1<p<\infty$ and $\N\ni s>1+d/\min\tcb{2,p}$. There exists a constant $\lambda\in\R^+$, depending only on $p$, $d$, and $s$ such that  the following hold.
\begin{enumerate}
    \item If $f \in B_{(H^s_{p,2}+W^{s,\infty})(\R^d)}(0,\lambda)$, then the map $\m{id}_{\R^d}+fe_d$ is a $C^m$ diffeomorphism of $\R^d$ onto itself for every $\N\ni m<s-d/\min\tcb{2,p}$.
    \item If  $k\in\N$, then the map $\Lambda$ defined by
    \begin{equation}\label{the composition map mapping properties}
        H^{s+k}_{p,2}(\R^d)\times B_{(H^s_{p,2}+W^{s,\infty})(\R^d)}(0,\lambda)\ni (F,f) \mapsto         \Lambda(F,f) = F(\m{id}_{\R^d}+fe_d)\in H^{s}_{p,2}(\R^d)
    \end{equation}
    is well-defined and  $C^k$.
\end{enumerate}
 \end{prop}
\begin{proof}
Thanks to Corollary~\ref{smoothness of inversion}, there exists a $\lambda\in\R^+$ for which the map   
\begin{equation}
    B_{(H^{s}_{p,2}+W^{s,\infty})(\R^d)}(0,\lambda)\ni f \mapsto  K_f =  (1+\pd_d f)^{-1}\in (H^{s-1}_{p,2}+W^{s-1})(\R^d)
\end{equation}
 is  analytic, and $K_f >0$ in $\R^d$ for each $f$ in the domain of the map.  In turn, from the above and the fact that $s-1>d/\min\tcb{2,p}$, we find that the map $\mathfrak{F}_f:\R^d\to\R^d$ given by $\mathfrak{F}_f=\m{id}_{\R^d}+fe_d$ is a $C^1$ diffeomorphism; indeed the condition that $1+\pd_d f$ is bounded and bounded away from zero guarantees that for every $x\in\R^d$ the map $\R\ni y\overset{\psi_{f,x}}{\mapsto} y+f(x,y)\in\R$ is a diffeomorphism $\R\to\R$. Thus $\mathfrak{F}_f^{-1}(x,y)=(x,\psi_{f,x}^{-1}(y))$, for $(x,y)\in\R^d\times\R$.
 
 That the maps $\mathfrak{F}_f$ and $\mathfrak{F}_f^{-1}$ are class $C^m$ for every $\N\ni m<s-d/\min\tcb{2,p}$ follows from Proposition~\ref{the embedding into the Linfty space} and the inverse function theorem. This completes the proof of the first item.

With $\lambda\in\R^+$ in hand, we now turn to the proof of the second item.  First, we prove~\eqref{the composition map mapping properties} in the case $k=0$ via an induction argument. For $j\in\tcb{0,1,\dots,s}$ let $\mathbb{P}_j$ denote the proposition that
\begin{equation}
    \Lambda:H^j_{p,2}(\R^d)\times B_{(H^s_{p,2}+W^{s,\infty})(\R^d)}(0,\lambda)\to H^j_{p,2}(\R^d)
\end{equation}
is well-defined and continuous.

For $\mathbb{P}_0$, we perform a change of variables $z = y + f(x,y)$ on each fiber to estimate
\begin{equation}\label{base case estimate}
    \tnorm{\Lambda(F,f)}_{H^0_{p,2}}=\bp{\int_{\R^{d-1}}\bp{\int_\R|F(x,y+f(x,y))|^2\;\m{d}y}^{p/2}\;\m{d}x}^{1/p}\le\tnorm{K_f}_{L^\infty}^{1/2}\tnorm{F}_{H^0_{p,2}}.
\end{equation}
Thus well-definedness is established. For continuity, we estimate
\begin{equation}
    \tnorm{\Lambda(F,f)-\Lambda(G,g)}_{H^{0}_{p,2}}\le\tnorm{\Lambda(F-G,f)}_{H^0_{p,2}}+\tnorm{\Lambda(G,f)-\Lambda(G,g)}_{H^0_{p,2}}.
\end{equation}
The former term is made small when $(G,g)$ are close to $(F,f)$ via estimate~\eqref{base case estimate}, while the latter term is made small by approximating $G$ via a smooth compactly supported function and using that $g\to f$ in $(H^s_{p,2}+W^{s,\infty})(\R^d)$ implies that $\mathfrak{F}_g\to \mathfrak{F}_f$ uniformly. This gives $\mathbb{P}_0$.

Now suppose that $j\in\tcb{0,\dots,s-1}$ is such that $\mathbb{P}_\ell$ is true for all $\ell\in\tcb{1,\dots,j}$. We compute
\begin{equation}
    \pd_q(\Lambda(F,f))=\Lambda(\pd_q F,f)+\Lambda(\pd_d F,f)\pd_q f.
\end{equation}
for $q\in\tcb{1,\dotsc,d}$. By combining Remark~\ref{remark on products 3} and the induction hypothesis, we have that
\begin{equation}
    H^{j+1}_{p,2}(\R^d)\times B_{(H^s_{p,2}+W^{s,\infty})(\R^d)}(0,\lambda)\ni (F,f)\mapsto \pd_q(\Lambda(F,f))\in H^j_{p,2}(\R^d)
\end{equation}
is a well-defined and continuous map. This paired with $\mathbb{P}_0$ gives $\mathbb{P}_{j+1}$. Thus the induction is complete, and we have proved~\eqref{the composition map mapping properties} in the case $k=0$.

We now turn to the proof of~\eqref{the composition map mapping properties} for the remaining cases of $k\in\N^+$.  For this we shall use the converse to Taylor's theorem: see, for instance, Theorem 2.4.15 and Supplement 2.4B in Abraham, Marsden, and Ratiu~\cite{AbMaRa_1988}. For $r\in\tcb{1,\dots,k}$, we define
\begin{equation}
    \Lambda^{(r)}:H^{s+k}_{p,2}(\R^d)\times B_{(H^s_{p,2}+W^{s,\infty})(\R^d)}(0,\lambda)\to\mathcal{L}^r_{\m{sym}}(H^s_{p,2}(\R^d)\times(H^s_{p,2}+W^{s,\infty})(\R^d);H^s_{p,2}(\R^d))
\end{equation}
via the formula
\begin{equation}\label{derivative def}
\Lambda^{(r)}(F,f)[(F_1,f_1),\dots,(F_r,f_r)]=\Lambda(\pd_d^rF,f)\prod_{\ell=1}^rf_\ell+\sum_{m=1}^r\Lambda(\pd_d^{r-1}F_m,f)\prod_{\ell\neq m}f_\ell.
\end{equation}
In the above $\mathcal{L}^r_{\m{sym}}$ refers to the space of symmetric $r$-multilinear maps. By using the already established continuity properties of $\Lambda$ along with Remark~\ref{remark on products 3}, we see that $\Lambda^{(r)}$ is continuous for $r\in\tcb{1,\dots,k}$. By similar considerations, we find that the map
\begin{equation}
    \mathcal{R}_k:\mathcal{U}\to\mathcal{L}^k_{\m{sym}}(H^s_{p,2}(\R^d)\times(H^s_{p,2}+W^{s,\infty})(\R^d);H^s_{p,2}(\R^d))
\end{equation}
defined via
\begin{equation}\label{remainder def}
    \mathcal{R}_k((F,f),(G,g))=\int_0^1\f{(1-t)^{k-1}}{(k-1)!}\tp{\Lambda^{(k)}(F+tG,f+tg)-\Lambda^{(k)}(F,f)}\;\m{d}t
\end{equation}
is continuous, where
\begin{equation}
    \mathcal{U}=\tcb{((F,f),(G,g))\in(H^{s+k}_{p,2}\times B(0,\lambda))^2\;:\;\forall\;t\in[0,1]\;(F+tG,f+tg)\in H^{s+k}_{p,2}\times B(0,\lambda)}.
\end{equation}
Let $((F,f),(G,g))\in\mathcal{U}$. Now, according to Taylor's theorem with integral remainder (used pointwise), we have that
\begin{equation}\label{expression 1}
    \Lambda(F,f+g)-\Lambda(F,f)=\sum_{r=1}^k\f{1}{r!}\Lambda(\pd_d^rF,f)g^r+\int_0^1\f{(1-t)^{k-1}}{(k-1)!}\tp{\Lambda(\pd_d^k F,f+tg)-\Lambda(\pd_d^kF,f)}g^k\;\m{d}t.
\end{equation}
We can also apply the same result to express for $h(s)=s\Lambda(G,f+sg)$
\begin{multline}\label{expression 2}
    \Lambda(G,f+g)=h(1)-h(0)=\sum_{r=1}^k\f{1}{r!}(\pd^r h)(0)+\int_0^1\f{(1-t)^{k-1}}{(k-1)!}\tp{(\pd^kh)(t)-(\pd^kh)(0)}\;\m{d}t\\
    =\sum_{r=1}^k\f{1}{(r-1)!}\Lambda(\pd_d^{r-1}G,f)g^{r-1}\\+\int_0^1\f{(1-t)^{k-1}}{(k-1)!}\tp{t\Lambda(\pd_d^kG,f+tg)g^k+k(\Lambda(\pd_d^{k-1}G,f+tg)-\Lambda(\pd_d^{k-1}G,f))g^{k-1}}\;\m{d}t.
\end{multline}
Adding~\eqref{expression 1} and~\eqref{expression 2} and using definitions~\eqref{derivative def} and~\eqref{remainder def} yields
\begin{equation}
    \Lambda(F+G,f+g)-\Lambda(F,f)=\sum_{r=1}^k\f{1}{r!}\Lambda^{(r)}(F,f)[(G,g)^{\otimes r}]+\mathcal{R}_k((F,f),(G,g))[(G,g)^{\otimes k}].
\end{equation}
Therefore, by the converse to Taylor's theorem, we deduce that $\Lambda$ is  $C^k$ on its domain.
\end{proof}

% _+__+_ -_+__+_ -_+__+_ -_+__+_ -_+__+_ -_+__+_ -_+__+_ -_+__+_ -_+__+_ -_+__+_ -_+__+_ -_+__+_ -_+__+_ -
\subsection{Subcritical gradient spaces}\label{section on properties of subcritical gradient spaces}
% _+__+_ -_+__+_ -_+__+_ -_+__+_ -_+__+_ -_+__+_ -_+__+_ -_+__+_ -_+__+_ -_+__+_ -_+__+_ -_+__+_ -_+__+_ -

We now turn our attention to subcritical gradient spaces. 

\begin{defn}[Subcritical gradient spaces]\label{defn of subcritical gradient spaces}
    For $1<p<d$, $s\in\N^+$, and $\F \in \{\R,\C\}$ we define the space
    \begin{equation}\label{definition of subcritical gradient spaces}
        \tilde{H}^{s,p}(\R^d;\F)=\tcb{f\in L^{dp/(d-p)}(\R^d;\F)\;:\;\grad f\in H^{s-1,p}(\R^d;\F^d)}
    \end{equation}
    and it equip it with a norm $\tnorm{f}_{\tilde{H}^{s,p}}=\tnorm{\grad f}_{H^{s-1,p}}$.  When $\F = \R$ we will typically abbreviate $\tilde{H}^{s,p}(\R^d) = \tilde{H}^{s,p}(\R^d;\R)$.
\end{defn}

Next, we verify that these spaces are complete.

\begin{prop}[Properties of subcritical gradient spaces]\label{prop on completeness of subcritical gradient spaces}
    Let $s\in\N^+$, $1<p<d$, and $\F \in \{\R,\C\}$.  Then the space $\tilde{H}^{s,p}(\R^d;\F)$ defined in~\eqref{definition of subcritical gradient spaces} is Banach, and we have the bound
    \begin{equation}\label{hom_sob}
    \norm{f}_{H^{s-1,dp/(d-p)}} \lesssim \norm{f}_{\tilde{H}^{s,p}} \text{ for all } f \in \tilde{H}^{s,p}(\R^d;\F).
    \end{equation}
    Moreover, $\mathscr{S}(\R^d;\F) \subset \tilde{H}^{s,p}(\R^d;\F)$ is dense. 
\end{prop}
\begin{proof}
    It suffices to prove the result when $\F =\R$, so we will assume this.   We first consider the case $s=1$ and note that in this case $\tilde{H}^{1,p}(\R^d)=\dot{W}^{1,p}(\R^d)\cap L^{dp/(d-p)}(\R^d)$ for the homogeneous Sobolev space   $\dot{W}^{1,p}(\R^d) = \tcb{f\in L^{1}_{\text{loc}}(\R^d) \;:\; \grad f\in L^{p}(\R^d;\R^d)}$.  Next, we note that a density result of Haj\l asz and Ka\l amajska~\cite{MR1315521} shows that if $f \in \dot{W}^{1,p}(\R^d)$, then  there exists a sequence $\{f_n\}_{n\in \N} \subset C^\infty_c(\R^d)$ such that $\norm{\nabla f_n - \nabla f}_{L^p} \to 0$ as $n \to \infty$.  Additionally, Theorem 12.9 in Leoni~\cite{MR3726909}, which is crucially based on this density assertion, provides a constant $C_p >0$ such that for each $f \in \dot{W}^{1,p}(\R^d)$ there exist a unique constant $c_f \in \R$ such that $\norm{f-c_f}_{L^{dp/(d-p)}} \le C_p \norm{\grad f}_{L^p}$.     Now, if $f \in \tilde{H}^{1,p}(\R^d)$ then upon applying this bound to $f$ and noting that $f \in L^{dp/(d-p)}(\R^d)$, we deduce that $c_{f} = 0$.  Hence, \eqref{hom_sob} holds when $s=1$, and with this bound in hand it is a routine matter to verify that $\tilde{H}^{1,p}(\R^d)$ is complete.  The above density assertion shows that $\mathscr{S}(\R^n;\R)$ is dense in $\tilde{H}^{1,p}(\R^d)$ .   To complete the proof for general $s \in \N^+$ we simply note that the map $\tbr{D}^{s-1}:\tilde{H}^{s,p}(\R^d)\to\tilde{H}^{1,p}(\R^d)$ is an isometric isomorphism that maps $\mathscr{S}(\R^n;\R)$ to itself.
\end{proof}

Next we record a frequency splitting result.

\begin{lem}[Frequency splitting in gradient spaces]\label{lemma on frequency splitting in gradient spaces}
    Let $s\in\N^+$, $1<p<d$, and $\F \in \{\R,\C\}$.     Suppose that $\varphi\in C^\infty_c(\R^d)$ is an even function that satisfies $\varphi=1$ on $B(0,1)$ and $\supp(\varphi) \subseteq B(0,2)$. Then  $(1-\varphi)(D):\tilde{H}^{s,p}(\R^d;\F)\to H^{s,p}(\R^d;\F)$ is a bounded linear map, and $\varphi(D):\tilde{H}^{s,p}(\R^d;\F)\to W^{k,dp/(d-p)}(\R^d;\F)$ is a bounded linear map for every $k \in \N$.
\end{lem}
\begin{proof}
    The proof is straightforward and we only sketch it. The mapping properties for $\varphi(D)$ follow directly from the embedding $\tilde{H}^{s,p}(\R^d;\F)\emb L^{dp/(d-p)}(\R^d;\F)$ and that $\varphi(D)$ is convolution with a band-limited Schwartz function. The mapping properties of $(1-\varphi(D))$ are verified via the Littlewood-Paley characterizations from Theorems~\ref{thm on annular littlewood paley, I} and~\eqref{thm on annular littlewood paley, II}.  The evenness of $\varphi$ guarantees that real-valued maps stay real-valued when either $\varphi(D)$ or $(1-\varphi)(D)$ is applied.
\end{proof}

By using the lifting map of Proposition~\ref{prop on lifting mixed type Sobolev spaces}, we can build a useful extension operator.

\begin{prop}[Extension operator]\label{prop on extension operator}
    Let $s\in\N^+$, $1<p<d$, and $\F \in \{\R,\C\}$.  For $b \in \R^+$ write $\Omega = \R^d \times (0,b)$ and $\Sigma = \R^d \times \{b\}$.  There exists a bounded linear operator  $\mathcal{E}_0:\tilde{H}^{3/2+s,p}(\Sigma;\F)\to H^{2+s}_{p,2}(\Omega;\F)$ and a linear operator  $\mathcal{E}_1:\tilde{H}^{3/2+s,p}(\Sigma;\F)\to W^{k,\infty}(\Omega;\F)$ that is bounded for every $k \in \N$ such that the  linear extension operator $\mathcal{E}=\mathcal{E}_0+\mathcal{E}_1$ satisfies 
    \begin{equation}
        \m{Tr}_{\Sigma_0}\mathcal{E}\eta=0
        \text{ and }
        \m{Tr}_{\Sigma}\mathcal{E}\eta=\eta
        \text{ for all }
        \eta\in\tilde{H}^{3/2+s,p}(\Sigma;\F).
    \end{equation}
\end{prop}
\begin{proof}
    Let $\varphi\in C^\infty_c(\R^d)$ be as in Lemma~\ref{lemma on frequency splitting in gradient spaces}. We define $\mathcal{E}_0$ via 
    \begin{equation}
        (\mathcal{E}_0\eta)(x,y)=\phi(b-y)(\m{L}_\Omega(1-\varphi)(D)\eta)(x,y)
        \text{ for }
        (x,y)\in\R^d\times(0,b),
    \end{equation}
    where $\m{L}_\Omega$ is the lifting operator from Proposition~\ref{prop on lifting mixed type Sobolev spaces} and $\phi\in C^\infty_c(\R)$ satisfies $\phi(0)=1$ and $\supp(\phi) \subseteq (-b/2,b/2)$. We also define $\mathcal{E}_1$ via
    \begin{equation}
        (\mathcal{E}_1\eta)(x,y)=(y/b)(\varphi(D)\eta)(x)
        \text{ for }
        (x,y)\in\R^d\times(0,b).
    \end{equation}
    The claimed mapping properties then follow immediately from Lemma~\ref{lemma on frequency splitting in gradient spaces} and Proposition~\ref{prop on lifting mixed type Sobolev spaces}.
\end{proof}
    
% _+__+_ -_+__+_ -_+__+_ -_+__+_ -_+__+_ -_+__+_ -_+__+_ -_+__+_ -_+__+_ -_+__+_ -_+__+_ -_+__+_ -_+__+_ -
\section{Linear analysis in mixed-type Sobolev spaces}\label{section on linear analysis in the mized type spaces}
% _+__+_ -_+__+_ -_+__+_ -_+__+_ -_+__+_ -_+__+_ -_+__+_ -_+__+_ -_+__+_ -_+__+_ -_+__+_ -_+__+_ -_+__+_ -

In this penultimate section, we conclude our linear theory for systems~\eqref{curl formulation of the linearization} and~\eqref{linearization of the nonlinear problem}. In Section~\ref{subsection on existence and uniqueness}, we synthesize our multiplier analysis of Section~\ref{section on vector-valued symbol calculus for the solution map} with our generalization of the Mikhlin-H\"ormander, Theorem~\ref{second HM multiplier theorem}, and produce a linear well-posedness result for~\eqref{curl formulation of the linearization} posed in the mixed-type Sobolev spaces of Section~\ref{appendix on properties of mixed-type Sobolev spaces}. In Section~\ref{section on reformulated well-posedness} we then return to the main linear equations, system~\eqref{linearization of the nonlinear problem}, and port over the extended linear theory for~\eqref{curl formulation of the linearization}. This lands us a linear well-posedness theory for~\eqref{linearization of the nonlinear problem} that employs both the mixed-type Sobolev spaces of Section~\ref{appendix on properties of mixed-type Sobolev spaces} and the subcritical gradient spaces of Section~\ref{section on properties of subcritical gradient spaces}. Armed with this result, we will then be ready to turn to the nonlinear analysis in Section~\ref{section on nonlinear analysis}.

% _+__+_ -_+__+_ -_+__+_ -_+__+_ -_+__+_ -_+__+_ -_+__+_ -_+__+_ -_+__+_ -_+__+_ -_+__+_ -_+__+_ -_+__+_ -
\subsection{Existence and uniqueness}\label{subsection on existence and uniqueness}
% _+__+_ -_+__+_ -_+__+_ -_+__+_ -_+__+_ -_+__+_ -_+__+_ -_+__+_ -_+__+_ -_+__+_ -_+__+_ -_+__+_ -_+__+_ -

We begin by setting some notation for the spaces in which we wish to extend our existence theory. Note that the mixed-type Sobolev spaces, which were introduced in Section~\ref{appendix on properties of mixed-type Sobolev spaces}, are used here.

\begin{defn}[Mixed-type function spaces, I]\label{definition of more function spaces}
    For $s\in\N$ and $1<r<\infty$, we define the function spaces
    \begin{equation}
        \X_{s,r}=H^{1+s}_{r,2}(\Omega;\C)\times H^{2+s}_{r,2}(\Omega;\C^3)\times H^{3/2+s,r}(\Sigma;\C^2),
    \end{equation}
    \begin{equation}
        \Y_{s,r}=H^s_{r,2}(\Omega;\C^3)\times H^{1/2+s,r}(\Sigma;\C^3)\times H^{5/2+s,r}(\Sigma;\C^2).
    \end{equation}
\end{defn}

We are now ready to state and prove our main existence result of this subsection. We remark that $\Y_s\cap\Y_{s,r}$ is dense in $\Y_{s,r}$. It is in this sense we use the word `extension' in  what follows.

\begin{thm}[Extension to mixed type spaces]\label{thm on the extension to the mixed type Sobolev spaces}
    For $s\in\N$ and $1<r<\infty$ the linear map $\Psi:\Y_s\to\X_s$ from Definition~\ref{defn of spaces and TILOs} has a unique bounded extension
    \begin{equation}\label{extended mapping properties}
        \Psi:\Y_{s,r}\to\X_{s,r}.
    \end{equation}
    Moreover, the above extension retains the property that if $(p,u,\chi)=\Psi(f,k,H)$, then system~\eqref{curl formulation of the linearization} is solved by the former tuple with data $(0,f,k,\grad_{\|}\cdot H,0)$.
\end{thm}
\begin{proof}
    $\Psi$ is given by the Fourier multiplier $\bf{m}$ from Definition~\ref{defn of the main symbol}, and $\bf{m}$ satisfies Mikhlin-H\"ormander bounds on its derivatives thanks to estimate~\eqref{the estimate on the dude} from Theorem~\ref{thm on analyticity of the symbol} and the definition of $\tjump{\cdot}_s$ from Definition~\ref{defn of admissable class of symbols}.  We then apply our  Mikhlin-H\"ormander variant, Theorem~\ref{second HM multiplier theorem} (with $\mu=0$), to each of the component maps $\bf{m}_{jk}$ for $j,k \in \{1,2,3\}$  to deduce the existence of the stated bounded extension.   It is straightforward to check by density that these extensions of $\Psi$ are still solution operators to~\eqref{curl formulation of the linearization}.
\end{proof}

Our linear analysis for the system~\eqref{curl formulation of the linearization} is synthesized with the following theorem.
\begin{thm}[Well-posedness of the linearization in mixed-type Sobolev spaces, I]\label{thm on well-posedness of the linearization}
    Let $s\in\N$ and $1<r\le 2$. For every $(f,k,H)\in\Y_{s,r}$, there exists a unique $(p,u,\chi)\in\X_{s,r}$ such that~\eqref{curl formulation of the linearization} is satisfied in the strong sense with data $(0,f,k,\grad_{\|}\cdot H,0)$.
\end{thm}
\begin{proof}
    Existence follows from Theorem~\ref{thm on the extension to the mixed type Sobolev spaces}, so it remains to prove uniqueness. So suppose that $(p,u,\chi)\in\X_{s,r}$ solve~\eqref{curl formulation of the linearization} with trivial data. Let $\varphi\in C^\infty_c(\R^2)$ be such that $\varphi =1$ in $B(0,1)$. Then for every $N$ we have the inclusion $\varphi(D/N)(p,u,\chi)\in\X_{s}$ thanks to Young's inequality and the fact that $\varphi(D/N)$ is a tangential convolution with a Schwartz function; moreover, the triple $\varphi(D/N)(p,u,\chi)$ remains  a solution to~\eqref{curl formulation of the linearization} with trivial data.  We may thus invoke Theorem~\ref{thm on analysis of strong solutions, II} to deduce that $\varphi(D/N)(p,u,\chi)=0$. As this holds for every $N\in\N$, we necessarily have that $(p,u,\chi)=0$. Uniqueness is proved.
\end{proof}

% _+__+_ -_+__+_ -_+__+_ -_+__+_ -_+__+_ -_+__+_ -_+__+_ -_+__+_ -_+__+_ -_+__+_ -_+__+_ -_+__+_ -_+__+_ -
\subsection{Reformulated well-posedness}\label{section on reformulated well-posedness}
% _+__+_ -_+__+_ -_+__+_ -_+__+_ -_+__+_ -_+__+_ -_+__+_ -_+__+_ -_+__+_ -_+__+_ -_+__+_ -_+__+_ -_+__+_ -

We now aim to make the transition from system~\eqref{curl formulation of the linearization} back to the original linearization~\eqref{linearization of the nonlinear problem}. The previous subsection gave us the well-posedness of the former system in the mixed-type Sobolev spaces.  The goal of this subsection is to port these result to the latter, and specialize to $\R$-valued functions. Note that the following definition implements the  notions of subcritical gradient spaces, which were introduced in Section~\ref{section on properties of subcritical gradient spaces}.

\begin{defn}[Mixed-type function spaces, II]\label{second definition of mixed type function spaces}
    For $s\in\N$ and $1<r<2$, we define the Banach spaces
    \begin{equation}
        \bf{X}_{s,r}=H^{1+s}_{r,2}(\Omega;\R)\times H^{2+s}_{r,2}(\Omega;\R^3)\times\tilde{H}^{5/2+s,r}(\Sigma;\R),
    \end{equation}
    \begin{equation}
        \bf{Y}_{s,r}=H^s_{r,2}(\Omega;\R^3)\times H^{1/2+s,r}(\Sigma;\R^3)\times (H^{3/2+s,r}\cap\dot{H}^{-1,r})(\Sigma;\R).
    \end{equation}
    Note that the space $\dot{H}^{-1,r}(\Sigma;\R)$ is defined in~\eqref{definition of the negative homogeneous Sobolev space}.  
\end{defn}

Provided that $s\in \N$ is sufficiently large relative to $r \in (1,2)$, the spaces $\bf{X}_{s,r}$ enjoy classical regularity.
 
\begin{prop}\label{Xsr classicality}
    Let $1<r<2$ and $\N \ni s > 3/r-1$.  Then 
        \begin{equation}
             \bf{X}_{s,r}  \emb  C^{k}_0(\Omega)\times C^{1+k}_0(\Omega;\R^3)\times C^{2+k}_0(\Sigma)
        \end{equation}
        for $k=s-\tfloor{3/r}\in\N$.
\end{prop}
\begin{proof}
    The embedding of the first two factors follows from the first item of Proposition~\ref{the embedding into the Linfty space}.  For the third factor, the embedding follows from Lemma~\ref{lemma on frequency splitting in gradient spaces}, the standard embedding of Bessel-Sobolev spaces, and the observation that $\tfloor{3/r} \ge 1 + \tfloor{2/r-1/2}$ for $1 < r < 2$.
\end{proof}

We can now state our well-posedness result.  

\begin{thm}[Well-posedness of the linearization in mixed-type Sobolev spaces, II]\label{thm on well-posedness of the linearization in mixed-type Sobolev spaces, II}
    For every $s\in\N$, $1<r<2$, and $(f,k,h)\in\bf{Y}_{s,r}$ there exists a unique $(p,u,\eta)\in\bf{X}_{s,r}$ such that system~\eqref{linearization of the nonlinear problem} is solved with data $(0,f,k,h)$.
\end{thm}
\begin{proof}
    We begin by proving uniqueness. Suppose that $(p,u,\eta)\in\bf{X}_{s,r}$ solve system~\eqref{linearization of the nonlinear problem} with trivial data. Then we set $\chi=\grad_{\|}\eta$ and observe that $(p,u,\chi)\in\X_{s,r}$ solves~\eqref{curl formulation of the linearization} with trivial data. Then Theorem~\ref{thm on well-posedness of the linearization} applies, and we learn that $(p,u,\chi)=0$ and hence $\eta$ is constant. However, $\eta\in L^{2r/(2-r)}(\Sigma;\R)$, so $\eta=0$. This completes the proof of uniqueness.

    We now prove existence. Suppose that $(f,k,h)\in\bf{Y}_{s,r}$.  Using Mikhlin-H\"ormander Theorem~\ref{hormander mikhlin multiplier theorem}, we may define $H=\grad_{\|}\Delta_{\|}^{-1}h\in H^{5/2+s,r}(\Sigma;\R^2)$, which obeys the estimate $\tnorm{H}_{H^{5/2+s,r}}\lesssim\tnorm{h}_{\dot{H}^{-1,r}\cap H^{3/2+s,r}}$ as well as the identity $\grad_{\|}\cdot H=h$.  We may then use Theorem~\ref{thm on well-posedness of the linearization} to acquire  $(p,u,\chi) = \Psi(f,k,H) \in\X_{s,r}$.  Next, we define $\tilde{\eta} = |\grad_{\|}|\Delta_{\|}^{-1}\grad_{\|}\cdot\chi$ and note that $\tilde{\eta} \in H^{3/2+s,r}(\Sigma;\C)$ thanks to another application of Mikhlin-H\"ormander.  In turn, this allows us to employ the Hardy-Littlewood-Sobolev inequality (see, for instance, Theorem 1 in Section 1.2 in Chapter V of Stein~\cite{MR0290095}) to set $\eta =  |\grad_{\|}|^{-1} \tilde{\eta} \in L^{2r/(2-r)}(\Sigma;\C)$.  Since  $\grad^\perp_{\|}\cdot\chi=0$, we then have that  $\grad_{\|}\eta=\chi\in H^{3/2+s,r}(\Sigma;\C^2)$, and so $\eta\in\tilde{H}^{5/2+s,r}(\Sigma;\C)$. 

    At this point we  have established that $(p,u,\eta)$ is a solution to~\eqref{linearization of the nonlinear problem} with the correct regularity and integrability properties, but the construction we have used does not guarantee a priori that the solution is real-valued.  To see that this is actually the case, we note that since $(f,k,H)$ have vanishing imaginary part,  we can take the imaginary part of the equations and use the fact that its coefficients are all real to deduce that $(\m{Im}p,\m{Im}u,\m{Im}\chi)=\Psi(0,0,0)$.  Thus, Theorem~\ref{thm on well-posedness of the linearization}'s uniqueness assertion shows that $(\m{Im}p,\m{Im}u,\m{Im}\chi)=0$, and the existence proof is complete.
\end{proof}

% _+__+_ -_+__+_ -_+__+_ -_+__+_ -_+__+_ -_+__+_ -_+__+_ -_+__+_ -_+__+_ -_+__+_ -_+__+_ -_+__+_ -_+__+_ -
\section{Nonlinear analysis}\label{section on nonlinear analysis}
% _+__+_ -_+__+_ -_+__+_ -_+__+_ -_+__+_ -_+__+_ -_+__+_ -_+__+_ -_+__+_ -_+__+_ -_+__+_ -_+__+_ -_+__+_ -

In this section we complete the proof of our main result, Theorem~\ref{the main theorem}, by synthesizing our previous analysis and appealing to the implicit function theorem.  Section \ref{subsection on operators and mapping properties} sets up the nonlinear framework, and then our main results are proved in Section \ref{subsection on well-posedness}.

% _+__+_ -_+__+_ -_+__+_ -_+__+_ -_+__+_ -_+__+_ -_+__+_ -_+__+_ -_+__+_ -_+__+_ -_+__+_ -_+__+_ -_+__+_ -
\subsection{Operators and mapping properties}\label{subsection on operators and mapping properties}
% _+__+_ -_+__+_ -_+__+_ -_+__+_ -_+__+_ -_+__+_ -_+__+_ -_+__+_ -_+__+_ -_+__+_ -_+__+_ -_+__+_ -_+__+_ -

The goal of this subsection is to define a nonlinear operator associated with system~\eqref{final nonlinear equations} and study its mapping properties. We begin by studying the flattening map $\eta\mapsto\mathfrak{F}_\eta$, which we recall is defined in~\eqref{definition of the flattening map}.

\begin{prop}[Properties of the flattening map]\label{prop on properties of the flattening map}
    For $1<r<2$ and $\N\ni 1+s>3/r$, there exists $\varrho\in\R^+$ such that the following hold.
    \begin{enumerate}
        \item For $\eta\in B(0,\varrho)\subset\tilde{H}^{3/2+s,r}(\Sigma)$ the flattening map $\mathfrak{F}_\eta=\m{id}_{\R^3}+\mathcal{E}\eta e_3$ is a smooth diffeomorphism from $\Omega$ to $\Omega[\eta]$ that extends to a $C^n$ diffeomorphism from $\Bar{\Omega}$ to $\Bar{\Omega[\eta]}$ for $\N\ni n<s+2-3/r$.

        \item Let $V$ be a finite dimensional real normed space.  For $\ell,m\in\N$ with $m\le 2+s$ the map
        \begin{equation}\label{the mapping properties of the flattening map}
            H^{\ell+m}_{r,2}(\R^3;V)\times B_{\tilde{H}^{3/2+s,r}(\Sigma)}(0,\varrho)\ni(F,\eta)\mapsto F\circ\mathfrak{F}_\eta\in H^{m}_{r,2}(\Omega;V)
        \end{equation}
        is well-defined and  $C^\ell$.
    \end{enumerate}
\end{prop}
\begin{proof}

    Let $\mathfrak{E}_\Omega$ denote the extension operator granted by Proposition~\ref{proposition on stein extensions}, and consider the extended flattening map $\Bar{\mathfrak{F}}_\eta:\R^3\to\R^3$ defined by $\Bar{\mathfrak{F}}_\eta=\m{id}_{\R^3}+\mathfrak{E}_{\Omega}\mathcal{E}\eta e_3$.
    By the mapping properties of $\mathfrak{E}_\Omega$ and $\mathcal{E}$ from Propositions~\ref{proposition on stein extensions} and~\ref{prop on extension operator}, we find that the map $E$ defined by
    \begin{equation}
        \tilde{H}^{3/2+s,r}(\Sigma)\ni\eta\mapsto E\eta = \mathfrak{E}_\Omega\mathcal{E}\eta\in(H^{2+s}_{r,2}+W^{2+s,\infty})(\R^3)
    \end{equation}
    is bounded and linear.  As such, there exists $\varrho\in\R^+$ such that $E(B(0,\varrho))\subseteq B(0,\lambda)$, where $\lambda\in\R^+$ is the constant from Proposition~\ref{proposition on composition}.  We may then invoke the conclusions of this proposition to deduce the mapping properties asserted in~\eqref{the mapping properties of the flattening map}.  This completes the proof of the second item. 

    To prove the first item we only need to make a trio of observations.  First, Proposition~\ref{proposition on composition} shows that $\Bar{\mathfrak{F}}_\eta$ is a $C^n$ diffeomorphism from $\R^3$ to itself when $\N\ni n<s+2-3/r$.  Second, $\mathfrak{F}_\eta$ is the restriction of $\Bar{\mathfrak{F}}_\eta$ to $\bar{\Omega}$, and by construction $\mathfrak{F}_\eta$ maps $\Bar{\Omega}$ to $\Bar{\Omega[\eta]}$.  Third, the construction of $\mathcal{E}\eta$ shows that its restriction to $\Omega$ itself is smooth.  Together, these prove the first item.

\end{proof}

Next, we analyze quantities derived from the flattening map. Recall that $J_\eta$, $\mathcal{A}_\eta$, and $M_\eta$ are defined in~\eqref{geometry_and_jacobian_def} and~\eqref{Mississippi}.

\begin{prop}[Properties of the Jacobian and geometry matrices]\label{prop on properties of the Jacobian and geometry matricies}
    Let $1<r<2$, $\N\ni 1+s>3/r$, and $\varrho\in\R^+$ (depending on $s$ and $r$) be as in Proposition~\ref{prop on properties of the flattening map}. Then the following mapping properties hold.
    \begin{enumerate}
        \item For $\eta\in\tilde{H}^{3/2+s,r}(\Sigma)$, we have that $J_\eta>0$ and both of the maps
        \begin{equation}
        \tilde{H}^{3/2+s,r}(\Sigma)\supset B(0,\varrho)\ni\eta\mapsto J_\eta,1/J_\eta\in(H^{1+s}_{r,2}+W^{1+s,\infty})(\Omega)
        \end{equation}
        are smooth.
        \item The maps
        \begin{equation}
            \tilde{H}^{3/2+s,r}(\Sigma)\supset B(0,\varrho)\ni\eta\mapsto \mathcal{A}_\eta,\mathcal{A}_\eta^{-1}\in(H^{1+s}_{r,2}+W^{1+s,\infty})(\Omega;\R^{3\times 3})
        \end{equation}
        are smooth.
        \item The maps
        \begin{equation}
            \tilde{H}^{3/2+s,r}(\Sigma)\supset B(0,\varrho)\ni\eta\mapsto M_\eta,M_\eta^{-1}\in(H^{1+s}_{r,2}+W^{1+s,\infty})(\Omega;\R^{3\times 3})
        \end{equation}
        are smooth.
    \end{enumerate}
\end{prop}
\begin{proof}
The maps $\eta\mapsto J_\eta$ and $\eta\mapsto M_\eta$ are affine, and thus smooth thanks to Proposition~\ref{prop on extension operator}.  By invoking Corollary~\ref{smoothness of inversion}, we have that $\eta\mapsto 1/J_\eta$ is smooth. This fact combined with Proposition~\ref{corollary Banach algebra of the sum space} implies that $\eta\mapsto\mathcal{A}_\eta=M_\eta^{\m{t}}/J_\eta$ is smooth. For the smoothness of the pointwise inversion in the third item, we appeal to Proposition~\ref{corollary Banach algebra of the sum space} again and the adjugate formula $\eta\mapsto M_\eta^{-1}=\m{adj}(M_\eta)/J_\eta^2$. The remaining assertion, the smoothness of pointwise inversion in the second item is then handled via the formula $\eta\mapsto\mathcal{A}_\eta^{-1}=J_\eta M_{\eta}^{-\m{t}}$. 
\end{proof}

Now, working towards a synthesis, we make the following definitions.  First we define some spaces.

\begin{defn}[Spaces for the nonlinear analysis]\label{definition on spaces for the nonlinear analysis}
We make the following definitions for $s\in\N$, $1<r<2$, and $\rho\in\R^+$:
\begin{enumerate}
    \item ${_\diamond}H^{2+s}_{r,2}(\Omega;\R^3)=\tcb{u\in H^{2+s}_{r,2}(\Omega;\R^3)\;:\;\grad\cdot u=0,\;\m{Tr}_{\Sigma_0}u=0}$,

    \item ${_\diamond}\bf{X}_{s,r}=H^{1+s}_{r,2}(\Omega)\times{{_\diamond}H^{2+s}_{r,2}(\Omega;\R^3)}\times\tilde{H}^{5/2+s,r}(\Sigma)$,
    \item $\mathcal{O}_{s,r}(\rho)=\tcb{(p,u,\eta)\in{{_\diamond}\bf{X}_{s,r}}\;:\;\eta\in B(0,\rho)}$, 
    \item $\bf{W}_{s,r}=H^{1+s}_{r,2}(\R^3;\R^{3\times 3})\times H^s_{r,2}(\R^3;\R^3)$.
\end{enumerate}
\end{defn}

Next, we define some maps.

\begin{defn}[Maps for the nonlinear analysis, I]\label{definition of the maps for the nonlinear analysis}
    For $1<r<2$, $\N\ni s>3/r$, and $\varrho\in\R^+$ as in Proposition~\ref{prop on properties of the flattening map} we make the following definitions.
    \begin{enumerate}
        \item $\Xi_1:\mathcal{O}_{s,r}(\varrho)\times\R \times \R^+ \times \R^+ \to H^s_{r,2}(\Omega;\R^3)$ is defined via
        \begin{equation}\label{definition the momentum equation map}
            \Xi_1(p,u,\eta,\gam,\mathfrak{g},\mu)= M_\eta^{-\m{t}}((u-\gam M_\eta e_1)\cdot\grad(M_\eta^{-1}u))+\grad(p+\mathfrak{g}\eta)-\mu M_\eta^{-\m{t}}(\grad\cdot((\mathbb{D}_{\mathcal{A}_\eta}(M_\eta^{-1}u))M_\eta^{\m{t}})).
        \end{equation}
        \item $\Xi_2:\mathcal{O}_{s,r}(\varrho) \times \R^+ \times \R^+ \to H^{1/2+s,r}(\Sigma;\R^3)$ is defined via
        \begin{equation}
            \Xi_2(p,u,\eta,\mu,\kappa)=\m{Tr}_{\Sigma}[-(pI-\mu \mathbb{D}_{\mathcal{A}_\eta}(M_\eta^{-1}u))M_\eta^{\m{t}}e_3-\kappa \mathscr{H}(\eta)M_\eta^{\m{t}}e_3].
        \end{equation}
        \item For $m\in\N$ and $\varrho$ as in Proposition~\ref{prop on properties of the flattening map}, $\Upsilon_1: H^{m+s}_{r,2}(\Omega;\R^3)\times B_{\tilde{H}^{3/2+s,r}}(0,\varrho)\to H^s_{r,2}(\Omega;\R^3)$ is defined via
        \begin{equation}
            \Upsilon_1(\mathcal{F},\eta)=-J_\eta M_\eta^{-\m{t}}(\mathcal{F}\circ\mathfrak{F}_\eta).
        \end{equation}
        \item For $m\in\N$ and $\varrho$ as in Proposition~\ref{prop on properties of the flattening map}, $\Upsilon_2: H^{m+1+s}_{r,2}(\Omega;\R^{3\times 3})\times B_{\tilde{H}^{3/2+s,r}}(0,\varrho)\to H^{1/2+s,r}(\Sigma;\R^3)$ is defined via
        \begin{equation}
            \Upsilon_2(\mathcal{T},\eta)=-\m{Tr}_\Sigma[(\mathcal{T}\circ\mathfrak{F}_\eta) M_\eta^{\m{t}}e_3].
        \end{equation}
    \end{enumerate}
\end{defn}

Our next two results study the smoothness of the nonlinear differential operators in the momentum equation and dynamic boundary condition in system~\eqref{final nonlinear equations}. 

\begin{prop}[Mapping properties of the nonlinearities]\label{prop on mapping properties of the momentum equation}
    For $1<r<2$ and $\N\ni s>3/r$, the following mapping properties hold.
    \begin{enumerate}
        \item $\Xi_1$ and $\Xi_2$, as defined in the first and second items of Definition~\ref{definition of the maps for the nonlinear analysis}, are smooth.
        \item $\Upsilon_1$ and $\Upsilon_2$, as defined in the third and fourth items of Definition~\ref{definition of the maps for the nonlinear analysis} are  $C^m$.
    \end{enumerate}
\end{prop}
\begin{proof}
    The first item follows from Propositions~\ref{prop on properties of the Jacobian and geometry matricies} and~\ref{proposition on smoothness of a certain product map} along with Remark~\ref{remark on products 3}, since all of the nonlinearities in $\Xi_1$ are various products of the derivatives of the velocity field with the geometry matrices and parameters. 
    
    The analysis of $\Xi_2$ follows similarly, with the exception of the mean curvature term $\mathscr{H}(\eta)$. For this we use that $\R^2\ni x\mapsto x\tbr{x}^{-1} \in \R^2$ is everywhere analytic and vanishing at zero and hence the map 
\begin{equation}
    H^{3/2+s,r}(\Sigma)^2 \ni (\partial_1 \eta,\partial_2 \eta) \mapsto \tbr{\grad_{\|}\eta}^{-1} (\partial_1 \eta,\partial_2 \eta) \in  H^{3/2+s,r}(\Sigma)^2
\end{equation}   
    is smooth since $H^{3/2+s,r}(\Sigma)$ is an algebra. In turn, we find that the map $\eta \mapsto \mathscr{H}(\eta)$ is also smooth. This completes the proof of the first item.  
    
    The second item follow by similar considerations, supplemented with the second item of Proposition~\ref{prop on properties of the flattening map}.
\end{proof}

The remainder of this subsection's nonlinear analysis is meant to deal with the technicalities arising in the slowly traveling limit ($\gam\to0$) in system~\eqref{final nonlinear equations}. As the equations are currently formulated, there is a change of natural function spaces that occurs in this limit, which suggests that the stationary problem is a (low-mode) singular limit of traveling problems. The effect of this is that the formulation~\eqref{final nonlinear equations} works fine for $\gam=0$ and $\gam\in\R\setminus\tcb{0}$ separately, but is ill-suited for capturing the slowly traveling limit $\gam\to 0$. To overcome this issue, we make a change of unknowns in the free surface.

\begin{defn}[Anisotropic parameterization operators]\label{the gamma parameterization operator}
    For $\gam\in\R$ we let $P_\gam$ be the Fourier multiplication operator with the following symbol
\begin{equation}
    P_\gam=\bf{p}_\gam(D),\quad \bf{p}_\gam(\xi)=\f{4\pi^2|\xi|^2}{4\pi^2|\xi|^2+2\pi\ii\gam\xi_1},
\end{equation}
Note that $P_0$ is the identity operator.
\end{defn}

The relevant properties of the maps $P_\gam$ are enumerated in the following result.  Recall  that the spaces $\dot{H}^{-1,r}(\R^2;\R)$ are defined in~\eqref{definition of the negative homogeneous Sobolev space}.

\begin{prop}[Properties of the anisotropic parameterization operators]\label{properties of the anisotroptic paramterization operators}
    The following hold for $s\in\N^+$ and $1<r<2$.
    \begin{enumerate}
        \item For each $\gam\in\R$ we have that $P_\gam\in\mathcal{L}(\tilde{H}^{s,r}(\R^2;\R))$. Moreover, the map $\R\ni\gamma\mapsto P_\gam\in\mathcal{L}(\tilde{H}^{s,r}(\R^2;\R))$ is bounded, and for any  $\eta \in \tilde{H}^{s,r}(\R^2;\R)$ the map  $\R\ni\gamma\mapsto P_\gam \eta \in \tilde{H}^{s,r}(\R^2;\R)$ is continuous.
  
        \item For each $\gam\in\R$ we have that $\gam\pd_1P_\gam\in\mathcal{L}(\tilde{H}^{s,r}(\R^2;\R);\dot{H}^{-1,r}(\R^2;\R))$. Moreover, the map $\R\ni\gamma\mapsto \gam\pd_1P_\gam\in\mathcal{L}(\tilde{H}^{s,r}(\R^2;\R))$ is bounded, and  for any  $\eta \in \tilde{H}^{s,r}(\R^2;\R)$ the map  $\R\ni\gamma\mapsto  \gamma \pd_1 P_\gam \eta \in \dot{H}^{-1,r}(\R^2;\R)$ is continuous. 
        
        \item The mappings
        \begin{equation}
            \R\times\tilde{H}^{s,r}(\R^2;\R)\ni(\gamma,\upeta)\mapsto\begin{cases}
                P_\gamma\upeta\in\tilde{H}^{s,r}(\R^2;\R),\\
                \gam\pd_1P_\gam\upeta\in\dot{H}^{-1,r}(\R^2;\R),\\
                \gam\mathcal{R}_1P_\gam\upeta\in H^{s,r}(\R^2;\R)
            \end{cases}
        \end{equation}
        are continuous, where $\mathcal{R}_1$ refers to the Riesz transform in the $e_1$ direction.
    \end{enumerate}
\end{prop}
\begin{proof}
    The third item follows from the first two, so we turn our attention to proving these.

    We claim that the multiplier $\bf{p}_\gam(\xi)=\f{4\pi^2|\xi|^2}{4\pi^2|\xi|^2+2\pi\ii\gam\xi_1}$ is of Marcinkiewicz type (see Theorem \ref{thm on Marcinkiewicz multiplier theorem}) and the defining inequalities \eqref{Marcinkiewicz ineqs} are satisfied uniformly over $\gam\in\R$.   To prove the claim, we first note that
    \begin{equation}\label{MKE0}
    |\bf{p}_\gam(\xi)| \le 1    
    \end{equation}
    and that $\bf{p}_\gam$ is smooth away from the coordinate axes.  Next, we compute 
    \begin{equation}
        \pd_{1}\bf{p}_\gam(\xi)=\f{2\pi\ii\gam(\xi_1^2-\xi_2^2)}{(2\pi|\xi|^2+\ii\gam\xi_1)^2},
        \quad
        \pd_{2}\bf{p}_\gam(\xi)=\f{4\pi\ii\gam\xi_1\xi_2}{(2\pi|\xi|^2+\ii\gam\xi_1)^2},
    \end{equation}
    and
    \begin{equation}
        \pd_{1}\pd_{2}\bf{p}_\gam(\xi)=-\f{4\pi\ii\gam\xi_2(6\pi\xi_1^2-2\pi\xi_2^2+\ii\gam\xi_1)}{(2\pi|\xi|^2+\ii\gam\xi_1)^3}.
        \end{equation}
        Thus, we have the following estimates from Cauchy's inequality:
        \begin{equation}\label{MKE1}
            |\xi_1\pd_{1}\bf{p}_\gam(\xi)|\le\f{|\gam\xi_1|2\pi|\xi|^2}{4\pi^2|\xi|^4+\gam^2|\xi_1|^2}\le\f12,
            \quad
            |\xi_2\pd_{2}\bf{p}_\gam(\xi)|\le\f{4\pi|\gam\xi_1||\xi|^2}{4\pi^2|\xi|^4+\gam^2|\xi_1|^2}\le 1,
        \end{equation}
        and
        \begin{equation}\label{MKE2}
            |\xi_1\xi_2\pd_{1}\pd_{2}\bf{p}_\gam(\xi)|\le\f{12\pi|\gam\xi_1||\xi|^2|2\pi|\xi|^2+\ii\gam\xi_1|}{(4\pi^2|\xi|^4+\gam^2|\xi_1|^2)^{3/2}}\le\f{12\pi|\gam\xi_1||\xi|^2}{4\pi^2|\xi|^4+\gam^2|\xi_1|^2}\le 3.
        \end{equation}
        This completes the proof of the claim, and so we may invoke Theorem~\ref{thm on Marcinkiewicz multiplier theorem} (with the observation that $\bf{p}_\gam(-\xi)=\Bar{\bf{p}_\gam(\xi)}$ for $\xi\in\R^2\setminus\tcb{0}$) to see that 
        \begin{equation}\label{P op est}
            \norm{P_\gamma}_{\mathcal{L}(L^r(\R^2))} \le C_{r}(3 +1) = 4 C_r
        \end{equation}
        for a constant $C_r$ depending only on $r$.

        Armed with \eqref{P op est}, we are now ready to prove the first item.  If $\upeta\in\tilde{H}^{s,r}(\R^2;\R)$, then we have
        \begin{equation}
          \sup_{\gamma \in \R}  \tnorm{P_\gam\upeta}_{\tilde{H}^{s,r}}= \sup_{\gamma \in \R} \tnorm{\bf{p}_\gam(D)\tbr{D}^{s-1}\grad_{\|}\upeta}_{L^r}\lesssim\tnorm{\tbr{D}^{s-1}\grad_{\|}\upeta}_{L^r}\lesssim\tnorm{\upeta}_{\tilde{H}^{s,r}}.
        \end{equation}
        Next, we prove that $\gam\mapsto P_\gam$ is continuous for the strong operator topology. Fix $\upeta\in\tilde{H}^{s,r}(\R^2;\R^2)$ and $\ep\in\R^+$. Thanks to density of $H^{s,r}(\R^2;\R)\emb\tilde{H}^{s,r}(\R^2;\R)$ (see Proposition \ref{prop on completeness of subcritical gradient spaces}), there exists $\upzeta\in H^{s,r}(\R^2;\R)$ such that $\tnorm{\upzeta-\upeta}_{\tilde{H}^{s,r}}\le\ep$. Thus, we learn from \eqref{P op est} that if $\gam,\gam_0\in\R$ then
        \begin{equation}\label{density estimate}
            \tnorm{(P_\gam-P_{\gam_0})\upeta}_{\tilde{H}^{s,r}}\lesssim\ep+\tnorm{(P_\gam-P_{\gam_{0}})\upzeta}_{\tilde{H}^{s,r}}.
        \end{equation}
        We compute via the fundamental theorem of calculus that
        \begin{equation}
            \bf{p}_\gam(\xi)-\bf{p}_{\gam_0}(\xi)=(\gam-\gam_0)\f{\ii\xi_1}{2\pi|\xi|^2}\bf{q}_{\gam,\gam_0}(\xi)
            \text{ for }
            \bf{q}_{\gam,\gam_0}(\xi)=\int_0^1(\bf{p}_{t\gam+(1-t)\gam_0}(\xi))^2\;\m{d}t.
        \end{equation}
        Estimates~\eqref{MKE0}, ~\eqref{MKE1}, and~\eqref{MKE2} and the Leibniz rule show that
        \begin{equation}
        \sup_{\gamma, \gamma_0}\left(  |\bf{q}_{\gam,\gam_0}(\xi)| +   |\xi_1\pd_{1}\bf{q}_{\gam,\gam_0}(\xi)| + |\xi_2\pd_{2}\bf{q}_{\gam,\gam_0}(\xi)| + |\xi_1 \xi_2 \pd_{1}\pd_2 \bf{q}_{\gam,\gam_0}(\xi)| \right) \lesssim 1,
        \end{equation}
        and we also have that $\bf{q}_{\gam,\gam_0}(-\xi)=\Bar{\bf{q}_{\gam,\gam_0}(\xi)}$, so we may once more appeal to Theorem~\ref{thm on Marcinkiewicz multiplier theorem} to learn that
        \begin{equation}
            \sup_{\gamma,\gamma_0}  \norm{\bf{q}_{\gam,\gam_0}(D)}_{\mathcal{L}(L^p(\R^2))} \lesssim 1.
        \end{equation}
        Writing $\mathcal{R} = (\mathcal{R}_1,\mathcal{R}_2)$ for the vector of Riesz transforms, we then have that 
        \begin{multline}\label{lipschitz estimate}
            \tnorm{(P_\gam-P_{\gam_0})\upzeta}_{\tilde{H}^{s,r}} = \tnorm{\nabla (P_\gam-P_{\gam_0})\upzeta}_{H^{s-1,r}} 
            = \frac{|\gam-\gam_0|}{2\pi}
            \tnorm{\bf{q}_{\gam,\gam_0}(D) \mathcal{R} \mathcal{R}_1 \br{D}^{s-1} \upzeta}_{L^r}
            \\
            \lesssim|\gam-\gam_0|\tnorm{ \br{D}^{s-1} \zeta}_{L^{r}}
            =|\gam-\gam_0|\tnorm{\zeta}_{H^{s-1,r}}.
        \end{multline}
        By combining~\eqref{density estimate} and~\eqref{lipschitz estimate}, we get
        \begin{equation}
            \limsup_{\gam_0\to\gam}\tnorm{(P_\gam-P_{\gam_0})\upeta}_{\tilde{H}^{s,r}}\lesssim\ep,
        \end{equation}
        so the continuity claim follows.  This completes the proof of the first item.

        The second item is proved by similar considerations  thanks to the identity
        \begin{equation}\label{key identity for the second item}
            \upgamma\mathcal{R}_1P_\upgamma= (1-P_\upgamma)2\pi|D| = (P_\upgamma -1) \mathcal{R} \cdot \nabla,
        \end{equation}
        which shows that for $\upeta\in\tilde{H}^{s,r}(\R^2;\R)$ 
        \begin{equation}
            \sup_{\gamma} \tsb{\gam\pd_1 P_\gam\upeta}_{\dot{H}^{-1,r}}
            =  \sup_{\gamma} \tnorm{\gam\mathcal{R}_1P_\gam \upeta}_{L^r}= \sup_{\gamma} \tnorm{(P_\gam-1) \mathcal{R} \cdot \nabla \upeta}_{L^r}\lesssim\tnorm{\eta}_{\tilde{H}^{s,r}},
        \end{equation}
        where the last inequality follows from the \eqref{P op est} and the $L^r$ boundedness of Riesz transforms. For continuity, we again fix $\eta\in\tilde{H}^{s,r}(\R^2;\R)$, $\gam_0,\gam\in\R$ and then use~\eqref{key identity for the second item} to deduce that
        \begin{equation}
            \tsb{(\gam\pd_1P_\gam-\gam_0\pd_1P_{\gam_0})\upeta}_{\dot{H}^{-1,r}}\lesssim\tnorm{(P_{\gam_0}-P_{\gam})\mathcal{R} \cdot \nabla \upeta}_{L^r}\lesssim\tnorm{(P_{\gam_0}-P_\gam)\upeta}_{\tilde{H}^{s,r}},
        \end{equation}
        which means that the continuity assertion of the second item follows from that of the first.  This completes the proof of the second item.
\end{proof}

The operators of Definition~\ref{the gamma parameterization operator} permit us to make a change of unknowns in~\eqref{final nonlinear equations} to overcome the aforementioned singular limit issue. We consider $\eta=P_\gam \upeta$, and view $\upeta$ as our new unknown. The main theorem of this section's nonlinear analysis, which implements this change of unknowns crucially, is now given as follows.

\begin{thm}[Mapping properties]\label{thm on mapping properties}
    For $1<r<2$, $\N\ni s>3/r$, and $\varrho\in\R^+$ as in Proposition~\ref{prop on properties of the flattening map}, there exists $C\in\R^+$ such that the map $\Upxi:\mathcal{O}_{s,r}(\varrho/C)\times\R \times  (\R^+)^3 \times\bf{W}_{1+s,r}\to\bf{Y}_{s,r}$ given by
    \begin{multline}\label{the nonlinear map}
    \Upxi(p,u,\upeta,\gamma,\mathfrak{g},\mu,\kappa, \mathcal{T},\mathcal{F}) \\
    =(\Xi_1(p,u,P_\gam\upeta,\gamma,\mathfrak{g},\mu)+\Upsilon_1(\mathcal{F},P_\gam\upeta),\Xi_2(p,u,P_\gam\upeta,\mu,\kappa)+\Upsilon_2(\mathcal{T},P_\gam\upeta),\m{Tr}_{\Sigma}u\cdot e_3+\gam\pd_1P_\gam\upeta)
    \end{multline}
    is well-defined and continuous.  Moreover,  the Fr\'echet derivative with respect to the first factor, $D_1\Upxi:\mathcal{O}_{s,r}(\varrho/C)\times\R\times(\R^+)^3\times\bf{W}_{1+s,r}\to\mathcal{L}({_\diamond}\bf{X}_{s,r};\bf{Y}_{s,r})$, exists and is continuous.
\end{thm}
\begin{proof}
The uniform boundedness assertions in the first item of Proposition~\ref{properties of the anisotroptic paramterization operators} guarantee that for some $C\in\R^+$ we have $(p,u,P_\gam\upeta)\in\mathcal{O}_{s,r}(\varrho)$ for all $(p,u,\upeta)\in\mathcal{O}_{s,r}(\varrho/C)$. Thus, by composition, linearity of $P_\gam$, and the third item of Proposition~\ref{properties of the anisotroptic paramterization operators} we may invoke is Proposition~\ref{prop on mapping properties of the momentum equation} to reach the desired conclusions for the first and second components of the map $\Upxi.$ The third  component of $\Upxi$ is handled via the linearity of $\gam\pd_1P_\gam$, together with the second and third items of Proposition~\ref{properties of the anisotroptic paramterization operators} and the divergence compatibility estimate of Proposition~\ref{the divergence compatibility condition is here for you to check it out}.
\end{proof}

% _+__+_ -_+__+_ -_+__+_ -_+__+_ -_+__+_ -_+__+_ -_+__+_ -_+__+_ -_+__+_ -_+__+_ -_+__+_ -_+__+_ -_+__+_ -
\subsection{Well-posedness}\label{subsection on well-posedness}
% _+__+_ -_+__+_ -_+__+_ -_+__+_ -_+__+_ -_+__+_ -_+__+_ -_+__+_ -_+__+_ -_+__+_ -_+__+_ -_+__+_ -_+__+_ -

We are now ready to prove our main theorem.  This subsection is split into four main results and then a list of corollaries, which combine to prove Theorem \ref{the main theorem}.  In the first main result, we invoke the implicit function theorem at a fixed tuple of positive physical parameters and obtain a solution map.  In the next, we show that we can glue these together across all parameter values.  One slight issue that remains after this is done is that the resulting solution map loses a derivative relative to what one would expect.  This fact stems from the numerology of higher order smoothness of composition-type nonlinearities (see e.g. Proposition~\ref{proposition on composition}). We are thus lead to the third result in this subsection, in which we show a posteriori that the solution map actually obeys the optimal derivative counting. In the fourth, and final, main result of this subsection, we recast the previous results into a more physically relevant formulation by anisotropically parameterizing the free surface variable with the operators in Proposition~\ref{properties of the anisotroptic paramterization operators}.

We shall use the following version of the inverse function theorem, as formulated as in Theorem A in Crandall and Rabinowitz~\cite{MR0288640} (for a verbose proof see Theorem 2.7.2 in Nirenberg~\cite{MR1850453}, but note that there is slight misstatement of the uniqueness assertion in the first item there that is correct in~\cite{MR0288640}). 

\begin{thm}[Implicit function theorem]\label{implicit function theorem}
    Let $X,Y,Z$ be Banach spaces and $f$ a continuous mapping of an open set $U\subset X\times Y\to Z$. Assume that $f$ has a Fr\'echet derivative with respect to the first factor, $D_1f:U\to\mathcal{L}(X;Z)$ that is continuous. Suppose that $(x_0,y_0)\in U$ and $f(x_0,y_0)=0$. If $D_1f(x_0,y_0)$ is an isomorphism of $X$ onto $Z$, then there exist balls $B(y_0,r_Y)\subset Y$ and $B(x_0,r_X)\subset X$ such that $B(x_0,r_X)\times B(y_0,r_Y)\subset U$ and a continuous unique function $u:B(y_0,r_Y)\to B(x_0,r_X)$ such that $u(y_0)=x_0$ and $f(u(y),y)=0$ for all $y\in B(y_0,r_Y)$. Moreover, the implicit function $u$ is continuous.
\end{thm}

We now apply Theorem~\ref{implicit function theorem} in our first well-posedness result.

\begin{thm}[Well-posedness, I]\label{first theorem on well-posedness for nonlinear problem}
    Let $1<r<2$, $\N\ni s>3/r$, and $\varrho,C\in\R^+$ be as in Theorem~\ref{thm on mapping properties}. For each $\bf{v}=(\Bar{\mathfrak{g}},\Bar{\mu},\Bar{\kappa})\in(\R^+)^3$ there exists $\uprho_{s,\bf{v}},\uprho_{s,\bf{v}}'\in\R^+$ and a unique mapping
    \begin{equation}\label{the map upiota}
        \upiota_{\bf{v}}:B((0,\bf{v},0),\uprho_{s,\bf{v}})\subset\R\times(\R^+)^3\times\bf{W}_{1+s,r}\times\bf{Y}_{s,r}\to B(0,\uprho_{s,\bf{v}}')\subset\mathcal{O}_{s,r}(\varrho/C)
    \end{equation}
    with the property that for all data $(\gam,\mathfrak{g},\mu,\kappa,\mathcal{T},\mathcal{F},f,k,h)\in B((0,\bf{v},0),\uprho_{s,\bf{v}})$ we have that $(p,u,\upeta)=\upiota_{\bf{v}}(\gam,\mathfrak{g},\mu,\kappa,\mathcal{T},\mathcal{F},f,k,h)\in\mathcal{O}_{s,r}(\varrho/C)$ is a solution to
    \begin{equation}\label{modified final nonlinear equations}
\begin{cases}
    M_{P_\gam\upeta}^{-\m{t}}((u-\gam M_{P_\gam\upeta} e_1)\cdot\grad(M_{P_\gam\upeta}^{-1}u))+\grad(p+\mathfrak{g} P_\gam\upeta)\\\quad-\mu M_{P_\gam\upeta}^{-\m{t}}(\grad\cdot((\mathbb{D}_{\mathcal{A}_{P_\gam\upeta}}(M_{P_\gam\upeta}^{-1}u))M_{P_\gam\upeta}^{\m{t}}))=f+J_{P_\gam\upeta} M_{P_\gam\upeta}^{-\m{t}}\mathcal{F}\circ\mathfrak{F}_{P_\gam\upeta}&\text{in }\Omega,\\
    \grad\cdot u=0&\text{in }\Omega,\\
    -(p-\mu \mathbb{D}_{\mathcal{A}_{P_\gam\upeta}}(M_{P_\gam\upeta}^{-1}u))M_{P_\gam\upeta}^{\m{t}}e_3-\kappa \mathscr{H}(P_\gam\upeta)M_{P_\gam\upeta}^{\m{t}}e_3=k+\mathcal{T}\circ\mathfrak{F}_{P_\gam\upeta} M_{P_\gam\upeta}^{\m{t}}e_3&\text{on }\Sigma,\\
    u\cdot e_3+\gam\pd_1P_\gam\upeta=h&\text{on }\Sigma,\\
    u=0&\text{on }\Sigma_0.
\end{cases}
\end{equation}
Moreover, $\upiota_{\bf{v}}$ in~\eqref{the map upiota} is continuous.
\end{thm}
\begin{proof}
Consider the map $\Bar{\Upxi}:\mathcal{O}_{s,r}(\varrho/C)\times\R\times(\R^+)^3\times\bf{W}_{1+s,r}\times\bf{Y}_{s,r}\to\bf{Y}_{s,r}$ defined via
\begin{equation}
    \Bar{\Upxi}(p,u,\upeta,\gam,\mathfrak{g},\mu,\kappa,\mathcal{T},\mathcal{F},f,k,h)=\Upxi(p,u,\upeta,\gam,\mathfrak{g},\mu,\kappa,\mathcal{T},\mathcal{F})-(f,k,h).
\end{equation}
Thanks to Theorem~\ref{thm on mapping properties}, this map is well-defined, continuous, and the Fr\'echet derivative with respect to the first factor, $D_1\Bar{\Upxi}:\mathcal{O}_{s,r}(\varrho/C)\times\R\times(\R^+)^3\times\bf{W}_{1+s,r}\times\bf{Y}_{s,r}\to\mathcal{L}({_\diamond}\bf{X}_{s,r};\bf{Y}_{s,r})$, exists and is continuous. Moreover, we have that $\Bar{\Xi}(0,\bf{v},0)=0$.  Theorem~\ref{thm on well-posedness of the linearization in mixed-type Sobolev spaces, II} then shows that $D_1\Bar{\Upxi}(0,\bf{v},0)$ is an isomorphism from ${_\diamond}\bf{X}_{s,r}$ to $\bf{Y}_{s,r}$.   We may then invoke the version of the implicit function theorem given in Theorem~\ref{implicit function theorem} to obtain parameters $\uprho_{s,\bf{v}},\uprho_{s,\bf{v}}'\in\R^+$ along with the map $\upiota_{\bf{v}}$.
\end{proof}

We can recast the previous theorem into the following more general statement via a gluing argument.

\begin{thm}[Well-posedness, II]\label{well-posedness thm after gluing}
    Let $1<r<2$, $\N\ni s>3/r$. There exists and open set
    \begin{equation}\label{open set for the data}
        \tcb{0}\times(\R^+)^3\times\tcb{0}\times\tcb{0}\subset \mathscr{U}_s\subset\R\times(\R^+)^3\times\bf{W}_{1+s,r}\times\bf{Y}_{s,r}
    \end{equation}
    and a continuous map
    \begin{equation}\label{the glue map is sticky}
        \upiota:\mathscr{U}_s\to\bigcup_{\bf{v}\in(\R^+)^3}B(0,\uprho'_{s,\bf{v}})\subset\mathcal{O}_{s,r}(\varrho/C),
    \end{equation}
    where the radii $\uprho'_{s,\bf{v}}>0$ are as in Theorem \ref{first theorem on well-posedness for nonlinear problem}, with the property that for all $\bf{U}=(\gam,\mathfrak{g},\mu,\kappa,\mathcal{T},\mathcal{F},f,k,h)\in\mathscr{U}_s$ we have that $\upiota(\bf{U})=(p,u,\eta)\in B(0,\uprho'_{s,(\mathfrak{g},\mu,\kappa)})$ is the unique solution to~\eqref{modified final nonlinear equations} with data $\bf{U}$.  
\end{thm}
\begin{proof}
    We set
    \begin{equation}
        \mathscr{U}_{s}=\bigcup_{\bf{v}\in(\R^+)^3}B((0,\bf{v},0),\uprho_{s,\bf{v}}),
    \end{equation}
    where $\uprho_{s,\bf{v}}$ are the radii granted by Theorem~\ref{first theorem on well-posedness for nonlinear problem}. We propose defining the map~\eqref{the glue map is sticky} via
    \begin{equation}
        \upiota=\upiota_{\bf{v}}\quad\text{on the set }B((0,\bf{v},0),\uprho_{s,\bf{v}}) \text{ whenever }\bf{v}\in(\R^+)^3,
    \end{equation}
    where the maps $\upiota_{\bf{v}}$ are the solution operators granted by Theorem~\ref{first theorem on well-posedness for nonlinear problem}.  This is well-defined since if $\bf{v},\bf{w}\in(\R^+)^3$ are such that $B((0,\bf{v},0),\uprho_{s,\bf{v}})\cap B((0,\bf{w},0),\uprho_{s,\bf{w}})\neq\es$, then the maps $\upiota_{\bf{v}}$ and $\upiota_{\bf{w}}$ agree on this intersection since, according to the aforementioned theorem, they are both the unique solution operators to the PDE~\eqref{modified final nonlinear equations}.  Once we know that $\upiota$ is well-defined, continuity follows from the continuity of each $\upiota_{\bf{v}}$.  The remaining uniqueness assertions are just a restatement of those from Theorem~\ref{first theorem on well-posedness for nonlinear problem}.   
    \end{proof}

The next well-posedness result gains an extra derivative on the solution to reach the optimal counting.

\begin{thm}[Well-posedness, III]\label{third theorem on well-posedness for nonlinear problem}
    Let $1<r<2$, $\N\ni s>3/r+1$. There exists an open set
    \begin{equation}
        \tcb{0}\times(\R^+)^3\times\tcb{0}\subset\mathscr{W}_s\subset\R\times(\R^+)^3\times\bf{W}_{s,r} 
    \end{equation}
    such that $\mathscr{W}_s\times\tcb{0}\subset\mathscr{U}_{s-1}$, where $\mathscr{U}_{s-1}$ is as in Theorem \ref{well-posedness thm after gluing},     and a continuous map
    \begin{equation}
        \tilde{\upiota}:\mathscr{W}_{s}\to\bf{X}_{s,r}\cap\bigcup_{\bf{v}\in(\R^+)^3}B_{\bf{X}_{s-1,r}}(0,\uprho'_{s-1,\bf{v}})\subset\mathcal{O}_{s,r}(\varrho/C)
    \end{equation}
    with the property that for all $\bf{W}=(\gam,\mathfrak{g},\mathfrak{\mu},\kappa,\mathcal{T},\mathcal{F})\in\mathscr{W}_s$ we have that $\tilde{\upiota}(\bf{W})=(p,u,\eta)\in\bf{X}_{s,r}\cap B_{\bf{X}_{s,r}}(0,\uprho'_{s-1,(\mathfrak{g},\mu,\kappa)})$ is the unique solution to~\eqref{modified final nonlinear equations} with data $(\bf{W},0)$.

\end{thm}
\begin{proof}
    We set $\tilde{\upiota}(\gam,\mathfrak{g},\mu,\kappa,\mathcal{T},\mathcal{F})=\upiota(\gam,\mathfrak{g},\mu,\kappa,\mathcal{T},\mathcal{F},0,0,0)$, where $\upiota$ is the solution operator granted by Theorem~\ref{well-posedness thm after gluing}. A priori, we only know that
    \begin{equation}\label{a priori continuity}
    \tilde{\upiota}:\bigcup_{\bf{v}\in(\R^3)^+}B((0,\bf{v},0),\uprho_{s,\bf{v}})\subset\R\times\tp{\R^+}^3\times\bf{W}_{s,r}\to\bigcup_{\bf{v}\in(\R^+)^3}B_{\bf{X}_{s-1,r}}(0,\uprho_{s-1,\bf{v}}')
    \end{equation}
    is a continuous mapping.
    
    To complete the proof, we claim that by shrinking the domain a little bit, if necessary, we can gain an additional derivative on the solution.  To see this let $0<\sig\le1$ and set
    \begin{equation}
        \mathscr{W}_s(\sig)=\bigcup_{\bf{v}\in(\R^3)^+}B((0,\bf{v},0),\sig \uprho_{s,\bf{v}})\subset\R\times\tp{\R^+}^3\times\bf{W}_{s,r}.
    \end{equation}
    Denote $(p,u,\upeta)=\tilde{\upiota}(\gam,\mathfrak{g},\mu,\kappa,\mathcal{T},\mathcal{F})$. By the second item of Proposition~\ref{prop on mapping properties of the momentum equation}, the continuity of the solution map in~\eqref{a priori continuity}, and the continuity of composition we have that the map
    \begin{equation}
        \mathscr{W}_s(\sig)\ni(\gam,\mathfrak{g},\mu,\kappa,\mathcal{T},\mathcal{F})\mapsto(\Upupsilon_1(\mathcal{F},P_\gam\upeta),\Upupsilon_2(\mathcal{T},P_\gam\upeta),0)\in\bf{Y}_{s,r}
    \end{equation}
    is continuous, vanishes whenever $\mathcal{T}$ and $\mathcal{F}$ are zero, and is independent of $(\mathfrak{g},\mu,\kappa)$. Thus, by taking $0<\sig_\star\le1$ sufficiently small, we guarantee that 
    \begin{equation}\label{wont you please arrange it}
        \mathscr{W}_s(\sig_\star)\ni(\gam,\mathfrak{g},\mu,\kappa,\mathcal{T},\mathcal{F})\mapsto(\gam,\mathfrak{g},\mu,\kappa,0,0,\Upupsilon_1(\mathcal{F},P_\gam\upeta),\Upupsilon_2(\mathcal{T},P_\gam\upeta),0)\in \mathscr{U}_s.
    \end{equation}
    Since the map in~\eqref{wont you please arrange it} actually has its range in the domain of the map $\upiota$, which is given in~\eqref{the glue map is sticky}, we have verified the following key identity:
    \begin{equation}\label{key oberservation}
        (p,u,\upeta)=\upiota(\gam,\mathfrak{g},\mu,\kappa,0,0,\Upupsilon_1(\mathcal{F},P_\gam\upeta),\Upupsilon_2(\mathcal{T},P_\gam\upeta),0).
    \end{equation}    
     Then the mapping properties of $\upiota$ reveal that the solution
     \begin{equation}
         (p,u,\upeta)\in B_{\bf{X}_{s,r}}(0,\uprho'_{s,(\mathfrak{g},\mu,\kappa)})\cap B_{\bf{X}_{s-1,r}}(0,\uprho'_{s-1,(\mathfrak{g},\mu,\kappa)})
     \end{equation}
     varies continuously with the data $(\gam,\mathfrak{g},\mu,\kappa,\mathcal{T},\mathcal{F})\in\mathscr{W}_s=\mathscr{W}_s(\sig_\star)$.  This proves the claim.
\end{proof}

This section's list of main theorems is concluded with the following.

\begin{thm}[Well-posedness, IV]\label{coro fourth times the charm for well-posedness}
    Let $1<r<2$, $\N\ni s>3/r+1$, and $\mathscr{W}_s$ be the open set from Theorem \ref{third theorem on well-posedness for nonlinear problem}. There exists a map
    \begin{equation}\label{solution operator}
    \mathscr{W}_s\ni(\gam,\mathfrak{g},\mu,\kappa,\mathcal{T},\mathcal{F})\mapsto(p,u,\eta)\in\mathcal{O}_{s,r}(\varrho)
    \end{equation}
    such that the following hold.
    \begin{enumerate}
        \item The map  is a continuous solution operator to the nonlinear system~\eqref{final nonlinear equations}.
        \item The map  is locally unique in the sense that there exist open sets $W_s\subseteq\mathscr{W}_s$ and $\tcb{V_s(\bf{v})}_{\bf{v}\in(\R^+)^3}\subseteq\mathcal{O}_{s,r}(\varrho)$, obeying the non-degeneracy conditions of~\eqref{non degeneracy conditions}, for which the following two conditions hold.
    \begin{enumerate}[(i)]
            \item The image of $W_s$ under~\eqref{solution operator} is contained within $\bigcup_{\bf{v}\in(\R^+)^3}V_s(\bf{v})$.
            \item For each $\bf{v}\in(\R^+)^3$ the restriction of~\eqref{solution operator} to the preimage of $V_s(\bf{v})$, thought of as a mapping to $V_s(\bf{v})$, is the unique function that is a solution operator to~\eqref{final nonlinear equations}.
    \end{enumerate}
        \item We have an extra `anisotropic' estimate on the free surface in the sense that the composition map
    \begin{equation}\label{the extra anisotropic estimate}
        \mathscr{W}_s\ni(\gam,\mathfrak{g},\mu,\kappa,\mathcal{T},\mathcal{F}) \mapsto (p,u,\eta)  \mapsto\gam\mathcal{R}_1\eta\in L^r(\Sigma)
    \end{equation}
        is well-defined and continuous, where $\mathcal{R}_1$ is the Riesz  transform in the $e_1$ direction.
    \end{enumerate}

\end{thm}
\begin{proof}
    We take~\eqref{solution operator} to be the composition $(\gam,\mathfrak{g},\mu,\kappa,\mathcal{T},\mathcal{F})\overset{\tilde{\upiota}}{\mapsto}(p,u,\upeta)\mapsto (p,u,\eta)$, with $\eta=P_\gam\upeta$. Thanks to Theorem~\ref{third theorem on well-posedness for nonlinear problem} and Proposition~\ref{properties of the anisotroptic paramterization operators}, this is a continuous solution operator for~\eqref{final nonlinear equations}. The final item of the aforementioned proposition also guarantees that the mapping of~\eqref{the extra anisotropic estimate} is well-defined and continuous.
    
    It remains to establish local uniqueness. Suppose that $(p,u,\eta),(p',u',\eta')\in \mathcal{O}_{s,p}(\varrho)$ are such that there exists $(\gam,\mathfrak{g},\mu,\kappa,\mathcal{T},\mathcal{F})\in \mathscr{W}_s$ and both $(p,u,\eta)$ and $(p',u',\eta')$ are solutions to~\eqref{final nonlinear equations} with the same data $(\mathcal{T},\mathcal{F})$, same wave speed $\gam$, and same physical constants $(\mathfrak{g},\mu,\kappa)$. Integrating the divergence free constraint over $y \in (0,b)$ and appealing to the kinematic boundary condition and the no slip condition yields the identity
    \begin{equation}\label{derivative identity}
        \gam\pd_1\zeta=(\grad_{\|},0)\cdot\int_0^bw(\cdot,y)\;\m{d}y
        \text{ for }
        (\zeta,w)\in\tcb{(\eta,u),(\eta',u')}.
    \end{equation}
    Therefore, by~\eqref{derivative identity} and the identity $P_\gam^{-1} \nabla =\nabla +\mathcal{R} \gam \mathcal{R}_1 $,  where $\mathcal{R} =(\mathcal{R}_1,\mathcal{R}_2)$ is the vector of Riesz transforms, we have the estimate
    \begin{equation}
        \tnorm{P_\gam^{-1}\zeta}_{\tilde{H}^{5/2+s,r}}\lesssim\tnorm{\zeta}_{\tilde{H}^{5/2+s,r}}+\bnorm{|D|^{-1}(\grad_{\|},0)\cdot\int_0^bw(\cdot,y)\;\m{d}y}_{H^{3/2+s,r}}\lesssim \tnorm{w,\zeta}_{H^{2+s}_{r,2}\times\tilde{H}^{5/2+s,r}},
    \end{equation}
    where in both of the inequalities we have applied the boundedness of Riesz transforms. If $(p,u,\eta)$ and $(p',u',\eta')$ are sufficiently small, we therefore guarantee that $(p,u,P_\gam^{-1}\eta),(p',u',P_\gam^{-1}\eta')\in \bf{X}_{s,r}\cap B_{\bf{X}_{s-1,r}}(0,\uprho'_{s-1,(\mathfrak{g},\mu,\kappa)})$, but we can then invoke the uniqueness assertion of Theorem~\ref{third theorem on well-posedness for nonlinear problem} to find that $(p,u,P_\gam^{-1}\eta)=(p',u',P_\gam^{-1}\eta')$, which implies $(p,u,\eta)=(p',u',\eta')$.
    
    Thus, we take $V_s(\mathfrak{g},\mu,\kappa)$ to be an open ball about the origin of a positive radius that obeys the above smallness requirements. The set $W_s$ is then defined to be the union of the preimages of these balls under the map~\eqref{solution operator}.
\end{proof}

We now enumerate some important consequences.

\begin{coro}[Some further conclusions]\label{coro on some further conclusions}
    For $1<r<2$ and $\N\ni s>3/r+1$ the following hold.
    \begin{enumerate}
        \item \textbf{Classical solutions:} Each triple $(p,u,\eta)$ produced by Theorem~\ref{coro fourth times the charm for well-posedness} is a classical solution    
        to system~\eqref{final nonlinear equations}.  More precisely, whenever $(\gam,\mathfrak{g},\mu,\kappa,\mathcal{T},\mathcal{F})\in\mathscr{W}_s$ we have that the associated solution satisfies
        \begin{equation}\label{classicality}
            (p,u,\eta)\in C^{2+k}_0(\Omega)\times C^{3+k}_0(\Omega;\R^3)\times C^{4+k}_0(\Sigma),
        \end{equation}
        for $k=s-2-\tfloor{3/r}\in\N$.
        
        \item \textbf{Eulerian transfer:}  Each solution $(p,u,\eta)$ to system~\eqref{final nonlinear equations}, produced by Theorem~\ref{coro fourth times the charm for well-posedness} gives rise to a corresponding classical solution
        \begin{equation}
            (q,v,\eta)\in C^{2+k}_0(\Omega[\eta])\times C^{3+k}_0(\Omega[\eta])\times C^{4+k}_0(\Omega[\eta]),\quad k=s-2-\tfloor{3/r}
        \end{equation}
        to the stationary-traveling Eulerian formulation of the problem given by system~\eqref{the stationarity ansatz} via unflattening.
        
        \item \textbf{Fixed physical parameters, variable wave speed:} For each $(\mathfrak{g},\mu,\kappa)\in(\R^+)^3$, there exists an open set $(0,0,0)\in W_s(\mathfrak{g},\mu,\kappa)\subset\R\times\bf{W}_{s,r}$ and a unique function
        \begin{equation}\label{the particular solution map}
            W_s(\mathfrak{g},\mu,\kappa)\ni(\gam,\mathcal{T},\mathcal{F})\mapsto(p,u,\eta)\in V_s(\mathfrak{g},\mu,\kappa)
        \end{equation}
        with the property that for all $(\gam,\mathcal{T},\mathcal{F})$ belonging to the domain, the corresponding $(p,u,\eta)$ solves~\eqref{final nonlinear equations} with wave speed $\gam$, physical parameters $(\mathfrak{g},\mu,\kappa)$, and stress-force data $(\mathcal{T},\mathcal{F})$. Moreover, the map~\eqref{the particular solution map} is continuous.
        \item \textbf{Well-posedness of the stationary wave problem:} There exists an open set $Z_s\subset(\R^+)^3\times\bf{W}_{s,r}$ satisfying~\eqref{open set conditions for Zs} and continuous a map
        \begin{equation}\label{suzan}
            Z_s\ni(\mathcal{T},\mathcal{F},\mathfrak{g},\mu,\kappa)\mapsto(p,u,\eta)\in\bigcup_{\bf{v}\in(\R^+)^3}V_s(\bf{v})
        \end{equation}
        with the property that for all $(\mathcal{T},\mathcal{F},\mathfrak{g},\mu,\kappa)\in Z_s$, the corresponding $(p,u,\eta)$ belongs to the set $V_s(\mathfrak{g},\mu,\kappa)$ and is the unique solution to~\eqref{final nonlinear equations} with $\gam=0$ in this set with data in $Z_s$.
    \end{enumerate}
\end{coro}
\begin{proof}
    The first item follows from Proposition \ref{Xsr classicality} and the condition $s > 3/r+1$.
    We continue by proving the second item. Let $(p,u,\eta)\in\mathcal{O}_{s,r}(\varrho)$ be a solution generated by the data $(\gam,\mathfrak{g},\mu,\kappa,\mathcal{T},\mathcal{F})\in\mathscr{W}_s$. We set $v:\Omega[\eta]\to\R^3$ and $q:\Omega[\eta]\to\R$ via $v=(M_\eta^{-1} u)\circ\mathfrak{F}_\eta^{-1}$ and $q=p\circ\mathfrak{F}_\eta^{-1}$. Proposition~\ref{prop on properties of the flattening map} verifies that the map $\mathfrak{F}_\eta:\Omega\to\Omega[\eta]$ is a smooth diffeomorphism that is sufficiently regular up to the boundary as to preserve the notion of classical solution. It is then elementary to verify that $(q,v,\eta)$ classically solve~\eqref{the stationarity ansatz} with wave speed $\gam$, physical parameters $(\mathfrak{g},\mu,\kappa)\in(\R^+)^3$, and stress-force data $(\mathcal{T},\mathcal{F})$

    The third and fourth items items are just particular `restrictions' of the map in~\eqref{solution operator} from Theorem~\ref{coro fourth times the charm for well-posedness}, so long as we define
    \begin{equation}
        W_s(\mathfrak{g},\mu,\kappa)=\tcb{(\gam,\mathcal{T},\mathcal{F})\;:\;(\gam,\mathfrak{g},\mu,\kappa,\mathcal{T},\mathcal{F})\in W_s}
    \end{equation}
    and
    \begin{equation}
        Z_s=\tcb{(\mathfrak{g},\mu,\kappa,\mathcal{T},\mathcal{F})\;:\;(0,\mathfrak{g},\mu,\kappa,\mathcal{T},\mathcal{F})\in W_s},
    \end{equation}
    and define the mappings~\eqref{the particular solution map} and~\eqref{suzan} via~\eqref{solution operator} and the `slice' identifications $W_s(\mathfrak{g},\mu,\kappa),Z_s\subset W_s$.
\end{proof}

% _+__+_ -_+__+_ -_+__+_ -_+__+_ -_+__+_ -_+__+_ -_+__+_ -_+__+_ -_+__+_ -_+__+_ -_+__+_ -_+__+_ -_+__+_ -

% _+__+_ -_+__+_ -_+__+_ -_+__+_ -_+__+_ -_+__+_ -_+__+_ -_+__+_ -_+__+_ -_+__+_ -_+__+_ -_+__+_ -_+__+_ -
\section*{Statements and Declarations}
% _+__+_ -_+__+_ -_+__+_ -_+__+_ -_+__+_ -_+__+_ -_+__+_ -_+__+_ -_+__+_ -_+__+_ -_+__+_ -_+__+_ -_+__+_ -

\small{\textbf{Conflict of interest:} There does not exist conflict of interest in this document.}

% _+__+_ -_+__+_ -_+__+_ -_+__+_ -_+__+_ -_+__+_ -_+__+_ -_+__+_ -_+__+_ -_+__+_ -_+__+_ -_+__+_ -_+__+_ -
\bibliographystyle{abbrv}
\bibliography{bib.bib}

\begin{thebibliography}{10}

\bibitem{AS_2003}
T.~Abe and Y.~Shibata.
\newblock {On a resolvent estimate of the {S}tokes equation on an infinite
  layer}.
\newblock {\em J. Math. Soc. Japan}, 55(2):469--497, 2003.

\bibitem{AS_2003_2}
T.~Abe and Y.~Shibata.
\newblock {On a resolvent estimate of the {S}tokes equation on an infinite
  layer. {II}. {$\lambda=0$} case}.
\newblock {\em J. Math. Fluid Mech.}, 5(3):245--274, 2003.

\bibitem{AY_2010}
T.~Abe and M.~Yamazaki.
\newblock {On a stationary problem of the {S}tokes equation in an infinite
  layer in {S}obolev and {B}esov spaces}.
\newblock {\em J. Math. Fluid Mech.}, 12(1):61--100, 2010.

\bibitem{Abels_2005}
H.~Abels.
\newblock {Reduced and generalized {S}tokes resolvent equations in
  asymptotically flat layers. {I}. {U}nique solvability}.
\newblock {\em J. Math. Fluid Mech.}, 7(2):201--222, 2005.

\bibitem{Abels_2005_2}
H.~Abels.
\newblock {Reduced and generalized {S}tokes resolvent equations in
  asymptotically flat layers. {II}. {$H_\infty$}-calculus}.
\newblock {\em J. Math. Fluid Mech.}, 7(2):223--260, 2005.

\bibitem{Abels_2005_3}
H.~Abels.
\newblock {The initial-value problem for the {N}avier-{S}tokes equations with a
  free surface in {$L^q$}-{S}obolev spaces}.
\newblock {\em Adv. Differential Equations}, 10(1):45--64, 2005.

\bibitem{Abels_2006}
H.~Abels.
\newblock {Generalized {S}tokes resolvent equations in an infinite layer with
  mixed boundary conditions}.
\newblock {\em Math. Nachr.}, 279(4):351--367, 2006.

\bibitem{AW_2005}
H.~Abels and M.~Wiegner.
\newblock {Resolvent estimates for the {S}tokes operator on an infinite layer}.
\newblock {\em Differential Integral Equations}, 18(10):1081--1110, 2005.

\bibitem{MR1215410}
F.~Abergel.
\newblock A geometric approach to the study of stationary free surface flows
  for viscous liquids.
\newblock {\em Proc. Roy. Soc. Edinburgh Sect. A}, 123(2):209--229, 1993.

\bibitem{abergel1995interfaces}
F.~Abergel and E.~Rouy.
\newblock {\em Interfaces stationnaires pour les {\'e}quations de
  Navier-Stokes}.
\newblock PhD thesis, INRIA, 1995.

\bibitem{AbMaRa_1988}
R.~Abraham, J.~E. Marsden, and T.~Ratiu.
\newblock {\em Manifolds, tensor analysis, and applications}, volume~75 of {\em
  Applied Mathematical Sciences}.
\newblock Springer-Verlag, New York, second edition, 1988.

\bibitem{allain_1987}
G.~Allain.
\newblock Small-time existence for the {N}avier-{S}tokes equations with a free
  surface.
\newblock {\em Appl. Math. Optim.}, 16(1):37--50, 1987.

\bibitem{amann_1997}
H.~Amann.
\newblock Operator-valued {F}ourier multipliers, vector-valued {B}esov spaces,
  and applications.
\newblock {\em Math. Nachr.}, 186:5--56, 1997.

\bibitem{Bae_2011}
H.~Bae.
\newblock {Solvability of the free boundary value problem of the
  {N}avier-{S}tokes equations}.
\newblock {\em Discrete Contin. Dyn. Syst.}, 29(3):769--801, 2011.

\bibitem{MR2768550}
H.~Bahouri, J.-Y. Chemin, and R.~Danchin.
\newblock {\em Fourier analysis and nonlinear partial differential equations},
  volume 343 of {\em Grundlehren der mathematischen Wissenschaften [Fundamental
  Principles of Mathematical Sciences]}.
\newblock Springer, Heidelberg, 2011.

\bibitem{Beale_1981}
J.~T. Beale.
\newblock {The initial value problem for the {N}avier-{S}tokes equations with a
  free surface}.
\newblock {\em Comm. Pure Appl. Math.}, 34(3):359--392, 1981.

\bibitem{Beale_1983}
J.~T. Beale.
\newblock {Large-time regularity of viscous surface waves}.
\newblock {\em Arch. Rational Mech. Anal.}, 84(4):307--352, 1983/84.

\bibitem{BN_1985}
J.~T. Beale and T.~Nishida.
\newblock {Large-time behavior of viscous surface waves}.
\newblock In {\em {Recent topics in nonlinear {PDE}, {II} ({S}endai, 1984)}},
  volume 128 of {\em {North-Holland Math. Stud.}}, pages 1--14. North-Holland,
  Amsterdam, 1985.

\bibitem{MR727877}
J.~Bemelmans.
\newblock Gleichgewichtsfiguren z\"{a}her {F}l\"{u}ssigkeiten mit
  {O}berfl\"{a}chenspannung.
\newblock {\em Analysis}, 1(4):241--282, 1981.

\bibitem{MR637857}
J.~Bemelmans.
\newblock Liquid drops in a viscous fluid under the influence of gravity and
  surface tension.
\newblock {\em Manuscripta Math.}, 36(1):105--123, 1981/82.

\bibitem{MR929474}
J.~Bemelmans.
\newblock On a free boundary problem for the stationary {N}avier-{S}tokes
  equations.
\newblock {\em Ann. Inst. H. Poincar\'{e} Anal. Non Lin\'{e}aire},
  4(6):517--547, 1987.

\bibitem{MR134999}
T.~B. Benjamin.
\newblock Wave formation in laminar flow down an inclined plane.
\newblock {\em J. Fluid Mech.}, 2:554--574, 1957.

\bibitem{MR0482275}
J.~Bergh and J.~L\"{o}fstr\"{o}m.
\newblock {\em Interpolation spaces. {A}n introduction}.
\newblock Grundlehren der Mathematischen Wissenschaften, No. 223.
  Springer-Verlag, Berlin-New York, 1976.

\bibitem{blasco_vanNeerven_2010}
O.~Blasco and J.~van Neerven.
\newblock Spaces of operator-valued functions measurable with respect to the
  strong operator topology.
\newblock In {\em Vector measures, integration and related topics}, volume 201
  of {\em Oper. Theory Adv. Appl.}, pages 65--78. Birkh\"{a}user Verlag, Basel,
  2010.

\bibitem{bourgain_1983}
J.~Bourgain.
\newblock Some remarks on {B}anach spaces in which martingale difference
  sequences are unconditional.
\newblock {\em Ark. Mat.}, 21(2):163--168, 1983.

\bibitem{bu_kim_2005}
S.~Bu and J.-M. Kim.
\newblock Operator-valued {F}ourier multiplier theorems on {T}riebel spaces.
\newblock {\em Acta Math. Sci. Ser. B (Engl. Ed.)}, 25(4):599--609, 2005.

\bibitem{burkholder_1981}
D.~L. Burkholder.
\newblock A geometrical characterization of {B}anach spaces in which martingale
  difference sequences are unconditional.
\newblock {\em Ann. Probab.}, 9(6):997--1011, 1981.

\bibitem{CDAD_2011}
Y.~Cho, J.~D. Diorio, T.~R. Akylas, and J.~H. Duncan.
\newblock {Resonantly forced gravity--capillary lumps on deep water. Part 2.
  Theoretical model}.
\newblock {\em J. Math. Fluid Mech.}, 672:288--306, 2011.

\bibitem{MR972259}
P.~Constantin and C.~Foias.
\newblock {\em Navier-{S}tokes equations}.
\newblock Chicago Lectures in Mathematics. University of Chicago Press,
  Chicago, IL, 1988.

\bibitem{MR0288640}
M.~G. Crandall and P.~H. Rabinowitz.
\newblock Bifurcation from simple eigenvalues.
\newblock {\em J. Functional Analysis}, 8:321--340, 1971.

\bibitem{MR0206190}
N.~Dinculeanu.
\newblock {\em Vector measures.}
\newblock Pergamon Press, Oxford-New York-Toronto, Ont.; VEB Deutscher Verlag
  der Wissenschaften, Berlin,, 1967.

\bibitem{DCDA_2011}
J.~D. Diorio, Y.~Cho, J.~H. Duncan, and T.~R. Akylas.
\newblock {Resonantly forced gravity--capillary lumps on deep water. Part 1.
  Experiments}.
\newblock {\em J. Math. Fluid Mech.}, 672:268--287, 2011.

\bibitem{EWZ_2023}
M.~Ehrnstr\"{o}m, S.~Walsh, and C.~Zeng.
\newblock Smooth stationary water waves with exponentially localized vorticity.
\newblock {\em J. Eur. Math. Soc. (JEMS)}, 25(3):1045--1090, 2023.

\bibitem{MR284802}
C.~Fefferman and E.~M. Stein.
\newblock Some maximal inequalities.
\newblock {\em Amer. J. Math.}, 93:107--115, 1971.

\bibitem{MR1245932}
R.~S. Gellrich.
\newblock Free boundary value problems for the stationary {N}avier-{S}tokes
  equations in domains with noncompact boundaries.
\newblock {\em Z. Anal. Anwendungen}, 12(3):425--455, 1993.

\bibitem{girardi_weis_2003_besov}
M.~Girardi and L.~Weis.
\newblock Operator-valued {F}ourier multiplier theorems on {B}esov spaces.
\newblock {\em Math. Nachr.}, 251:34--51, 2003.

\bibitem{girardi_weis_2003_lp}
M.~Girardi and L.~Weis.
\newblock Operator-valued {F}ourier multiplier theorems on {$L_p(X)$} and
  geometry of {B}anach spaces.
\newblock {\em J. Funct. Anal.}, 204(2):320--354, 2003.

\bibitem{girardi_weis_2003_survey}
M.~Girardi and L.~Weis.
\newblock Vector-valued extentions of some classical theorems in harmonic
  analysis.
\newblock In {\em Analysis and applications---{ISAAC} 2001 ({B}erlin)},
  volume~10 of {\em Int. Soc. Anal. Appl. Comput.}, pages 171--185. Kluwer
  Acad. Publ., Dordrecht, 2003.

\bibitem{MR3243734}
L.~Grafakos.
\newblock {\em Classical {F}ourier analysis}, volume 249 of {\em Graduate Texts
  in Mathematics}.
\newblock Springer, New York, third edition, 2014.

\bibitem{MR3243741}
L.~Grafakos.
\newblock {\em Modern {F}ourier analysis}, volume 250 of {\em Graduate Texts in
  Mathematics}.
\newblock Springer, New York, third edition, 2014.

\bibitem{Groves_2004}
M.~D. Groves.
\newblock {Steady water waves}.
\newblock {\em J. Nonlinear Math. Phys.}, 11(4):435--460, 2004.

\bibitem{GT_2013_inf}
Y.~Guo and I.~Tice.
\newblock {Decay of viscous surface waves without surface tension in
  horizontally infinite domains}.
\newblock {\em Anal. PDE}, 6(6):1429--1533, 2013.

\bibitem{GT_2013_lwp}
Y.~Guo and I.~Tice.
\newblock {Local well-posedness of the viscous surface wave problem without
  surface tension}.
\newblock {\em Anal. PDE}, 6(2):287--369, 2013.

\bibitem{MR1315521}
P.~Haj\l~asz and A.~Ka\l~amajska.
\newblock Polynomial asymptotics and approximation of {S}obolev functions.
\newblock {\em Studia Math.}, 113(1):55--64, 1995.

\bibitem{haller_heck_noll_2002}
R.~Haller, H.~Heck, and A.~Noll.
\newblock Mikhlin's theorem for operator-valued {F}ourier multipliers in {$n$}
  variables.
\newblock {\em Math. Nachr.}, 244:110--130, 2002.

\bibitem{MR4406719}
S.~V. Haziot, V.~M. Hur, W.~A. Strauss, J.~F. Toland, E.~Wahl\'{e}n, S.~Walsh,
  and M.~H. Wheeler.
\newblock Traveling water waves---the ebb and flow of two centuries.
\newblock {\em Quart. Appl. Math.}, 80(2):317--401, 2022.

\bibitem{hieber_1999}
M.~Hieber.
\newblock Operator valued {F}ourier multipliers.
\newblock In {\em Topics in nonlinear analysis}, volume~35 of {\em Progr.
  Nonlinear Differential Equations Appl.}, pages 363--380. Birkh\"{a}user,
  Basel, 1999.

\bibitem{MR0423094}
E.~Hille and R.~S. Phillips.
\newblock {\em Functional analysis and semi-groups}.
\newblock American Mathematical Society Colloquium Publications, Vol. XXXI.
  American Mathematical Society, Providence, R.I., 1974.
\newblock Third printing of the revised edition of 1957.

\bibitem{MR3617205}
T.~Hyt\"{o}nen, J.~van Neerven, M.~Veraar, and L.~Weis.
\newblock {\em Analysis in {B}anach spaces. {V}ol. {I}. {M}artingales and
  {L}ittlewood-{P}aley theory}, volume~63 of {\em Ergebnisse der Mathematik und
  ihrer Grenzgebiete. 3. Folge. A Series of Modern Surveys in Mathematics
  [Results in Mathematics and Related Areas. 3rd Series. A Series of Modern
  Surveys in Mathematics]}.
\newblock Springer, Cham, 2016.

\bibitem{MR3135704}
H.~Inci, T.~Kappeler, and P.~Topalov.
\newblock On the regularity of the composition of diffeomorphisms.
\newblock {\em Mem. Amer. Math. Soc.}, 226(1062):vi+60, 2013.

\bibitem{MR2170526}
B.~Ja~Jin.
\newblock Free boundary problem of steady incompressible flow with contact
  angle {$\frac\pi2$}.
\newblock {\em J. Differential Equations}, 217(1):1--25, 2005.

\bibitem{MR591221}
M.~Jean.
\newblock Free surface of the steady flow of a {N}ewtonian fluid in a finite
  channel.
\newblock {\em Arch. Rational Mech. Anal.}, 74(3):197--217, 1980.

\bibitem{koganemaru2022traveling}
J.~Koganemaru and I.~Tice.
\newblock Traveling wave solutions to the inclined or periodic free boundary
  incompressible {N}avier-{S}tokes equations.
\newblock {\em J. Funct. Anal.}, to appear.

\bibitem{lancien_lancien_lemerdy_1998}
F.~Lancien, G.~Lancien, and C.~Le~Merdy.
\newblock A joint functional calculus for sectorial operators with commuting
  resolvents.
\newblock {\em Proc. London Math. Soc. (3)}, 77(2):387--414, 1998.

\bibitem{MR3726909}
G.~Leoni.
\newblock {\em A first course in {S}obolev spaces}, volume 181 of {\em Graduate
  Studies in Mathematics}.
\newblock American Mathematical Society, Providence, RI, second edition, 2017.

\bibitem{leoni2019traveling}
G.~Leoni and I.~Tice.
\newblock Traveling wave solutions to the free boundary incompressible
  {N}avier-{S}tokes equations.
\newblock {\em Comm. Pure Appl. Math.}, 2022.

\bibitem{MD_2017}
N.~Masnadi and J.~H. Duncan.
\newblock The generation of gravity–capillary solitary waves by a pressure
  source moving at a trans-critical speed.
\newblock {\em J. Fluid Mech.}, 810:448–474, 2017.

\bibitem{mcconnell_1984}
T.~R. McConnell.
\newblock On {F}ourier multiplier transformations of {B}anach-valued functions.
\newblock {\em Trans. Amer. Math. Soc.}, 285(2):739--757, 1984.

\bibitem{MR1215650}
S.~Nazarov and K.~Pileckas.
\newblock On noncompact free boundary problems for the plane stationary
  {N}avier-{S}tokes equations.
\newblock {\em J. Reine Angew. Math.}, 438:103--141, 1993.

\bibitem{nguyen_tice_2022}
H.~Nguyen and I.~Tice.
\newblock Traveling wave solutions to the one-phase muskat problem: existence
  and stability.
\newblock {\em Preprint, arXiv:2211.06286}, 2022.

\bibitem{MR1850453}
L.~Nirenberg.
\newblock {\em Topics in nonlinear functional analysis}, volume~6 of {\em
  Courant Lecture Notes in Mathematics}.
\newblock New York University, Courant Institute of Mathematical Sciences, New
  York; American Mathematical Society, Providence, RI, 2001.
\newblock Chapter 6 by E. Zehnder, Notes by R. A. Artino, Revised reprint of
  the 1974 original.

\bibitem{PC_2016}
B.~Park and Y.~Cho.
\newblock {Experimental observation of gravity--capillary solitary waves
  generated by a moving air suction}.
\newblock {\em J. Math. Fluid Mech.}, 808:168--188, 2016.

\bibitem{PC_2018}
B.~Park and Y.~Cho.
\newblock {Two-dimensional gravity--capillary solitary waves on deep water:
  generation and transverse instability}.
\newblock {\em J. Math. Fluid Mech.}, 834:92--124, 2018.

\bibitem{MR1936348}
K.~Pileckas and J.~Socolowsky.
\newblock Analysis of two linearized problems modeling viscous two-layer flows.
\newblock {\em Math. Nachr.}, 245:129--166, 2002.

\bibitem{MR2128433}
K.~Pileckas and J.~Socolowsky.
\newblock Viscous two-fluid flows in perturbed unbounded domains.
\newblock {\em Math. Nachr.}, 278(5):589--623, 2005.

\bibitem{MR3208795}
K.~Pileckas and V.~A. Solonnikov.
\newblock Viscous incompressible free-surface flow down an inclined perturbed
  plane.
\newblock {\em Ann. Univ. Ferrara Sez. VII Sci. Mat.}, 60(1):225--244, 2014.

\bibitem{MR1083808}
K.~I. Pileckas.
\newblock On plane motion of a viscous incompressible capillary liquid with a
  noncompact free boundary.
\newblock volume~41, pages 329--342 (1990). 1989.
\newblock XVIIIth Symposium on Advanced Problems and Methods in Fluid
  Mechanics.

\bibitem{MR741910}
K.~Piletskas.
\newblock Gliding of a flat plate of infinite span over the surface of a heavy
  viscous incompressible fluid of finite depth.
\newblock {\em Differentsialnye Uravneniya i Primenen.}, (34):60--74, 1983.

\bibitem{MR789999}
K.~Piletskas.
\newblock A remark on the paper: ``{G}liding of a flat plate of infinite span
  over the surface of a heavy viscous incompressible fluid of finite depth''.
\newblock {\em Differentsialnye Uravneniya i Primenen.}, (36):55--60, 139,
  1984.

\bibitem{MR643983}
K.~I. Piletskas.
\newblock Solvability of a problem on the planar motion of a viscous
  incompressible fluid with a free noncompact boundary.
\newblock {\em Zap. Nauchn. Sem. Leningrad. Otdel. Mat. Inst. Steklov. (LOMI)},
  110:174--179, 245, 1981.
\newblock Boundary value problems of mathematical physics and related questions
  in the theory of functions, 13.

\bibitem{MR966116}
K.~I. Piletskas.
\newblock On the problem of the flow of a heavy viscous incompressible fluid
  with a free noncompact boundary.
\newblock {\em Litovsk. Mat. Sb.}, 28(2):315--333, 1988.

\bibitem{pruss_simonett_2016}
J.~Pr\"{u}ss and G.~Simonett.
\newblock {\em Moving interfaces and quasilinear parabolic evolution
  equations}, volume 105 of {\em Monographs in Mathematics}.
\newblock Birkh\"{a}user/Springer, [Cham], 2016.

\bibitem{MR0308618}
V.~V. Puhna\v{c}ev.
\newblock The plane stationary problem with a free boundary for the
  {N}avier-{S}tokes equations.
\newblock {\em J. Appl. Mech. Tech. Phys.}, 13(3):340--349, 1972.

\bibitem{MR0493419}
M.~Reed and B.~Simon.
\newblock {\em Methods of modern mathematical physics. {I}. {F}unctional
  analysis}.
\newblock Academic Press, New York-London, 1972.

\bibitem{robinson_2001}
J.~C. Robinson.
\newblock {\em Infinite-dimensional dynamical systems}.
\newblock Cambridge Texts in Applied Mathematics. Cambridge University Press,
  Cambridge, 2001.
\newblock An introduction to dissipative parabolic PDEs and the theory of
  global attractors.

\bibitem{schwartz_1961}
J.~Schwartz.
\newblock A remark on inequalities of {C}alderon-{Z}ygmund type for
  vector-valued functions.
\newblock {\em Comm. Pure Appl. Math.}, 14:785--799, 1961.

\bibitem{SS_2007}
Y.~Shibata and S.~Shimizu.
\newblock {Free boundary problems for a viscous incompressible fluid}.
\newblock In {\em {Kyoto {C}onference on the {N}avier-{S}tokes {E}quations and
  their {A}pplications}}, {RIMS K{\^o}ky{\^u}roku Bessatsu, B1}, pages
  356--358. Res. Inst. Math. Sci. (RIMS), Kyoto, 2007.

\bibitem{SS_2011}
Y.~Shibata and S.~Shimizu.
\newblock {Report on a local in time solvability of free surface problems for
  the {N}avier-{S}tokes equations with surface tension}.
\newblock {\em Appl. Anal.}, 90(1):201--214, 2011.

\bibitem{MR0564661}
D.~Socolescu.
\newblock Existenz- und {E}indeutigkeitsbeweis f\"{u}r ein freies
  {R}andwertproblem f\"{u}r die station\"{a}ren {N}avier-{S}tokesschen
  {B}ewegungsgleichungen.
\newblock {\em Arch. Rational Mech. Anal.}, 73(3):191--242, 1980.

\bibitem{MR1246506}
J.~Socolowsky.
\newblock The solvability of a free boundary problem for the stationary
  {N}avier-{S}tokes equations with a dynamic contact line.
\newblock {\em Nonlinear Anal.}, 21(10):763--784, 1993.

\bibitem{MR3032463}
J.~Socolowsky.
\newblock On a two-fluid inclined film flow with evaporation.
\newblock {\em Math. Model. Anal.}, 18(1):22--31, 2013.

\bibitem{MR525947}
V.~A. Solonnikov.
\newblock Solvability of the problem of the plane motion of a heavy viscous
  incompressible capillary fluid that partially fills a certain vessel.
\newblock {\em Izv. Akad. Nauk SSSR Ser. Mat.}, 43(1):203--236, 239, 1979.

\bibitem{MR737234}
V.~A. Solonnikov.
\newblock Solvability of a three-dimensional boundary value problem with a free
  surface for the stationary {N}avier-{S}tokes system.
\newblock In {\em Partial differential equations ({W}arsaw, 1978)}, volume~10
  of {\em Banach Center Publ.}, pages 361--403. PWN, Warsaw, 1983.

\bibitem{MR1638135}
V.~A. Solonnikov.
\newblock Solvability of two stationary free boundary problems for the
  {N}avier-{S}tokes equations.
\newblock {\em Boll. Unione Mat. Ital. Sez. B Artic. Ric. Mat. (8)},
  1(2):283--342, 1998.

\bibitem{MR1692348}
V.~A. Solonnikov.
\newblock Stationary free boundary problems for the {N}avier-{S}tokes
  equations.
\newblock In {\em Advanced topics in theoretical fluid mechanics ({P}aseky nad
  {J}izerou, 1997)}, volume 392 of {\em Pitman Res. Notes Math. Ser.}, pages
  147--212. Longman, Harlow, 1998.

\bibitem{MR3916796}
V.~A. Solonnikov and I.~V. Denisova.
\newblock Classical well-posedness of free boundary problems in viscous
  incompressible fluid mechanics.
\newblock In {\em Handbook of mathematical analysis in mechanics of viscous
  fluids}, pages 1135--1220. Springer, Cham, 2018.

\bibitem{MR0290095}
E.~M. Stein.
\newblock {\em Singular integrals and differentiability properties of
  functions}.
\newblock Princeton Mathematical Series, No. 30. Princeton University Press,
  Princeton, N.J., 1970.

\bibitem{MR1232192}
E.~M. Stein.
\newblock {\em Harmonic analysis: real-variable methods, orthogonality, and
  oscillatory integrals}, volume~43 of {\em Princeton Mathematical Series}.
\newblock Princeton University Press, Princeton, NJ, 1993.
\newblock With the assistance of Timothy S. Murphy, Monographs in Harmonic
  Analysis, III.

\bibitem{MR4337506}
N.~Stevenson and I.~Tice.
\newblock Traveling wave solutions to the multilayer free boundary
  incompressible {N}avier-{S}tokes equations.
\newblock {\em SIAM J. Math. Anal.}, 53(6):6370--6423, 2021.

\bibitem{stevenson2023wellposedness}
N.~Stevenson and I.~Tice.
\newblock Well-posedness of the traveling wave problem for the free boundary
  compressible {N}avier-{S}tokes equations, 2023.

\bibitem{Strauss_2010}
W.~A. Strauss.
\newblock {Steady water waves}.
\newblock {\em Bull. Amer. Math. Soc. (N.S.)}, 47(4):671--694, 2010.

\bibitem{tani_1996}
A.~Tani.
\newblock Small-time existence for the three-dimensional {N}avier-{S}tokes
  equations for an incompressible fluid with a free surface.
\newblock {\em Arch. Rational Mech. Anal.}, 133(4):299--331, 1996.

\bibitem{MR1121019}
M.~E. Taylor.
\newblock {\em Pseudodifferential operators and nonlinear {PDE}}, volume 100 of
  {\em Progress in Mathematics}.
\newblock Birkh\"{a}user Boston, Inc., Boston, MA, 1991.

\bibitem{Toland_1996}
J.~F. Toland.
\newblock {Stokes waves}.
\newblock {\em Topol. Methods Nonlinear Anal.}, 7(1):1--48, 1996.

\bibitem{Wu_2014}
L.~Wu.
\newblock {Well-posedness and decay of the viscous surface wave}.
\newblock {\em SIAM J. Math. Anal.}, 46(3):2084--2135, 2014.

\bibitem{Zadrynska_2004}
E.~Zadrzy{\'n}ska.
\newblock {Free boundary problems for nonstationary {N}avier-{S}tokes
  equations}.
\newblock {\em Dissertationes Math. (Rozprawy Mat.)}, 424:135, 2004.

\bibitem{zimmermann_1989}
F.~Zimmermann.
\newblock On vector-valued {F}ourier multiplier theorems.
\newblock {\em Studia Math.}, 93(3):201--222, 1989.

\end{thebibliography}
% _+__+_ -_+__+_ -_+__+_ -_+__+_ -_+__+_ -_+__+_ -_+__+_ -_+__+_ -_+__+_ -_+__+_ -_+__+_ -_+__+_ -_+__+_ -

% _+__+_ -_+__+_ -_+__+_ -_+__+_ -_+__+_ -_+__+_ -_+__+_ -_+__+_ -_+__+_ -_+__+_ -_+__+_ -_+__+_ -_+__+_ -% _+__+_ -_+__+_ -_+__+_ -_+__+_ -_+__+_ -_+__+_ -_+__+_ -_+__+_ -_+__+_ -_+__+_ -_+__+_ -_+__+_ -_+__+_ -% _+__+_ -_+__+_ -_+__+_ -_+__+_ -_+__+_ -_+__+_ -_+__+_ -_+__+_ -_+__+_ -_+__+_ -_+__+_ -_+__+_ -_+__+_ -This is the end of the document yay!% _+__+_ -_+__+_ -_+__+_ -_+__+_ -_+__+_ -_+__+_ -_+__+_ -_+__+_ -_+__+_ -_+__+_ -_+__+_ -_+__+_ -_+__+_ -% _+__+_ -_+__+_ -_+__+_ -_+__+_ -_+__+_ -_+__+_ -_+__+_ -_+__+_ -_+__+_ -_+__+_ -_+__+_ -_+__+_ -_+__+_ -% _+__+_ -_+__+_ -_+__+_ -_+__+_ -_+__+_ -_+__+_ -_+__+_ -_+__+_ -_+__+_ -_+__+_ -_+__+_ -_+__+_ -_+__+_ -
\end{document}